\documentclass[11pt]{article} 

\usepackage[utf8]{inputenc}

\usepackage{geometry}
\usepackage[dvips]{graphicx}

\usepackage{booktabs} 
\usepackage{array} 
\usepackage{paralist} 
\usepackage{verbatim} 
\usepackage{subfig} 

\usepackage{fancyhdr} 
\usepackage[nottoc,notlof,notlot]{tocbibind}
\usepackage[titles,subfigure]{tocloft}

\usepackage{amsmath,amsfonts,amsthm,mathrsfs,amssymb}
\usepackage{latexsym}
\usepackage{enumerate}
\usepackage[usenames]{color}

\newtheorem{thm}{Theorem}[section]

\newtheorem{lem}{Lemma}[section]
\newtheorem{prop}{Proposition}[section]
\theoremstyle{definition}
\newtheorem{defin}{Definition}[section]
\theoremstyle{remark}
\newtheorem{rem}{Remark}[section]
\numberwithin{equation}{section}

{
\newtheorem{oss}[thm]{Remark}

\newtheorem{assum}{Assumption}}

\newcommand{\cvd}{\hfill$\square$}

\title{ Regularized Transformation-Optics Cloaking for the Helmholtz Equation: From Partial Cloak to Full Cloak}

\author{Jingzhi Li\thanks{Faculty of Science, South University of Science and
Technology of China, Shenzhen 518055, P.~R.~China.  Email: {\tt li.jz@sustc.edu.cn}}, Hongyu Liu\thanks{Department of Mathematics and Statistics, University of North Carolina, Charlotte, NC 28223, USA.   Email:  {\tt hongyu.liuip@gmail.com}}, Luca Rondi\thanks{Dipartimento di Matematica e Geoscienze, Universit\`a degli Studi di Trieste, Trieste, Italy.  Email: {\tt rondi@units.it}}, Gunther Uhlmann\thanks{Department of Mathematics, University of Washington, Seattle, WA 98195, USA and Fondation des Sciences Math\'ematiques de Paris.  Email:  {\tt gunther@math.washington.edu}}    }

\begin{document}

\date{}

\maketitle

\begin{abstract}
We develop a very general theory on the regularized approximate invisibility cloaking for the wave scattering governed by the Helmholtz equation in any space dimensions $N\geq 2$ via the approach of transformation optics. There are four major ingredients in our proposed theory: 1).~The non-singular cloaking medium is obtained by the push-forwarding construction through a transformation which blows up  a subset $K_\varepsilon$ in the virtual space, where $\varepsilon\ll 1$ is an asymptotic regularization parameter. $K_\varepsilon$ will degenerate to $K_0$ as $\varepsilon\rightarrow +0$, and in our theory $K_0$ could be any convex compact set in $\mathbb{R}^N$, or any set whose boundary consists of Lipschitz hypersurfaces, or a finite combination of those sets. 2).~~A general lossy layer with the material parameters satisfying certain compatibility integral conditions is employed right between the cloaked and cloaking regions. 3).~The contents being cloaked could also be extremely general, possibly including, at the same time, generic mediums and, sound-soft, sound-hard and impedance-type obstacles, as
well as some sources or sinks. 4).~In order to achieve a cloaking device of compact size, particularly for the case when $K_\varepsilon$ is not ``uniformly small", an assembly-by-components, the (ABC) geometry is developed for both the virtual and physical spaces and the blow-up construction is based on concatenating different components.

Within the proposed framework, we show that the scattered wave field $u_\varepsilon$ corresponding to a cloaking problem will converge to $u_0$ as $\varepsilon\rightarrow +0$, with $u_0$ being the scattered wave field corresponding to a sound-hard $K_0$. The convergence result is used to theoretically justify the approximate full and partial invisibility cloaks, depending on the geometry of $K_0$. On the other hand, the convergence results are conducted in a much more general setting than what is needed for the invisibility cloaking, so they are of significant mathematical interest for their own sake. As for applications, we construct three types of full and partial cloaks. Some numerical experiments are also conducted to illustrate our theoretical results.

\medskip


\end{abstract}

\section{Introduction}

This paper is concerned with the invisibility cloaking for the wave scattering governed by the Helmholtz equation via the approach of transformation optics \cite{GLU,GLU2,Leo,PenSchSmi}, which is a rapidly growing scientific field with many potential applications. We refer to  \cite{CC,GKLU4,GKLU5,Nor,U2,YYQ} and the references therein
for discussions of the recent progress on both the theory and experiments.

Let $\Omega$ and $D$ be two bounded Lipschitz domains in $\mathbb{R}^N$, $N\geq 2$, such that $D\Subset \Omega$. 
Let $\sigma=\sigma(x)=(\sigma^{ij}(x))\in\mathbb{R}^{N\times N}$, $x\in\mathbb{R}^N$, be a symmetric-matrix valued measurable function such that, for some $\lambda$, $0<\lambda\leq 1$, we have
\begin{equation}\label{eq:reg1}
\lambda \|\xi\|^2\leq\sum_{i,j=1}^N\sigma^{ij}(x)\xi_i\xi_j\leq \lambda^{-1}\|\xi\|^2\quad \mbox{for any $\xi\in\mathbb{R}^N$ and for a.e. $x\in\mathbb{R}^N$}.
\end{equation}
Let $q=q_1+i q_2=q(x)$, $x\in\mathbb{R}^N$, be a complex-valued bounded measurable function with real and imaginary parts $q_1$ and $q_2$ respectively, such that, for some $\lambda$, $0<\lambda\leq 1$, we have
\begin{equation}\label{eq:reg2}
q_1(x)\geq \lambda,\quad q_2(x)\geq 0\quad \mbox{for a.e. $x\in\mathbb{R}^N$}.
\end{equation}
Furthermore, we assume that $q(x)=q_0:=1$ and $\sigma^{ij}(x)=\sigma_0^{ij}:=\delta_{ij}$ for $x\in\mathbb{R}^N\backslash\overline{\Omega}$, where $\delta_{ij}$ denotes the Kronecker delta function. In the following, \eqref{eq:reg1} and \eqref{eq:reg2} will be referred to as the {\it regular conditions} on $\sigma$ and $q$, and $\lambda$ is called the {\it regular constant}.

Next, we introduce the time-harmonic wave scattering governed by the Helmholtz equation
whose weak solution is $u=u(x,d,k)$, $x\in\mathbb{R}^N$, where $d\in\mathbb{S}^{N-1}$, $k\in\mathbb{R}_+$,
\begin{equation}\label{eq:Helmholtz II}
\begin{cases}
\ \displaystyle{\mathrm{div}(\sigma\nabla u)+{k^2}q u}=0& \hspace*{-1.7cm} \text{in }\mathbb{R}^N,\\
\ \text{$u(x,d,k)-e^{ikx\cdot d}$ satisfies the radiation condition}.
\end{cases}
\end{equation}
The last statement in \eqref{eq:Helmholtz II} means that if one lets $u^s(x,d,k)=u(x,d,k)-e^{ikx\cdot d}$, then
\begin{equation}\label{eq:sommerfeldn}
\lim_{r\rightarrow\infty}r^{\frac{N-1}{2}}\left(\frac{\partial u^s(x)}{\partial r}-i k u^s(x)\right )=0, \quad r=\|x\|.
\end{equation}

In the physical situation, \eqref{eq:Helmholtz II} can be used to describe the time-harmonic acoustic scattering due to an inhomogeneous acoustical medium $(\Omega; \sigma, q)$ located in an otherwise uniformly homogeneous space $(\mathbb{R}^N\backslash\overline{\Omega}; \sigma_0, q_0)$. $\sigma$ and $q$, respectively, denote the density tensor and acoustic modulus of the acoustical medium, and $u(x)$ denotes the wave pressure with $U(x,t):=u(x) e^{-ik t}$ representing the wave field satisfying the scalar wave equation
\[
q(x)U_{tt}(x,t)-\sum_{i,j=1}^N \frac{\partial}{\partial x_i}\left(\sigma^{ij}(x) \frac{\partial}{\partial x_j}U(x,t) \right )=0\qquad\mbox{in\ \ $\mathbb{R}^N\times \mathbb{R}$}.
\]
The function $u^i(x):=e^{ik x\cdot d}$ is an incident plane wave with $k$ denoting the wave number and $d\in\mathbb{S}^{N-1}$ denoting the impinging direction. $u(x)$ is called the total wave field and $u^s(x)$ is called scattered wave field, which is the perturbation of the incident plane wave caused by the presence of the inhomogeneity $(\Omega; \sigma, q)$ in the whole space. Indeed, it is easily seen that if there is no presence of the inhomogeneity, $u^s$ will be vanishing. For the particular case with $N=2$, \eqref{eq:Helmholtz II} can also be used to describe the transverse-electric (TE) polarized electromagnetic (EM) wave propagation with the presence of an infinitely long cylindrical EM inhomogeneity $(\Omega; \sigma, q)$ (see, e.g., \cite{CakCol}). In this case, $\varepsilon:=\Re q$, $\gamma:=k \Im q$ and $\mu:=\sigma^{-1}$ denote, respectively, the electric permittivity, conductivity and magnetic permeability, where $\Re$ and $\Im$ denote taking the
respective real and imaginary parts. We refer to \cite{ColKre,Ned} for related physical background. In the rest of the paper, in order to ease the exposition, we stick to the terminologies with the acoustic scattering.

We recall that by a weak solution to \eqref{eq:Helmholtz II}  we mean that $u\in H^1_{loc}(\mathbb{R}^N)$ and that it satisfies
$$\int_{\mathbb{R}^N}\sigma\nabla u\cdot\varphi-k^2qu\varphi=0\quad\text{for any }\varphi\in C^{\infty}_0(\mathbb{R}^N).$$
The limit in \eqref{eq:sommerfeldn} has to hold uniformly for every direction $\hat{x}=x/\|x\|\in \mathbb{S}^{N-1}$ and is also known as the {\it Sommerfeld radiation condition} which characterizes the radiating nature of the scattered wave field $u^s$ (cf. \cite{ColKre,Ned}). There exists a unique weak solution $u(x,d,k)=u^-\chi_{\Omega}+u^+\chi_{\mathbb{R}^N\backslash\overline{\Omega}}\in H_{loc}^1(\mathbb{R}^N)$ to \eqref{eq:Helmholtz II}, and we refer to Appendix in \cite{LSSZ} for a convenient proof. We remark that, if the coefficients are regular enough, \eqref{eq:Helmholtz II} corresponds to the following classical transmission problem
\begin{equation}\label{eq:Helmholtz IIclassic}
\begin{cases}
\ \displaystyle{\sum_{i,j=1}^{N}\frac{\partial}{\partial x_i}\left(\sigma^{ij}\frac{\partial}{\partial x_j} u^-(x,d,k)\right)+{k^2}q u^-(x,d,k)}=0\qquad\qquad & x\in\Omega,\\
\ \displaystyle{\Delta u^+(x,d, k)+k^2 u^+(x,d,k)=0}\quad & x\in \mathbb{R}^N\backslash\overline{\Omega},\\
\ \displaystyle{u^-(x)=u^+(x),\ \ \sum_{i,j=1}^N(\nu_i\sigma^{ij}\frac{\partial u^-}{\partial x_j})(x)=(\nu\cdot\nabla u^+)(x)}\quad & x\in\partial\Omega,\\
\ \displaystyle{u^+(x,d,k)=e^{i k x\cdot d}+u^s(x,d,k)}\quad & x\in\mathbb{R}^N\backslash\overline{\Omega},\\
\ \displaystyle{\lim_{r\rightarrow\infty}r^{\frac{N-1}{2}}\left(\frac{\partial u^s(x)}{\partial r}-i k u^s(x)\right )=0}\quad & r=\|x\|,
\end{cases}
\end{equation}
where $\nu=(\nu_i)_{i=1}^N$ is the outward unit normal vector to $\partial\Omega$.

Furthermore, $u(x)$ admits the following asymptotic
development as $\|x\|\rightarrow+\infty$
\begin{equation}\label{eq:asymptotic}
u(x,d,k)=e^{ik x\cdot d}+\frac{e^{i k \|x\|}}{\|x\|^{\frac{N-1}{2}}}u_\infty\left(\frac{x}{\|x\|},d,k\right)+\mathcal{O}\left(\frac{1}{\|x\|^{\frac{N+1}{2}}}\right).
\end{equation}
In \eqref{eq:asymptotic}, $u_\infty(\hat{x},d,k)$ with $\hat{x}:=x/\|x\|\in\mathbb{S}^{N-1}$ is known as the {\it far-field pattern} or the {\it scattering
amplitude}, which depends on the impinging direction $d$ and wave number $k$ of the incident wave $u^i(x):=e^{ik x\cdot d}$, observation direction $\hat{x}$, and obviously, also the underlying
scattering object $(\Omega;\sigma,q)$. In the following, we shall also write $u_\infty(\hat{x}, d; (\Omega; \sigma, q))$ to indicate such dependences, noting that we consider $k$ to be fixed and we drop the dependence on $k$.
An important inverse scattering problem arising from practical applications is to recover the medium $(\Omega; \sigma, q)$ by knowing of $u_\infty(\hat{x},d)$.  This inverse problem is of fundamental importance to many areas of science and technology, such as radar and sonar, geophysical exploration, non-destructive testing, and medical imaging to name just a few; see \cite{ColKre,Isa} and the references therein. In this work, we shall be mainly concerned with the invisibility cloaking for
the inverse scattering problem, which could be generally introduced as follows.

\begin{defin}\label{def:cloaking device}
Let $\Omega$ and $D$ be bounded Lipschitz domains such that $D\Subset\Omega$.
$\Omega\backslash\overline{D}$ and $D$ represent, respectively, the
cloaking region and the cloaked region. Let $\Gamma$ and $\Gamma'$
be two subsets of $\mathbb{S}^{N-1}$. $(\Omega\backslash\overline{D};
\sigma_c, q_c)$ is said to be an (ideal/perfect) {\it invisibility
cloaking device} for the region $D$ if
\begin{equation}\label{eq:definvisibility}
u_\infty\left(\hat{x},d; (\Omega;\sigma_e,q_e) \right)=0\quad \mbox{for}\ \
\hat{x}\in\Gamma,\ d\in\Gamma',
\end{equation}
where the extended object
\[
(\Omega;\sigma_e,q_e)=\begin{cases}
\ \sigma_a, q_a\quad & \mbox{in\ \ $D$},\\
\ \sigma_c,q_c\quad & \mbox{in\ \ $\Omega\backslash\overline{D}$},
\end{cases}
\]
with $(D; \sigma_a,q_a)$ denoting a target medium. If $\Gamma=\Gamma'=\mathbb{S}^{N-1}$, then it is called a {\it
full cloak}, otherwise it is called a {\it partial cloak} with
limited apertures $\Gamma$ of observation angles, and $\Gamma'$ of
impinging angles.
\end{defin}

By Definition~\ref{def:cloaking device}, we have that the cloaking layer $(\Omega\backslash \overline{D}; \sigma_c, q_c)$ makes the target medium $(D;\sigma_a, q_a)$ invisible to the exterior scattering measurements when the detecting waves come from the aperture $\Gamma'$ and the observations are made in the aperture $\Gamma$.

One efficient way of constructing the invisibility cloak that has received significant attentions in recent years is the so-called transformation optics  \cite{GLU,GLU2,Leo,PenSchSmi}. By taking advantage of the push-forward properties of the material parameters $\sigma$ and $q$, the transformation optics approach via a blow-up transformation in constructing an (ideal) invisibility cloak can be simply described as follows. Let $(\Omega;\sigma_0,q_0)$ be selected for constructing the cloaking device, and let $P\in\Omega$ be a point. $(\Omega\backslash P; \sigma_0, q_0)$ lives in the so-called {\it virtual space}. Suppose that there exists a transformation $F$ which blows up the point $P$ to an open subset $D$ within $\Omega$. The homogeneous virtual space is then pushed-forward to form the cloaking layer $(\Omega\backslash\overline{D};\sigma_c, q_c)$. The cloaking layer together with a filling-in target medium $(D;\sigma_a,q_a)$ forms the cloaking device, which lives in the so-called {\it physical space}. Due to the transformation invariance of the  Helmholtz equation, it can be heuristically argued that the scattering amplitude in the physical space is the same as the scattering amplitude in the virtual space. Since the scatterer in the virtual space is a singular point $P$, whose scattering effect is negligible, this implies that the scattering amplitude in the physical space is also vanishing. Here, we would like to emphasize that from a practical viewpoint, the target medium should be arbitrary or as general as possible, and this viewpoint shall be adopted throughout our current study. The blow-up-a-point construction yields singular cloaking materials, namely, the material parameters violate the regular conditions. The singular media present a great challenge for both theoretical analysis and practical fabrications (cf. \cite{GKLU3,LZ1}). In order to avoid the singular structure, several regularized constructions have been developed. In \cite{GKLUoe,GKLU_2,RYNQ}, a truncation of singularities has been introduced. In \cite{KOVW,KSVW,Liu}, the `blow-up-a-point' transformation in \cite{GLU2,Leo,PenSchSmi} has been regularized to become the `blow-up-a-small-region' transformation. In the current study, we shall adopt the latter one for the construction of our cloaking device. Nevertheless, as pointed out in \cite{KocLiuSunUhl}, the truncation-of-singularity construction and the blow-up-a-small-region construction are equivalent to each other. Hence, all the obtained results in this work equally hold for the truncation-of-singularity construction. Instead of ideal/perfect invisibility, one would consider approximate/near invisibility for a regularized construction; that is, one intends to make the corresponding scattering amplitude due to a regularized cloaking device as small as possible depending on an asymptotically small regularization parameter $\varepsilon\in\mathbb{R}_+$. This is the main subject of study for the present paper.

Due to its practical importance, the approximate cloaking has recently been extensively studied. In \cite{Ammari1,KSVW}, approximate cloaking schemes were developed for EIT (electric impedance tomography) which might be regarded as optics at zero frequency. In \cite{Ammari2,Ammari3,KOVW,LiLiuSun,LiuSun,Liu,N1,N2}, various near-cloaking schemes were presented for scalar waves governed by the Helmholtz equation. In all the aforementioned work, the constructions of the cloaking layer $(\Omega\backslash\overline{D}; \sigma_c^\varepsilon, q_c^\varepsilon)$ are based on blowing up a uniformly small neighborhood $P_\varepsilon$ of a singular point $P$; namely, $P_\varepsilon$ degenerates to the single point $P$ as $\varepsilon\rightarrow +0$. In order to stabilize and enhance the accuracy of approximation of the near-cloaks, various mechanisms have been developed in those literatures. Particularly, we would like to note that, in \cite{KOVW}, it is shown that the regularized approximate cloak is unstable due to the existence of cloak-busting inclusions, and the authors propose to incorporate a special lossy layer to stabilize the approximation. A different lossy layer was proposed and investigated in \cite{LiLiuSun,LiuSun}. The cloaking of impenetrable obstacles, which could be taken as lossy mediums with extreme material parameters, were considered in \cite{Ammari1,Ammari2,Ammari3} and \cite{Liu}. Also, we would like to point out that, in all those studies, the approximate full invisibility cloaks were obtained.

In the present work, we develop a very  general theory on the regularized approximate invisibility cloaking for the wave scattering governed by the Helmholtz equation in any space dimensions $N\geq 2$ via the approach of transformation optics. First, the non-singular cloaking medium is obtained by the push-forwarding construction through a transformation which blows up a subset $K_\varepsilon$ in the virtual space, with $K_\varepsilon$ degenerating to $K_0$ as $\varepsilon\rightarrow +0$. In our theory, $K_0$ could be very general. It could be any convex compact subset in $\mathbb{R}^N$, or any set whose boundary consists of Lipschitz hypersurfaces, or a finite combination of those sets. For example, in $\mathbb{R}^3$, it could be a single point, or a line segment, or a bounded planar subset. This includes all the existing studies in the literature by blowing up `point-like' regions as a very special case.  Second, in order to stabilize the approximation process, a lossy layer with the material parameters satisfying certain mild compatibility integral conditions is employed right between the cloaked and cloaking regions. The lossy layer is also very general and could be variable and even be anisotropic. Third, the proposed cloaking scheme is shown to be capable of nearly cloaking an very general content, possibly including, at the same time, generic passive mediums, and sound-soft, sound-hard, and impedance-type obstacles, and some active sources or sinks as well. Finally, in order to achieve a cloaking device of compact size, particularly for the case when $K_\varepsilon$ is not `point-like', assembled-by-components (ABC) geometry is developed for both the virtual and physical spaces and the blow-up construction is based on concatenating different components. Within the proposed framework, we show that the scattered wave field $u_\varepsilon$ corresponding to a cloaking problem will converge to $u_0$ as $\varepsilon\rightarrow +0$, with $u_0$ being the scattered wave field corresponding to a sound-hard $K_0$. The convergence result is used to theoretically justify the approximate full and partial invisibility cloaks, depending on the geometry of $K_0$. On the other hand, the convergence results are conducted in a much more general setting than what is needed for the invisibility cloaking, so they are of significant mathematical interest for their own sake. As for applications, we construct three types of full and partial cloaks. Some numerical experiments are also conducted to illustrate our theoretical results.

It is interesting to note that in addition to the blow-up-a-single-point construction, the cloaking constructions based on blowing up an arc curve or a planar rectangle are also proposed and investigated in \cite{GKLU2,LiP}, and they respectively yield the so-called electromagnetic wormholes and carpet-cloaking. As discussed earlier, the regularized blow-up-a-single-point construction, namely the blow-up-a-small-region construction, has been extensively studied in the literature. Using the general framework developed in the present work, one can easily construct the regularized electromagnetic wormholes and carpet-cloaking by employing non-singular materials. In this paper, we focus entirely on the transformation optics approach in achieving the cloaks. We would like to mention in passing other promising cloaking techniques which we did not consider in the present study including the one based on anomalous localized resonance \cite{Ammari0,MN}, and another one based on special (object-dependent) coatings \cite{AE}.

The rest of the paper is organized as follows. In the next section, we present the general blow-up construction of the proposed regularized cloaks and give some relevant discussions. Section 3 is devoted to the convergence analysis in the virtual space. Section 4 is on the application of the results obtained in Section 3 to the construction of full and partial cloaks in the physical space. In Section 5, we develop the ABC-geometry for both the virtual and physical spaces, and construct three types of full and partial cloaks that are new to the literature. Finally, in Section 6, we give some numerical simulations.

\section{General construction of the regularized cloaks}\label{sect:general construction}

In this section, we shall give the general construction of a regularized cloaking device via the transformation optics approach based on a blow-up mapping between the virtual and the physical spaces. The main purpose of this section is to pave the way for our convergence analysis study in the virtual space that shall be conducted in the next section.
We first give a definition of an admissible acoustic configuration.

For any $x\in\mathbb{R}^N$, $N\geq 2$,  we denote $x=(x',x_N)\in\mathbb{R}^{N-1}\times \mathbb{R}$ and $x=(x'',x_{N-1},x_N)\in\mathbb{R}^{N-2}\times\mathbb{R}\times\mathbb{R}$.
For any $r>0$ and any $x\in\mathbb{R}^N$,
$B_r(x)$ denotes the Euclidean ball contained in $\mathbb{R}^N$ with radius $r$ and center $x$, whereas $B'_r(x')$ denotes the Euclideean ball contained in $\mathbb{R}^{N-1}$ with radius $r$ and center $x'$.
Moreover, $B_r=B_r(0)$ and $B'_r=B'_r(0)$. Finally, for any $E\subset \mathbb{R}^N$, we denote $B_r(E)=\bigcup_{x\in E}B_r(x)$.

\begin{defin}
We say that $K\subset \mathbb{R}^N$ is a \emph{scatterer} if $K$ is compact and $G=\mathbb{R}^N\backslash K$ is connected.

A scatterer $K\subset \overline{B_R}$, for some $R>0$, is \emph{regular} if the immersion
$W^{1,2}(B_{R+1}\backslash K)\to L^2(B_{R+1}\backslash K)$ is compact.

We say that a scatterer $K$ is \emph{Lipschitz-regular} if, for some positive constants $r$, $L_1$ and $L_2$, for any $x\in\partial K$ there exists a bi-Lipschitz function $\Phi_x: B_{r}(x)\to\mathbb{R}^N$ such that the following properties hold.
First, for any $z_1$, $z_2\in B_{r}(x)$ we have
$$L_1\|z_1-z_2\|\leq\|\Phi_x(z_1)-\Phi_x(z_2)\|\leq L_2\|z_1-z_2\|.$$
Second, $\Phi_x(x)=0$ and
$\Phi_x(\partial K\cap B_{r}(x))\subset \pi=\{y\in\mathbb{R}^N:\ y_N=0\}$.

A scatterer $K$ is said to be \emph{Lipschitz} if, for some positive constants $r$ and $L$, the following assumptions hold.

For any $x\in\partial K$, there exists a function $\varphi:\mathbb{R}^{N-1}\to\mathbb{R}$, such that $\varphi(0)=0$ and which is Lipschitz with Lipschitz constant bounded by $L$, such that, up to a rigid change of coordinates, we have $x=0$ and
$$B_r(x)\cap \partial K\subset \{y\in B_r(x): y_N=\varphi(y')\}.$$

We say that $x\in \partial K$ belongs to the interior of $\partial K$ if there exists $\delta$, $0<\delta\leq r$, such that $B_{\delta}(x)\cap \partial K = \{y\in B_{\delta}(x): y_N=\varphi(y')\}$. Otherwise we say that $x$ belongs to the boundary of $\partial K$. We remark that the boundary of $\partial K$ might be empty and that, if $x\in \partial K$ belongs to the interior of $\partial K$, then $K$ may lie at most on one side of $\partial K$, that is $B_{\delta}(x)\cap K=B_{\delta}(x)\cap \partial K$, or
$B_{\delta}(x)\cap K=\{y\in B_{\delta}(x): y_N\geq \varphi(y')\}$, or
$B_{\delta}(x)\cap K=\{y\in B_{\delta}(x): y_N\leq \varphi(y')\}$.

For any $x$ belonging to the boundary of $\partial K$, we assume  that there exists another function
$\varphi_1:\mathbb{R}^{N-2}\to\mathbb{R}$, such that $\varphi_1(0)=0$ and which is Lipschitz with Lipschitz constant bounded by $L$, such that, up to the previous rigid change of coordinates, we have $x=0$ and
$$B_r(x)\cap \partial K= \{y\in B_r(x): y_N=\varphi(y'),\ y_{N-1}\leq\varphi_1(y'')\}.$$

Finally, for any $x\in \partial K$, let $e_1(x),\ldots,e_N(x)$ be the unit vectors representing the orthonormal base of the coordinate system for which the previous representations hold. Then we assume that $e_N(x)$ is a Lipschitz function of $x\in\partial K$, with Lipschitz constant bounded by $L$, and $e_{N-1}(x)$ is a Lipschitz function of $x$, as $x$ varies on the boundary of $\partial K$, with Lipschitz constant bounded by $L$.
\end{defin}

Properties of Lipschitz scatterers are thoroughly investigated in \cite[Section~4]{Men-Ron}. Let us just notice that a Lipschitz scatterer $K$ is Lipschitz-regular. Furthermore, a Lipschitz-regular scatterer $K\subset \overline{B_R}$ is regular and the immersion
$W^{1,2}(B_{R+1}\backslash K)\to L^2(\partial K)$ is compact. Notice that for a connected component $\tilde{K}$ of $K$ with empty interior, with a slight abuse of notation, with $\partial \tilde{K}$ we denote two copies of $\tilde{K}$, $(\tilde{K}^+,\tilde{K}^-)$, and by $L^2(\partial \tilde{K})$ we denote the couple $(u^+,u^-)\in L^2(\tilde{K}^+)\times L^2(\tilde{K}^-)$. In such a way we can define the trace of a function $u\in W^{1,2}(B_{R+1}\backslash K)$ on both sides of $\tilde{K}$.

\begin{defin}
We fix $k>0$. An acoustic configuration
$\mathcal{C}=(K_1,K_2,K_3,s,\sigma,q,h,H)$
is \emph{admissible} if the following assumptions hold.

There exist three scatterers $K_1$, $K_2$, $K_3$ which are pairwise disjoint (possibly some or all of them may be the empty set) and such that $K_2$ is regular and $K_3$ is Lipschitz-regular. We set $K=K_1\cup K_2\cup K_3$.

Let $s=s_1+i s_2=s(x)$, $x\in\partial K_3$, be a
complex-valued bounded $\mathcal{H}^{N-1}$-measurable function, with real and imaginary part $s_1$ and $s_2$ respectively, such that
$$
s_1(x)\leq  0\quad\text{and}\quad
s_2(x)\geq 0
\qquad\text{for $\mathcal{H}^{N-1}$-a.e. }x\in \partial K_3.$$
Again, on a connected component $\tilde{K}$ of $K_3$ with empty interior,
$s\in L^{\infty}(\partial \tilde{K})$ means $(s^+,s^-)\in L^{\infty}(\tilde{K}^+)\times L^{\infty}(\tilde{K}^-)$.

Let $\sigma=\sigma(x)$, $x\in\mathbb{R}^N\backslash K$, be an $N\times N$ symmetric matrix whose entries are real-valued measurable functions such that, for some $\lambda$, $0<\lambda\leq 1$, we have
$$\lambda\|\xi\|^2\leq\sigma(x)\xi\cdot\xi\leq \lambda^{-1}\|\xi\|^2\quad\text{for any }\xi\in\mathbb{R}^N\text{ and for a.e. }x\in\mathbb{R}^N\backslash K.$$

Let $q=q_1+i q_2=q(x)$, $x\in\mathbb{R}^N\backslash K$, be a
complex-valued bounded measurable function, with real and imaginary part $q_1$ and $q_2$ respectively, such that, for some $\lambda$, $0<\lambda\leq 1$, we have
$$
q_1(x)\geq \lambda\quad\text{and}\quad
q_2(x)\geq 0\qquad
\text{for a.e. }x\in\mathbb{R}^N\backslash K.$$
Furthermore, we assume that
$$\sigma(x)\equiv I\quad\text{and}\quad q(x)\equiv 1\quad\text{for any $x$ outside a compact subset}.$$

The source term $(h,H)\in L^2(\mathbb{R}^N\backslash K,\mathbb{C}\times \mathbb{C}^N)$ and it has compact support.

Finally, we require that the following scattering problem has only a trivial weak solution
\begin{equation}\label{uniquenesspbm}
\left\{\begin{array}{ll}
\mathrm{div}(\sigma\nabla u)+k^2qu=0 & \text{in }\mathbb{R}^N\backslash K,\\
u=0 & \text{on }\partial K_1,\\
\sigma\nabla u\cdot\nu=0 &\text{on }\partial K_2,\\
\sigma\nabla u\cdot\nu-su=0 &\text{on }\partial K_3,\\
\text{$u$ satisfies the radiation condition,}
\end{array}\right.
\end{equation}
where $\nu$ denotes the exterior normal to $G=\mathbb{R}^N\backslash K$.
\end{defin}

Physically speaking, we have that $K_1$ is a sound-soft scatterer, $K_2$ is sound-hard and we have an impedance boundary condition on $K_3$. The admissible configuration $\mathcal{C}$ represents the scattering object for our study, which consists of the impenetrable obstacle $(K,s)$, the passive medium $(\sigma,q)$, and the active source/sink $(h,H)$. We note the following result that can be proved in a standard way (cf. \cite{LSSZ}).

\begin{prop}\label{existence}
Let $u^i$ be an entire solution to the Helmholtz equation $\Delta u+k^2u=0$ in $\mathbb{R}^N$.
Let us consider an admissible configuration $\mathcal{C}$ in $\mathbb{R}^N$. Then there exists a unique weak solution to the following scattering problem
\begin{equation}\label{scatpbm}
\left\{\begin{array}{ll}
\mathrm{div}(\sigma\nabla u)+k^2qu=-h+\mathrm{div}(H) & \text{in }\mathbb{R}^N\backslash K,\\
u=0 & \text{on }\partial K_1,\\
\sigma\nabla u\cdot\nu=0 &\text{on }\partial K_2,\\
\sigma\nabla u\cdot\nu-su=0 &\text{on }\partial K_3,\\
\text{$u-u^i$ satisfies the radiation condition.}
\end{array}\right.
\end{equation}
\end{prop}

If we take $u^i(x)=e^{ik x\cdot d}$, $x\in\mathbb{R}^N$, to be the plane wave, then the solution $u$ to \eqref{scatpbm} has exactly the same asymptotic development as that in \eqref{eq:asymptotic}, and we denote by
$u_{\infty}(\hat{x},d;\mathcal{C})$, $\hat{x}\in\mathbb{S}^{N-1}$, the far-field pattern of the scattered field $u^s$.

Next, we present a lemma with some key ingredients of the transformation optics, the proofs of which are available in \cite{GKLU4,KOVW,Liu}. The main remark is that the definition of an admissible configuration is stable under bi-Lipschitz transformations.

\begin{lem}\label{lem:trans acoustics}
Let $\Omega$ and $\widetilde{\Omega}$ be two Lipschitz domains in $\mathbb{R}^N$ and $\widetilde
x=F(x):\Omega\rightarrow\widetilde\Omega$ be a bi-Lipschitz and
orientation-preserving mapping.

Let $\mathcal{C}=(\Omega;K_1,K_2,K_3,s,\sigma,q,h,H)$
be an admissible configuration in $\Omega$. For simplicity we assume that $K$ and the supports of $h$ and $H$ are contained in $\Omega$. We let the {\it push-forwarded} configuration be
defined as
\begin{multline}\label{eq:pushforward}
\widetilde{\mathcal{C}}=(\widetilde{\Omega};\widetilde{K}_1,\widetilde{K}_2,\widetilde{K}_3,\widetilde{s},\widetilde{\sigma},\widetilde{q},\widetilde{h},\widetilde{H})
=F_*(\Omega;K_1,K_2,K_3,s,\sigma,q,h,H):=\\(F(\Omega); F(K_1),F(K_2),F(K_3),F_*s,
F_*\sigma, F_*q,
F_*h,F_*H).
\end{multline}

Notice that $\widetilde{K}_1$ is a scatterer, whereas $\widetilde{K}_2$ is a regular scatterer and $\widetilde{K}_3$ is a Lipschitz-regular scatterer. Moreover,
\begin{equation}
\widetilde{s}(\widetilde{x})=F_*s(\widetilde x):=\frac{s(x)}{|\mbox{\emph{det}}(D_{\tau}F(x))|}\bigg|_{x=F^{-1}(\widetilde x)},
\end{equation}
where $D_{\tau}F$ denotes the tangential component of the Jacobian matrix of $F$.
We also have
\begin{equation}
\begin{split}
&\widetilde{\sigma}(\widetilde
x)=F_*\sigma(\widetilde x):=\frac{DF(x)\cdot \sigma(x)\cdot DF(x)^T}{|\mbox{\emph{det}}(DF(x))|}\bigg|_{x=F^{-1}(\widetilde x)},\\
&\widetilde{q}(\widetilde x)=F_*q(\widetilde x):=\frac{q(x)}{|\mbox{\emph{det}}(DF(x))|}\bigg|_{x=F^{-1}(\widetilde x)},
\end{split}
\end{equation}
where $DF$ denotes the Jacobian matrix of $F$.
Finally,
\begin{equation}
\begin{split}
&\widetilde{h}(\widetilde{x})=F_*h(\widetilde x):=\frac{h(x)}{|\mbox{\emph{det}}(DF(x))|}\bigg|_{x=F^{-1}(\widetilde x)},\\
&\widetilde{H}(\widetilde{x})=F_*H(\widetilde x):=\frac{DF(x)\cdot H(x)}{|\mbox{\emph{det}}(DF(x))|}\bigg|_{x=F^{-1}(\widetilde x)}.\\
\end{split}
\end{equation}
Then $u\in H^1(\Omega\backslash K)$ solves the Helmholtz equation
\[\left\{\begin{array}{ll}
\nabla\cdot(\sigma\nabla u)+k^2q u=-h+\nabla\cdot H &\mbox{in\ $\Omega\backslash K$,}\\
u=0&\mbox{on\ $\partial K_1$,}\\
\sigma\nabla u\cdot\nu=0&\mbox{on\ $\partial K_2$,}\\
\sigma\nabla u\cdot\nu-su=0&\mbox{on\ $\partial K_3$,}
\end{array}\right.
\]
if and only if the pull-back field $\widetilde u=(F^{-1})^*u:=u\circ F^{-1}\in H^1(\widetilde\Omega\backslash \widetilde{K})$ solves
\[\left\{\begin{array}{ll}
\widetilde\nabla\cdot(\widetilde\sigma\nabla \widetilde u)+k^2\widetilde q \widetilde u=-\widetilde h+\widetilde\nabla\cdot \widetilde H &\mbox{in\ $\widetilde\Omega\backslash \widetilde K$,}\\
\widetilde u=0&\mbox{on\ $\partial \widetilde K_1$,}\\
\widetilde\sigma\nabla \widetilde u\cdot\nu=0&\mbox{on\ $\partial \widetilde K_2$,}\\
\widetilde\sigma\nabla \widetilde u\cdot\nu-\widetilde s\widetilde u=0&\mbox{on\ $\partial \widetilde K_3.$}
\end{array}\right.
\]
We have made use of $\nabla$ and $\widetilde\nabla$ to distinguish the
differentiations respectively in $x$- and $\widetilde x$-coordinates.
\end{lem}

As a consequence of Lemma~\ref{lem:trans acoustics}, one can directly verify that if $F:\Omega\rightarrow\Omega$ is a bi-Lipschitz map with $F|_{\partial\Omega}=\mbox{Identity}$, then the push-forward of an admissible configuration is again an admissible configuration and
\begin{equation}\label{eq:equal}
u_\infty(\hat{x}, d; \mathcal{C})=u_\infty(\hat{x}, d; \widetilde{\mathcal{C}}),\quad \text{for any }\hat{x}, d\in\mathbb{S}^{N-1}.
\end{equation}
The observation \eqref{eq:equal} is of critical importance for our following general construction of the cloaking scheme. Finally, let us point out that the following quantities remain unchanged under the push-forward
$$\int_{\Omega}(q_i)^{-1}|h|^2=\int_{\widetilde\Omega}(\widetilde q_i)^{-1}|\widetilde h|^2,\ i=1,2,\quad\text{and}\quad
\int_{\Omega}(\sigma)^{-1}H\cdot\overline{H}=\int_{\widetilde\Omega}(\widetilde \sigma)^{-1}\widetilde H\cdot\overline{\widetilde H}.$$

With the above preparations, we are ready to present the general construction of our proposed cloaks.
By a bit abuse of notation, we let from now on $\Omega$, $D$ and $\Sigma$ be the closures of three bounded Lipschitz domains in $\mathbb{R}^N$ such that $\mathbb{R}^N\backslash \Omega$ is connected and $\Sigma\subset \stackrel{\circ}{D}\Subset \stackrel{\circ}{\Omega}$. Let $K_0$ be a compact subset in $\mathbb{R}^N$, and $K_\varepsilon$ be an $\varepsilon$-neighborhood of $K_0$. Both $K_0$ and $K_\varepsilon$ shall be made precise in the next section. We assume there exists a bi-Lipschitz and orientation-preserving mapping $F_\varepsilon$ such that
\begin{equation}\label{eq:FF}
F_\varepsilon=\mbox{Identity}\ \mbox{on\ $\mathbb{R}^N\backslash \Omega$};
\qquad F_\varepsilon^{(1)}\ \mbox{on\ $\Omega\backslash K_\varepsilon$};\qquad F_\varepsilon^{(2)}\ \mbox{on\ $K_\varepsilon$}
\end{equation}
with
\begin{equation}\label{eq:FF1}
F_\varepsilon^{(2)}(K_{\varepsilon/2})=\Sigma,\quad F_\varepsilon^{(2)}(K_{\varepsilon})=D,\quad F_\varepsilon^{(1)}(\Omega\backslash K_\varepsilon)=\Omega\backslash D.
\end{equation}
Let, in the physical space, $\widetilde{\mathcal{C}}^{\varepsilon}=
(\Sigma_1,\Sigma_2,\Sigma_3,s,\widetilde{\sigma}^{\varepsilon},\widetilde{q}^{\varepsilon},h,H)$ be an admissible configuration as described in what follows.
Set $\widetilde \Sigma=\Sigma_1\cup\Sigma_2\cup\Sigma_3$. The following properties are required. First, $\widetilde\Sigma\subset \Sigma$ and  $h$ and $H$ are supported in $\Sigma$.
Second, we have
\begin{equation}\label{eq:phc1}
(\Omega\backslash D; \widetilde\sigma^\varepsilon, \widetilde q^\varepsilon)=(F_\varepsilon)_*(\Omega\backslash K_\varepsilon; I, 1); \quad \widetilde\sigma^\varepsilon(x)=I,\ \ \widetilde q^\varepsilon(x)=1\quad \mbox{for a.e. $x\in\mathbb{R}^N\backslash \Omega$},
\end{equation}
and
\begin{equation}\label{eq:phc2}
(D\backslash\Sigma; \widetilde\sigma^\varepsilon, \widetilde q^\varepsilon)=(F_\varepsilon)_*(K_\varepsilon\backslash K_{\varepsilon/2}; \sigma_l^\varepsilon, q_l^\varepsilon),
\end{equation}
where $(K_\varepsilon\backslash K_{\varepsilon/2}; \sigma_l^\varepsilon, q_l^\varepsilon)$ is a lossy layer that is admissible and shall also be made precise in the subsequent sections. In the physical space, $(\Omega\backslash D; \widetilde\sigma^\varepsilon, \widetilde q^\varepsilon)$ is the cloaking layer, and $(D\backslash\Sigma; \widetilde\sigma^\varepsilon, \widetilde q^\varepsilon)$ is a layer obtained with specially chosen parameters $\sigma^\varepsilon_l$ and $q^\varepsilon_l$. $\Sigma$ is the cloaked region where the target objects are located, and they include the impenetrable obstacle $(\widetilde\Sigma,s)$, the medium  $(\Sigma\backslash\bigcup_{l=1}^3\Sigma_l; \widetilde\sigma^\varepsilon, \widetilde q^\varepsilon)$ and the active source/sink $(h,H)$.
%
Here, we would like to emphasize that in our cloaking scheme, if $h$ and $H$ are zero, namely no source/sink is present inside the cloaked region, then $(\Sigma\backslash\bigcup_{l=1}^3\Sigma_l; \widetilde\sigma^\varepsilon, \widetilde q^\varepsilon)$ and $(\widetilde\Sigma, s)$ could be arbitrary except for admissibility; and only when there is a source/sink present, we would require some generic condition on the cloaked medium $(\Sigma\backslash\bigcup_{l=1}^3\Sigma_l; \widetilde\sigma^\varepsilon, \widetilde q^\varepsilon)$. We finally notice that the configuration inside $\Sigma$ do no depend on $\varepsilon$, that is $\widetilde{\Sigma}$, $s$, $h$ and $H$, and the coefficients $\widetilde{\sigma}^{\varepsilon}$ and $\widetilde{q}^{\varepsilon}$ inside $\Sigma$, do not depend on $\varepsilon$.

 The scattering problem corresponding to the cloaking construction described above is given by
\begin{equation}\label{eq:ps1}
\begin{cases}
\mbox{div}(\widetilde\sigma^\varepsilon \nabla\widetilde u_\varepsilon)+k^2 \widetilde q^\varepsilon \widetilde u_\varepsilon=-h+\mbox{div}(H) \qquad &\mbox{in\ \ $\displaystyle{\mathbb{R}^N\backslash\bigcup_{l=1}^3\Sigma_l}$}\,,\\
\widetilde u_\varepsilon=0 &\mbox{on\ \ $\partial \Sigma_1$},\\
\widetilde\sigma^\varepsilon\nabla \widetilde u_\varepsilon\cdot \nu=0 &\mbox{on\ \ $\partial \Sigma_2$,}\\
\widetilde\sigma^\varepsilon\nabla \widetilde u_\varepsilon\cdot \nu-s\widetilde u_\varepsilon=0 &\mbox{on\ \ $\partial \Sigma_3$,}\\
\text{$\widetilde u_\varepsilon-u^i$ satisfies the radiation condition.}
\end{cases}
\end{equation}
We assume that the configuration is admissible, therefore there exists a unique solution $\widetilde{u}_\varepsilon\in H_{loc}^1(\mathbb{R}^N\backslash\bigcup_{l=1}^3 \Sigma_l)$ to \eqref{eq:ps1}. In the sequel, we let $\widetilde u_\infty^\varepsilon$ denote the corresponding scattering amplitude.

One of the main results of this work is to show the convergence of $\widetilde{u}_\varepsilon$ (and hence $\widetilde u_\infty^\varepsilon$), and then apply it to derive the approximate cloaking in various settings. To that end,
in the virtual space, the admissible configuration $\mathcal{C}^{\varepsilon}=
(K_1^{\varepsilon},K_2^{\varepsilon},K_3^{\varepsilon},s^{\varepsilon},\sigma^{\varepsilon},q^{\varepsilon},h^{\varepsilon},H^{\varepsilon})$
is given by
$$\mathcal{C}^{\varepsilon}=F^{-1}_*(\widetilde{\mathcal{C}}^{\varepsilon}),$$
that is
\begin{equation}\label{eq:obstacle virtual}
F_\varepsilon^{(2)}(K^\varepsilon_l)=\Sigma_l,\quad l=1,2,3,
\end{equation}
and $K^\varepsilon=K_1^\varepsilon\cup K_2^\varepsilon\cup K_3^\varepsilon\subset K_{\varepsilon/2}$.
Furthermore,
$$\sigma^\varepsilon(x)=I,\ \ q^\varepsilon(x)=1\quad \mbox{for a.e. $x\in\mathbb{R}^N\backslash K_{\varepsilon}$},
$$
and
$$\sigma^\varepsilon=\sigma^{\varepsilon}_l,\ \ q^\varepsilon=q^{\varepsilon}_l\quad \mbox{in $K_{\varepsilon}\backslash K_{\varepsilon/2}$}.
$$

We let $u_\varepsilon=(F_\varepsilon)^*\widetilde u_\varepsilon\in H_{loc}^1(\mathbb{R}^N\backslash \bigcup_{l=1}^3 K^\varepsilon_l)$ and,
by Lemma~\ref{lem:trans acoustics}, it is straightforward to verify that $u_\varepsilon$ is the unique solution to
\begin{equation}\label{eq:vs1}
\begin{cases}
\mbox{div}(\sigma^\varepsilon \nabla u_\varepsilon)+k^2  q^\varepsilon  u_\varepsilon=-h^\varepsilon+\mbox{div}(H^{\varepsilon}) \qquad &\mbox{in\ \ $\displaystyle{\mathbb{R}^N\backslash\bigcup_{l=1}^3 K_l^\varepsilon}$},\\
 u_\varepsilon=0 &\mbox{on\ \ $\partial K_1^\varepsilon$},\\
\sigma^\varepsilon\nabla  u_\varepsilon\cdot \nu=0 &\mbox{on\ \ $\partial K_2^\varepsilon$,}\\
\sigma^\varepsilon\nabla u_\varepsilon\cdot \nu-s^\varepsilon u_\varepsilon=0 &\mbox{on\ \ $\partial K_3^\varepsilon$,}\\
\text{$u_\varepsilon-u^i$ satisfies the radiation condition.}
\end{cases}
\end{equation}

We notice that \eqref{eq:ps1} describes the scattering in the physical space corresponding to the cloaking construction, whereas \eqref{eq:vs1} describes the scattering in the virtual space. Clearly, we have $u_\varepsilon=\widetilde u_\varepsilon$ outside a sufficiently large ball. Hence, in order to study the convergence of $\widetilde u_\varepsilon$, it suffices for us to study the convergence of $u_\varepsilon$, which is the main task of Section~\ref{stabsec} in the sequel. We note the following peculiar structure of the virtual scattering problem \eqref{eq:vs1} for our subsequent study. The scattering objects, including the passive penetrable medium and impenetrable obstacles $\bigcup_{l=1}^3 K^\varepsilon_l$, and the active source/sink are included into the innermost region $K_{\varepsilon/2}$, and then they are enclosed by a lossy layer $(K_\varepsilon\backslash K_{\varepsilon/2}; \sigma_l^\varepsilon, q_l^\varepsilon)$. When $K_\varepsilon$ is uniformly small, namely $K_\varepsilon$ will degenerate to a singular point as $\varepsilon\rightarrow +0$, this is the case that has been extensively investigated in \cite{Ammari1,Ammari2,Ammari3,KOVW,KSVW,LiLiuSun,LiuSun}. As emphasized in the introduction, we shall extend such studies to an extremely general setting, especially $K_\varepsilon$ could be `partially small', as will be seen in our subsequent study. The critical scaling arguments developed in \cite{Ammari1,Ammari2,Ammari3,KOVW,KSVW,LiLiuSun,LiuSun} cannot be adapted to treating the much more challenging cases of the current study.

\section{Analysis in the virtual space}\label{stabsec}

This section is devoted to the convergence analysis of the virtual scattering problem \eqref{eq:vs1}. However, our study shall be conducted in a much more general setting than is needed for the cloaking purpose.

Let $K_0$ be a scatterer in $\mathbb{R}^N$.
Let us denote by  $d:\mathbb{R}^N\to[0,+\infty)$ the distance function from $K_0$ defined as follows
$$d(x)=\mathrm{dist}(x,K_0)\quad\text{for any }x\in\mathbb{R}^N.$$

We assume that there exists a Lipschitz function $\tilde{d}:\mathbb{R}^N\to [0,+\infty)$ such that
the following properties are satisfied.

First, there exist constants $a$ and $b$, $0<a\leq1\leq b$, such that
$$ad(x)\leq\tilde{d}(x)\leq bd(x)\quad\text{for any }x\in\mathbb{R}^N.$$

For any $\varepsilon>0$, let us call $K_{\varepsilon}=\{x\in\mathbb{R}^N:\ \tilde{d}(x)\leq \varepsilon\}$.
For some constants $\varepsilon_0>0$, $p>2$, $C_1>0$ and $R>0$, we require that for any $\varepsilon$, $0<\varepsilon\leq \varepsilon_0$,  $K_{\varepsilon}\subset \overline{B_R}$,
$\mathbb{R}^N\backslash K_{\varepsilon}$ is connected
and
$$\|u\|_{L^p(B_{R+1}\backslash K_{\varepsilon})}\leq C_1\|u\|_{H^1(B_{R+1}\backslash K_{\varepsilon})}\quad\text{for any }u\in H^1(B_{R+1}\backslash K_{\varepsilon}).$$

We notice that, under these assumptions, $K_{\varepsilon}$, for any $0\leq\varepsilon\leq\varepsilon_0$, is a regular scatterer.
A simple sufficient condition for these assumptions to hold is that $K_0$ is a compact convex set. In fact, clearly we have that
$\mathbb{R}^N\backslash K_0$ is connected. Then,
we can take $\tilde{d}=d$ or the distance from $K_0$ with respect to any norm on $\mathbb{R}^N$, not only with respect to the Euclidean one. Then, for any $\varepsilon>0$, $K_{\varepsilon}=\overline{B_{\varepsilon}(K_0)}$, clearly with respect to the chosen norm.
For any $\varepsilon>0$, we have that $\overline{B_{\varepsilon}(K)}$ is still a convex set, therefore $\mathbb{R}^N\backslash \overline{B_{\varepsilon}(K_0)}$ is connected and satisfies a cone condition with a cone independent on $\varepsilon$. Hence
also the other required properties are satisfied,
for instance using \cite[Theorem~5.4]{Ada}.

Another sufficient condition is that $K_0$ is a Lipschitz scatterer, see \cite[Section~4]{Men-Ron}.
Finally, $K_0$ may be the union of a finite number of pairwise disjoint
compact convex sets and Lipschitz scatterers.

For any fixed $\varepsilon$, $0<\varepsilon\leq \varepsilon_0$, let us assume that
$\{K_1^{\varepsilon},K_2^{\varepsilon},K_3^{\varepsilon},s^{\varepsilon},\sigma^{\varepsilon},q^{\varepsilon},h^{\varepsilon},H^{\varepsilon}\}$ is an admissible configuration.

Moreover, we require that one of the following two assumptions holds.

\begin{assum}\label{assumption1}
$\{K_1^{\varepsilon},K_2^{\varepsilon},K_3^{\varepsilon},s^{\varepsilon},\sigma^{\varepsilon},q^{\varepsilon},h^{\varepsilon},H^{\varepsilon}\}$ is an admissible configuration with the following properties.
First, $K^{\varepsilon}=K_1^{\varepsilon}\cup K_2^{\varepsilon}\cup K_3^{\varepsilon}\subset K_{\varepsilon/2}$.
Second,
the behavior of $\sigma^{\varepsilon}$ and $q^{\varepsilon}$ is different in the following three regions, $K_{\varepsilon/2}\backslash K^{\varepsilon}$, $K_{\varepsilon}\backslash K_{\varepsilon/2}$
and $\mathbb{R}^N\backslash K_{\varepsilon}$.
\begin{enumerate}[a)]
\item We have that $\sigma^{\varepsilon}(x)=I$ and $q^{\varepsilon}(x)=1$ for almost every $x\in \mathbb{R}^N\backslash K_{\varepsilon}$.
\item There exist a continuous nondecreasing function $\omega_1:(0,\varepsilon_0]\to (0,+\infty)$, such that
$\lim_{s\to 0^+}\omega_1(s)=0$, and positive constants $E_1$, $E'_1$, $\Lambda$, and $E_2$ such that for almost any $x\in K_{\varepsilon}\backslash K_{\varepsilon/2}$
\begin{equation}\label{eq:ll1}
0<q_1^{\varepsilon}(x)\leq E_1(\omega_1(\varepsilon))^{-1}\quad\text{and}\quad 0<E'_1(\omega_1(\varepsilon))^{-1}\leq q_2^{\varepsilon}(x).
\end{equation}
Furthermore
\begin{equation}\label{firstcondsigma}
\sigma^{\varepsilon}(x)\xi\cdot\xi\leq \Lambda\|\xi\|^2\quad\text{for any }\xi\in\mathbb{R}^N\text{ and for a.e. }x\in K_{\varepsilon}\backslash K_{\varepsilon/2}
\end{equation}
and
\begin{equation}\label{secondcondsigmabis}
\frac{1}{\varepsilon^2}\sigma^{\varepsilon}(x)\nabla \tilde{d}(x)\cdot\nabla \tilde{d}(x)\leq E_2\quad\text{for a.e. }x\in K_{\varepsilon}\backslash K_{\varepsilon/2}.
\end{equation}
Finally we require that
\begin{equation}\label{secondassumption}
\lim_{\varepsilon\to 0^+}\int_{K_{\varepsilon}\backslash K_{\varepsilon/2}}|q^{\varepsilon}|=0.
\end{equation}
\item There exists a positive constant $E_3$ such that
$$0<\sup_{K_{\varepsilon/2}\backslash K^{\varepsilon}} q^{\varepsilon}_1\leq E_3\inf_{K_{\varepsilon/2}\backslash K^{\varepsilon}} q^{\varepsilon}_2.$$

We assume that $(h^{\varepsilon},H^{\varepsilon})(x)=0$ for almost every $x\in \mathbb{R}^N\backslash K_{\varepsilon/2}$. Moreover, we have that
\begin{equation}\label{sourceassumption}
\int_{K_{\varepsilon/2}\backslash K^{\varepsilon}}(k^2q_2^{\varepsilon})^{-1}|h^{\varepsilon}|^2 \leq E_3\quad\text{and}\quad\int_{K_{\varepsilon/2}\backslash K^{\varepsilon}}(\sigma^{\varepsilon})^{-1}H^{\varepsilon}\cdot\overline{H^{\varepsilon}}   \leq E_3.
\end{equation}
\end{enumerate}
Here the constants $E_1$, $E'_1$, $\Lambda$, $E_2$ and $E_3$ do not depend on $\varepsilon$.
\end{assum}

We remark that the previously stated assumption, essentially because $q^{\varepsilon}_2>0$ in $K_{\varepsilon}\backslash K^{\varepsilon}$, implies that the corresponding problem \eqref{uniquenesspbm} has only the trivial solution without any further assumption on the obstacles $K_1^{\varepsilon}$, $K_2^{\varepsilon}$, $K_3^{\varepsilon}$ and the coefficients
$s^{\varepsilon}$ and
 $\sigma^{\varepsilon}$, $q^{\varepsilon}$.

\begin{assum}\label{assumption2}
Assumption~\ref{assumption2} coincides with Assumption~\ref{assumption1} except for point c) which is replaced by the following.
\begin{enumerate}[c)]
\item We assume that $H^{\varepsilon}\equiv 0$ and that
$h^{\varepsilon}(x)=0$ for almost every $x\in \mathbb{R}^N\backslash K_{\varepsilon/2}$. Moreover, there exists a positive constant $E_3$ such that
\begin{equation}\label{sourceassumptionbis}
\int_{K_{\varepsilon/2}\backslash K^{\varepsilon}}(k^2q_2^{\varepsilon})^{-1}|h^{\varepsilon}|^2 \leq E_3.
\end{equation}
Notice that the above condition means that $q_2^{\varepsilon}>0$ almost everywhere on the set $\{h^{\varepsilon}\neq 0\}$ and that the integral is actually performed on such a set.
\end{enumerate}
\end{assum}

\begin{oss}
Let us point out here the main differences between these two assumptions.
In Assumption~\ref{assumption2} we do not impose any condition on the coefficients and scatterers inside $K_{\varepsilon/2}$ except for the indirect condition in \eqref{sourceassumptionbis}. In particular, if there are no sources, then no condition at all is required for the configuration inside $K_{\varepsilon/2}$ except for its admissibility. The drawback is that the kind of sources we can allow in Assumption~\ref{assumption2} is more limited than those
considered in Assumption~\ref{assumption1}.

Let us finally notice that, only for Assumption~\ref{assumption1}, we may replace \eqref{secondcondsigmabis} with the following weaker condition
\begin{equation}\label{secondcondsigma}
\lim_{\varepsilon\to 0^+}\frac{1}{\varepsilon^2}\int_{K_{\varepsilon}\backslash K_{\varepsilon/2}}\sigma^{\varepsilon}\nabla \tilde{d}\cdot\nabla \tilde{d}=0.
\end{equation}
\end{oss}

We shall keep fixed throughout this section two functions $\phi_1,\ \phi_2\in C^{\infty}(\mathbb{R})$ such that they are nondecreasing and $\phi_1$ is identically equal to $0$ on a neighborhood of  $(-\infty,3/4]$ and identically equal to $1$ on a neighborhood of $[1,+\infty)$, whereas
$\phi_2$ is identically equal to $0$ on a neighborhood of  $(-\infty,1/2]$ and identically equal to $1$ on a neighborhood of $[3/4,+\infty)$. Also for any symmetric positive definite matrix $\sigma$, we define $\sqrt{\sigma}$ the symmetric positive definite matrix such that $\sqrt{\sigma}\cdot\sqrt{\sigma}=\sigma$.

Let us also fix $\varepsilon_n$, $0<\varepsilon_n\leq \varepsilon_0$,
$n\in\mathbb{N}$, a decreasing sequence of positive numbers such that $\lim_n\varepsilon_n=0$.
For simplicity
we denote
$$\begin{array}{c}K_n=K_{\varepsilon_n}\quad\text{and}\quad H_n=K_{\varepsilon_n/2}
\quad\text{and}\quad J_n=K_{3\varepsilon_n/4},\\
\{K_1^{\varepsilon_n},K_2^{\varepsilon_n},K_3^{\varepsilon_n},s^{\varepsilon_n},\sigma^{\varepsilon_n},q^{\varepsilon_n},h^{\varepsilon_n},H^{\varepsilon_n}\}=\{K_1^n,K_2^n,K_3^n,s^n,\sigma^n,q^n,h^n,H^n\},\\
K^n=K_1^n\cup K_2^n\cup K_3^n=K^{\varepsilon_n}.
\end{array}$$

We begin with the following result where a Mosco-type convergence is proved.

\begin{lem}\label{Moscotypelemma}
Under the previous assumptions, the following two properties hold.
\begin{enumerate}[i\textnormal{)}]
\item \label{Mosco1}
Let  us consider a subsequence $\varepsilon_{n_l}$ and a sequence $u_l\in H^1(B_{R+1}\backslash H_{n_l})$, $l\in\mathbb{N}$.
Let us consider $(u_l,\sqrt{\sigma^{n_l}}\nabla u_l)\in L^2(B_{R+1},\mathbb{R}^{N+1})$ by extending $u_l$ and $\sqrt{\sigma^{n_l}}\nabla u_l$ to $0$ in $H_{n_l}$.
If  $(u_l,\sqrt{\sigma^{n_l}}\nabla u_l(1-\chi_{J_{n_l}}))$ converges weakly to $(u,V)$ in $L^2(B_{R+1},\mathbb{R}^{N+1})$ as $l\to\infty$, then $u\in H^1(B_{R+1}\backslash K_0)$ and $V=\nabla u$  in $B_{R+1}\backslash K_0$.
\item \label{Mosco2}
For any $\varphi\in H^1(B_{R+1}\backslash K_0)$ such that $\varphi$ is bounded, there exists a sequence $\varphi_n\in H^1(B_{R+1})$, $n\in\mathbb{N}$, such that
$\varphi_n=0$ in $J_n$ and
$(\varphi_n,\sqrt{\sigma^n}\nabla \varphi_n)$ converges strongly to $(\varphi,\nabla\varphi)$ in $L^2(B_{R+1},\mathbb{R}^{N+1})$ as $n\to\infty$,
where $\varphi$ and $\nabla\varphi$ are extended to $0$ in $K_0$.
\end{enumerate}
\end{lem}

\proof We begin by proving $\ref{Mosco1})$. Let $D$ be any open subset of $B_{R+1}$ such that
$\overline{D}\cap K_0=\emptyset$. There exists $\overline{l}\in\mathbb{N}$ such that for any $l\geq \overline{l}$ we have that, in $D$,  $(u_l,\sqrt{\sigma^{n_l}}\nabla u_l(1-\chi_{J_{n_l}}))=(u_l,\nabla u_l)$. Therefore $V=\nabla u$ in $D$ and we may easily conclude that $u\in H^1(B_{R+1}\backslash K_0)$ and $V=\nabla u$ in $B_{R+1}\backslash K_0$. Moreover, $v$ and $V$ are $0$ almost everywhere in $K_0$.

For what concerns $\ref{Mosco2})$, we define $\chi_n=\phi_1(\tilde{d}/\varepsilon_n)$ and $\varphi_n=\chi_n\varphi$. We immediately obtain that $\varphi_n$ is identically $0$ in $J_n$ and that
$\varphi_n=\varphi$ in $B_{R+1}\backslash K_n$. Clearly $\varphi_n\in H^1(B_{R+1})$ and, as $n\to\infty$, $\varphi_n$ converges to $\varphi$ in $L^2(B_{R+1})$.
Then we need to estimate
$$\int_{B_{R+1}}\|\sqrt{\sigma^n}\nabla\varphi_n-\nabla\varphi\|^2\leq
\int_{K_n\backslash K_0}\|\nabla\varphi\|^2+\int_{K_n\backslash H_n}\|\sqrt{\sigma^n}\nabla\varphi_n\|^2.
$$
The first term of the right-hand side clearly goes to $0$ as $n\to\infty$. About the second term, we may notice that
$$\|\sqrt{\sigma^n}\nabla\varphi_n\|^2\leq 2\left(
\|(\sqrt{\sigma^n}\nabla\varphi)\chi_n\|^2+\|(\sqrt{\sigma^n}\nabla\chi_n)\varphi\|^2\right)
$$
and that
$$\nabla\chi_n=\phi'_1(\tilde{d}/\varepsilon_n)
\frac{\nabla \tilde{d}}{\varepsilon_n}.$$
The conclusion immediately follows by \eqref{firstcondsigma} and \eqref{secondcondsigma}.\cvd

\bigskip

Let $u^i$ be an entire solution to the Helmholtz equation $\Delta u+k^2u=0$ in $\mathbb{R}^N$.
We consider the scattering problem \eqref{scatpbm} with the configuration
$\{K_1,K_2,K_3,s,\sigma,q,h,H\}$
replaced by
$\{K_1^n,K_2^n,K_3^n,s^n,\sigma^n,q^n,h^n,H^n\}$ and we denote by $u_n$ its solution.


We have the following theorem.

\begin{thm}\label{scatteo}
Under the previous assumptions, $v_n=u_n(1-\chi_{H_n})$ converges to a function $u$ strongly in $L^2(B_r)$ for any $r>0$, with $u$ solving
\begin{equation}\label{uscateq}
\left\{\begin{array}{ll}
\Delta u+k^2u=0 &\text{in }\ \mathbb{R}^N\backslash K_0,\\
u=u^i+u^s &\text{in }\ \mathbb{R}^N\backslash K_0,\\
\nabla u\cdot \nu=0 &\text{on }\ \partial K_0,\\
\lim_{r\to+\infty}r^{(N-1)/2}\left(\frac{\partial u^s}{\partial r}-i ku^s\right)=0 &r=\|x\|,
\end{array}\right.
\end{equation}
and $u=0$ in $K_0$.
\end{thm}

\proof We develop the proof when Assumption~\ref{assumption1} holds true. At suitable points of the proof, we shall describe the needed modification if Assumption~\ref{assumption2} is used instead.

Let us begin by assuming the following further condition
\begin{equation}\label{boundedness}
\|v_n\|_{L^2(B_{R+1})}\leq C\quad\text{for any }n\in \mathbb{N}.
\end{equation}

Let us assume that, up to a subsequence, we have that $v_n$ converges weakly to $u$
in $L^2(B_{R+1})$ as $n\to\infty$. We immediately infer that $u=0$ in $K_0$.

We notice that $u_n$ satisfies the Helmholtz equation $\Delta u+k^2 u=0$ in $\mathbb{R}^N\backslash \overline{B_R}$.
By standard regularity estimates, we may infer that there exists a constant $C_1$, depending on the constant $C$ in \eqref{boundedness} and $R$, such that for $R_1=R+(1/2)$ we have
$$\|u_n\|_{C^0(\partial B_{R_1})}+\|\nabla u_n\|_{C^0(\partial B_{R_1})}\leq C_1\quad\text{for any }n\in\mathbb{N}.$$

By the weak formulation of \eqref{scatpbm}, we infer that
\begin{multline}
\int_{B_{R_1}\backslash K^n}\sigma^n\nabla u_n\cdot\nabla \overline{u_n}=\\
\int_{B_{R_1}\backslash K^n}k^2q^n |u_n|^2+\int_{\partial K_3^n}s^n|u_n|^2+\int_{H_n\backslash K^n}h\overline{u_n}+
\int_{H_n\backslash K^n}H\cdot\nabla \overline{u_n}
+\int_{\partial B_{R_1}}(\nabla u_n\cdot\nu) \overline{u_n},
\end{multline}
where we recall that $K^n=K_1^n\cup K_2^n\cup K_3^n$.

Clearly the exists a constant $C_2$, depending on $C_1$ and $R$ only, such that
$$\left|\int_{\partial B_{R_1}}(\nabla u_n\cdot\nu) \overline{u_n}\right|\leq C_2\quad\text{for any }n\in\mathbb{N}.$$

We need that
$\{(v_n,\sqrt{\sigma^n}\nabla v_n(1-\chi_{J_n}))\}_{n\in\mathbb{N}}$ is uniformly bounded in $L^2(B_{R_1},\mathbb{R}^{N+1})$, that is
\begin{equation}\label{boundedness2}
\|(v_n,\sqrt{\sigma^n}\nabla v_n(1-\chi_{J_n}))\|_{L^2(B_{R_1},\mathbb{R}^{N+1})}\leq C_3\quad\text{for any }n\in \mathbb{N}.
\end{equation}
Notice that we are assuming as usual that $\sqrt{\sigma^n}\nabla v_n$ is extended to zero in
$H_n$. If this is the case,
up to a subsequence, $(v_n,\sqrt{\sigma^n}\nabla v_n(1-\chi_{J_n}))$ converges weakly in $L^2(B_{R_1},\mathbb{R}^{N+1})$ to $(v,V)$ as $n\to\infty$. Clearly $v=u$, and,
by property $\ref{Mosco1}\textnormal{)}$ of Lemma~\ref{Moscotypelemma}, we have that $u\in H^1(B_{R_1}\backslash K_0)$ and $V=\nabla u$ in $B_{R_1}\backslash K_0$. Moreover, $u$ and $V$ are zero in $K_0$.

Let us prove such a uniform boundedness. Here the proof is slightly different depending on the Assumption used. We begin with Assumption~\ref{assumption1}.

Obviously $$0\leq \int_{B_{R_1}\backslash K_n}k^2q^n |u_n|^2=k^2
\int_{B_{R_1}\backslash K_n}|u_n|^2\leq k^2C^2.$$
We split into the real and imaginary parts and, by the properties of $q^n$ and $s^n$, we obtain that,
inside $K_n$, $\int_{H_n\backslash K^n}q^n_2 |u_n|^2$ and $\int_{K_n \backslash H_n}q^n_2 |u_n|^2$ are not negative. Furthermore,
\begin{multline*}
0\leq \int_{H_n\backslash K^n}\!\!\! k^2q^n_2 |u_n|^2+\int_{K_n \backslash H_n}\!\!\! k^2q^n_2 |u_n|^2\leq\\
\frac{E_3}{2}+\frac{1}{2}\int_{H_n\backslash K^n}\!\!\! k^2q^n_2|u_n|^2
+\frac{E_3}{2c}+\frac{c}{2}\int_{H_n\backslash K^n}\!\!\! \sigma^n\nabla u_n\cdot\nabla \overline{u_n}+
C_2,
\end{multline*}
for any positive constant $c$, which we shall fix later. Hence, we have that both integrals
$\int_{H_n\backslash K^n}k^2q^n_2 |u_n|^2$ and $\int_{K_n \backslash H_n}k^2q^n_2 |u_n|^2$ are bounded by the following quantity
$$A_n=E_3
+\frac{E_3}{c}+c\int_{H_n\backslash K^n}\!\!\! \sigma^n\nabla u_n\cdot\nabla \overline{u_n}+
2C_2.$$
We conclude that
\begin{equation}\label{u_nbound}
k^2\int_{K_n\backslash H_n} |u_n|^2\leq (A_n/E'_1)\omega_1(\varepsilon_n)
\quad\text{and}\quad
k^2\int_{H^n\backslash K^n} |u_n|^2\leq A_n\left(\inf_{H_n\backslash K^n} q^{\varepsilon}_2\right)^{-1}.
\end{equation}
Hence
$$\int_{K_n\backslash H_n}k^2q^n_1|u_n|^2\leq (E_1/E'_1)A_n
\quad\text{and}\quad
\int_{H_n\backslash K^n}k^2q^n_1 |u_n|^2\leq E_3A_n.$$
We deduce that
\begin{multline*}
\int_{B_{R_1}\backslash K^n} \sigma^n\nabla u_n\cdot\nabla \overline{u_n}\leq\\
 k^2C^2+
(E_1/E'_1)A_n+E_3A_n+\frac{E_3}{2}+\frac{1}{2}A_n
+\frac{E_3}{2c_1}+\frac{c_1}{2}\int_{H_n\backslash K^n}\!\!\! \sigma^n\nabla u_n\cdot\nabla \overline{u_n}+C_2
\end{multline*}
for any positive constant $c_1$, that we shall fix later.
Setting $\tilde{E}=(E_1/E'_1)+E_3+1/2$ and $\tilde{C}=k^2C^2+E_3/2+C_2$, we choose $c$ and $c_1$ such that $\tilde{E}c+c_1/2=1/2$ and we obtain that
$$\int_{B_{R_1}\backslash K^n}\sigma^n\nabla u_n\cdot\nabla \overline{u_n}\leq 2\tilde{C}+
2\tilde{E}\tilde{E}_1+E_3/c_1
$$
where $\tilde{E}_1=E_3+E_3/c+2C_2$.

We conclude that there exists a constant $\tilde{C}_1$ such that
$$\int_{B_{R_1}\backslash K^n}\sigma^n\nabla u_n\cdot\nabla \overline{u_n}\leq \tilde{C}_1\quad\text{for any }n\in\mathbb{N}$$
thus in particular
\begin{equation}\label{unifboundgrad}
\int_{B_{R_1}\backslash J_n}\sigma^n\nabla u_n\cdot\nabla \overline{u_n}\leq \tilde{C}_1\quad\text{for any }n\in\mathbb{N}
\end{equation}
and the required uniform boundedness property is achieved.

Moreover, we can also infer that
there exists a constant $A$ such that $A_n\leq A$ for any $n\in\mathbb{N}$ and that
\begin{equation}\label{u_nbound2}
k^2\int_{K_n\backslash H_n} |u_n|^2\leq (A/E'_1)\omega_1(\varepsilon_n)\quad\text{for any }n\in\mathbb{N}.
\end{equation}
Finally, there exists a constant $\tilde{C}_2$ such that
$$\int_{B_{R_1}\backslash K^n}k^2|q^n| |u_n|^2\leq
\tilde{C}_2\quad\text{for any }n\in\mathbb{N}$$
and in particular such that
\begin{equation}\label{u_nbound3}
\int_{B_{R_1}\backslash H_n}k^2|q^n| |u_n|^2\leq
\tilde{C}_2\quad\text{for any }n\in\mathbb{N}.
\end{equation}

For what concerns Assumption~\ref{assumption2}, the reasoning may be changed in the following way. We have that
$$0\leq \int_{H_n\backslash K^n}\!\!\! k^2q^n_2 |u_n|^2+\int_{K_n \backslash H_n}\!\!\! k^2q^n_2 |u_n|^2\leq
\frac{E_3}{2}+\frac{1}{2}\int_{H_n\backslash K^n}\!\!\! k^2q^n_2|u_n|^2+C_2.$$
Therefore, we obtain that
$$0\leq \int_{K_n \backslash H_n}\!\!\! k^2q^n_2 |u_n|^2\leq
\frac{E_3}{2}+C_2=A.$$
We easily conclude that \eqref{u_nbound2} and \eqref{u_nbound3} hold true also in this case.

Let us now define, for any $n\in\mathbb{N}$, $\chi_n=\phi_2(\tilde{d}/\varepsilon_n)$. We apply a Caccioppoli-type inequality. Namely, we have that
$$\int_{B_{R_1}\backslash H_n}\sigma^n\nabla u_n\cdot\nabla(\overline{u_n}\chi_n^2)=
\int_{B_{R_1}\backslash H_n}k^2q^n|u_n|^2\chi_n^2+\int_{\partial B_{R_1}}(\nabla u_n\cdot\nu) \overline{u_n}.$$
Therefore,
\begin{multline*}
\int_{B_{R_1}\backslash H_n}\chi_n^2\sigma^n\nabla u_n\cdot\nabla \overline{u_n}\leq
\tilde{C}_2+C_2+2\left|
\int_{B_{R_1}\backslash H_n}\overline{u_n}\chi_n\sigma^n\nabla u_n\cdot\nabla \chi_n
\right|\leq \\
\tilde{C}_2+C_2+\frac{1}{2}\int_{B_{R_1}\backslash H_n}\chi_n^2
\sigma^n\nabla u_n\cdot\nabla \overline{u_n}+2
\int_{B_{R_1}\backslash H_n}
|u_n|^2
\sigma^n\nabla \chi_n\cdot\nabla \chi_n,
\end{multline*}
that is
$$\int_{B_{R_1}\backslash H_n}\chi_n^2\sigma^n\nabla u_n\cdot\nabla \overline{u_n}\leq
2\tilde{C}_2+2C_2+\frac{4}{\varepsilon_n^2}
\int_{K_n\backslash H_n}
|u_n|^2
(\phi'_2(\tilde{d}/\varepsilon_n))^2\sigma^n
\nabla \tilde{d}\cdot\nabla \tilde{d}.
$$
By \eqref{secondcondsigmabis}, we conclude that \eqref{unifboundgrad} also holds true.

Let us fix $\varphi\in H^1(B_{R+1}\backslash K_0)$ whose support is contained in $B_{R_1}$
and such that $\varphi$ is bounded.
By property $\ref{Mosco2}\textnormal{)}$ of Lemma~\ref{Moscotypelemma},
we can construct for any $n\in\mathbb{N}$ a function
$\varphi_n\in H^1(B_{R+1})$, whose support is contained in $B_{R_1}$ and such that
$\varphi_n=0$ in $J_n$, with
$(\varphi_n,\sqrt{\sigma^n}\nabla \varphi_n)$ converging strongly to $(\varphi,\nabla\varphi)$ in $L^2(B_{R+1},\mathbb{R}^{N+1})$ as $n\to\infty$,
where $\varphi$ and $\nabla\varphi$ are extended to $0$ in $K_0$.

We have that
$$\int_{B_{R_1}}\sigma^n\nabla v_n\cdot\nabla \varphi_n-\int_{B_{R_1}}k^2q^n v_n\varphi_n=0.$$
It is easy to show that
$$\lim_{n\to\infty}\int_{B_{R_1}}\sigma^n\nabla v_n\cdot\nabla \varphi_n=\int_{B_{R+1}\backslash K_0}\nabla u\cdot\nabla \varphi.$$

We observe that
\begin{multline*}
\int_{B_{R_1}}\!\! k^2q^n v_n\varphi_n=\int_{B_{R+1}\backslash K_n}\!\! k^2v_n\varphi_n+\int_{K_n\backslash H_n}\!\! k^2q^n u_n\varphi_n=\\
\int_{B_{R+1}}\!\! k^2v_n\varphi\chi_{B_{R+1}\backslash K_n}+
\int_{K_n\backslash H_n}\!\! k^2q^n u_n\varphi_n.
\end{multline*}
Obviously
$$\lim_{n\to\infty}\int_{B_{R+1}}\!\! v_n\varphi\chi_{B_{R+1}\backslash K_n}=
\int_{B_{R+1}\backslash K_0}u\varphi.$$
Finally, since $\varphi_n$ is uniformly bounded and
$\int_{K_n\backslash H_n}k^2|q^n| |u_n|^2$ is uniformly bounded,
 and by \eqref{secondassumption}, we are able to pass to the limit and prove that
$$\int_{B_{R+1}\backslash K_0}\nabla u\cdot\nabla \varphi-k^2 u\varphi=0.
$$
By density, we have this last equation holds for
any $\varphi\in H^1(B_{R+1}\backslash K_0)$ whose support is compactly contained in $B_{R_1}$.

By Lemma~3.1 in \cite{Ron03}, we have that, up to a subsequence,
$u_n$ converges to a function $u$ uniformly on compact subsets of $\mathbb{R}^N\backslash \overline{B_R}$,
with $u$ solving
$$
\left\{\begin{array}{ll}
\Delta u+k^2u=0 &\text{in }\mathbb{R}^N\backslash \overline{B_R}\\
u=u^i+u^s &\text{in }\mathbb{R}^N\backslash \overline{B_R}\\
\lim_{r\to+\infty}r^{(N-1)/2}\left(\frac{\partial u^s}{\partial r}-i ku^s\right)=0 &r=\|x\|.
\end{array}\right.
$$

Then we immediately conclude that $u$ is the unique solution to \eqref{uscateq}. We now show that
the whole sequence $v_n$ converges to $u$ strongly in $L^2(B_r)$ for any $r>0$.
By uniqueness of the solution to \eqref{uscateq}, we have that the whole sequence $v_n$ converges to $u$ weakly in $L^2(B_{R+1/2})$ and uniformly
on compact subsets of $\mathbb{R}^N\backslash \overline{B_R}$. Moreover,
by standard regularity estimates, we may also assume that $v_n$ converges to $u$ strongly in $L^2$ on any compact subset of $B_{R+1}\backslash K_0$.

We recall that there exist constants $p>2$ and $C_4>0$ such that for any $n\in\mathbb{N}$ we have
\begin{equation}\label{Sobolev}
\|u\|_{L^p(B_{R_1}\backslash K_n)}\leq C_4\|u\|_{H^1(B_{R_1}\backslash K_n)}\quad\text{for any }u\in H^1(B_{R_1}\backslash K_n).
\end{equation}

Let us fix a positive constant $h_0$, $0<h_0\leq \min\{\varepsilon_0,1/4\}$.
For any $h$, $0<h\leq h_0$, let $A_h=K_h\cup (B_{R_1}\backslash \overline{B_{R_1-h}})$ and $B_h=B_{R_1}\backslash A_h$. We notice that $B_h$ is compactly contained in $B_{R_1}\backslash K_0$
and that $\lim_{h\to 0^+}|A_h\backslash K_0|=0$.

Therefore, for any $h$, $0<h\leq h_0$, there exists $\overline{n}\in\mathbb{N}$ such that for any $n\geq \overline{n}$ we have $\varepsilon_n\leq h$ and
$$\|v_n-u\|_{L^2(B_{R_1})}\leq \|u\|_{L^2(A_h\backslash K_0)}+\|u_n\|_{L^2(K_n\backslash H_n)}+\|u_n\|_{L^2(A_h\backslash K_n)}+
  \|u_n-u\|_{L^2(B_h)}.$$
By \eqref{Sobolev} and \eqref{boundedness2}, we infer that
$$\|v_n-u\|_{L^2(B_{R_1})}\leq
 \|u\|_{L^2(A_h\backslash K_0)}+\|u_n\|_{L^2(K_n\backslash H_n)}+C_3C_4|A_h\backslash K_n|^{(p-2)/(2p)}+
  \|u_n-u\|_{L^2(B_h)}.
$$
By \eqref{u_nbound2}, we conclude that for any $n\geq \overline{n}$
$$\|v_n-u\|_{L^2(A_h)}\leq
 \|u\|_{L^2(A_h\backslash K_0)}+
\sqrt{(A/(k^2E'_1))\omega_1(\varepsilon_n)}
+C_3C_4|A_h\backslash K_0|^{(p-2)/(2p)}.
$$

We fix $\delta>0$. There exist $\overline{h}$, $0<\overline{h}\leq h_0$, and $\overline{n}\in\mathbb{N}$ such that
for any $n\geq \overline{n}$ we have $\varepsilon_n\leq \overline{h}$ and
$$\|v_n-u\|_{L^2(A_{\overline{h}})}\leq \delta/2.
$$
There exists $\tilde{n}\geq \overline{n}$ such that for any $n\geq \tilde{n}$ we have
$$\|u_n-u\|_{L^2(B_{\overline{h}})}\leq \delta/2,$$
therefore we have proved that $v_n$ converges to $u$ strongly in $L^2(B_{R_1})$ and consequently strongly in $L^2(B_r)$ for any $r>0$.

In order to conclude the proof, we have to show that \eqref{boundedness} holds true.
Let $a_n=\|v_n\|_{L^2(B_{R+1})}$.
By contradiction, let us assume that $\lim_{n\to\infty} a_n=+\infty$, possibly by passing to a subsequence.
Let us consider $w_n=u_n/a_n$. We have that
$$\|w_n(1-\chi_{H_n})\|_{L^2(B_{R+1})}=1.$$

Therefore $w_n(1-\chi_{H_n})$ converges to a function $w$ strongly in $L^2$ on any compact subset of $\mathbb{R}^N$.
The function $w$ satisfies
\begin{equation}
\left\{\begin{array}{ll}
\Delta w+k^2w=0 &\text{in }\mathbb{R}^N\backslash K_0\\
\nabla w\cdot \nu=0 &\text{on }\partial K_0.
\end{array}\right.
\end{equation}
Clearly we also have that $\|w\|_{L^2(B_{R+1})}=1$.

We have that $\|u^s_n/a_n\|_{L^2(B_{R+1}\backslash \overline{B_R})}$, $n\in\mathbb{N}$, is uniformly bounded. Therefore, again up to a subsequence, $u^s_n/a_n$ converges, as $n\to\infty$, to a function $\tilde{w}$ strongly in $L^2$ on any compact subset of $\mathbb{R}^N\backslash \overline{B_R}$.
Such a function $\tilde{w}$ satisfies
\begin{equation}
\left\{\begin{array}{ll}
\Delta \tilde{w}+k^2\tilde{w}=0 &\text{in }\mathbb{R}^N\backslash \overline{B_R}\\
\lim_{r\to+\infty}r^{(N-1)/2}\left(\frac{\partial \tilde{w}}{\partial r}-i k\tilde{w}\right)=0 &r=\|x\|.
\end{array}\right.
\end{equation}

Since $w_n=u^i/a_n+u^s_n/a_n$, we may immediately conclude that, outside $\overline{B_R}$, we have $w=\tilde{w}$.
That is $w$ solves
\begin{equation}
\left\{\begin{array}{ll}
\Delta w+k^2w=0 &\text{in }\mathbb{R}^N\backslash K_0\\
\nabla w\cdot \nu=0 &\text{on }\partial K_0\\
\lim_{r\to+\infty}r^{(N-1)/2}\left(\frac{\partial w}{\partial r}-i kw\right)=0 &r=\|x\|.
\end{array}\right.
\end{equation}
By uniqueness, and since $\mathbb{R}^N\backslash K_0$ is connected, we may conclude that
$w$ is identically zero, which leads to a contradiction since
$\|w\|_{L^2(B_{R+1})}$ should be equal to $1$.\cvd

\bigskip

We summarize the results of this section in the following theorem.

\begin{thm}\label{thm:vmain1}
Under the previous assumptions, for any $\varepsilon$, $0<\varepsilon\leq\varepsilon_0$, let us fix ${\sigma}_l^\varepsilon={\sigma}_l^\varepsilon(x)$,
$x\in K_{\varepsilon}\backslash K_{\varepsilon/2}$,
an $N\times N$ symmetric matrix whose entries are real-valued measurable functions such that, for some $\lambda_{\varepsilon}$, $0<\lambda_{\varepsilon}\leq 1$, we have
$$\lambda_{\varepsilon}\|\xi\|^2\leq\hat\sigma^\varepsilon(x)\xi\cdot\xi\quad\text{for any }\xi\in\mathbb{R}^N\text{ and for a.e. }x\in K_{\varepsilon}\backslash K_{\varepsilon/2}.$$
Let ${q}_l^{\varepsilon}=\hat{q}_1^{\varepsilon}+i \hat{q}_2^{\varepsilon}={q}_l^{\varepsilon}(x)$, $x\in K_{\varepsilon}\backslash K_{\varepsilon/2}$, be a
complex-valued bounded measurable function, with real and imaginary parts $\hat{q}_1^{\varepsilon}$ and $\hat{q}_2^{\varepsilon}$ respectively, such that, for some $\lambda_{\varepsilon}$, $0<\lambda_{\varepsilon}\leq 1$, we have
$$\hat{q}_1^{\varepsilon}(x)\geq \lambda_{\varepsilon}\quad\text{and}\quad
\hat{q}_2^{\varepsilon}(x)\geq 0\qquad
\text{for a.e. }x\in K_{\varepsilon}\backslash K_{\varepsilon/2}.$$

We further require that ${\sigma}_l^\varepsilon$ and ${q}_l^{\varepsilon}$ satisfy Assumption~\textnormal{\ref{assumption1}b)}, for a given function $\omega_1$ and
given constants $E_1$, $E'_1$, $\Lambda$, and $E_2$.

Then let $\{K_1^{\varepsilon},K_2^{\varepsilon},K_3^{\varepsilon},s^{\varepsilon},\sigma^{\varepsilon},q^{\varepsilon},h^{\varepsilon},H^{\varepsilon}\}$ be any admissible configuration such that
$$\sigma^{\varepsilon}={\sigma}_l^{\varepsilon}\text{ and }q^{\varepsilon}={q}_l^{\varepsilon}\quad\text{  in }K_{\varepsilon}\backslash K_{\varepsilon/2}
$$
and satisfies, for a given constant $E_3$, either Assumption~\textnormal{\ref{assumption1}} or Assumption~\textnormal{\ref{assumption2}}. Notice that if Assumption~\textnormal{\ref{assumption1}} is used, we may replace \eqref{secondcondsigmabis} with \eqref{secondcondsigma}.

For any $d\in \mathbb{S}^{N-1}$, let $u^i(x)=e^{i kx\cdot d}$, $x\in\mathbb{R}^N$.
Let $u_{\varepsilon}$ be the solution to \eqref{scatpbm} with the configuration
$\{K_1,K_2,K_3,s,\sigma,q,h,H\}$.
replaced by
$\{K_1^{\varepsilon},K_2^{\varepsilon},K_3^{\varepsilon},s^{\varepsilon},\sigma^{\varepsilon},q^{\varepsilon},h^{\varepsilon},H^{\varepsilon}\}$, and $u$ be the solution to \eqref{uscateq}. Let $u_{\infty}^{\varepsilon}$ and $u_{\infty}$ be the scattering amplitudes of the corresponding scattered waves, respectively.

Then there exists a function $\omega:(0,\varepsilon_0]\to (0,+\infty]$ such that $\lim_{s\to 0^+}\omega(s)=0 $, depending on $K_0$, $\tilde{d}$, $R$,
${\sigma}_l^{\varepsilon}$, ${q}_l^{\varepsilon}$, and $E_3$ only, and a constant $C$ depending on $R$ only, such that for any $\varepsilon$, $0<\varepsilon\leq\varepsilon_0$, we have
\begin{equation}
\|u^{\varepsilon}-u\|_{L^2(B_{R+1}\backslash \overline{B_R})}\leq \omega(\varepsilon)
\end{equation}
and
\begin{equation}
\|u^{\varepsilon}_{\infty}-u_{\infty}\|_{L^{\infty}(\mathbb{S}^{N-1})}\leq C\omega(\varepsilon).
\end{equation}
\end{thm}

\section{Full and partial cloaks in the physical space}\label{sect:4}

In this section, we shall apply Theorem~\ref{thm:vmain1} to the constructions of full and partial invisibility cloaks in the physical space. We shall stick to the terminologies introduced in Section~\ref{sect:general construction} for the general construction of the cloaks. Specifically, we shall let $\{K_1^\varepsilon,K_2^\varepsilon, K_3^\varepsilon, s^\varepsilon, \sigma^\varepsilon,q^\varepsilon,  h^\varepsilon, H^\varepsilon\}$ and $\{\Sigma_1,\Sigma_2, \Sigma_3, s, \widetilde\sigma^\varepsilon, \widetilde q^\varepsilon, h, H\}$ be the scattering objects respectively in the virtual and physical spaces. Let $F_\varepsilon$ be the blow-up transformation from the virtual space to the physical space.

\subsection{Regularized full invisibility cloak}\label{subsect:full}

Suppose $D$ is the closure of a bounded convex domain in $\mathbb{R}^N$ that contains the origin.
Let $\Omega$ be the closure of a bounded Lipschitz domain in $\mathbb{R}^N$ such that $D\subset \stackrel{\circ}{\Omega}$ and $\Omega\backslash D$ is connected.

For any $\varepsilon>0$, let $D_\varepsilon=\{\varepsilon x; x\in D\}$. It is supposed that for any $0<\varepsilon<1$ there exists a (uniformly) bi-Lipschitz and orientation-preserving map,
\begin{equation}\label{eq:F1}
F_\varepsilon^{(1)}: \Omega\backslash D_\varepsilon\rightarrow\Omega\backslash D,\quad F_\varepsilon^{(1)}|_{\partial\Omega}=\mbox{Identity}.
\end{equation}
That is, $F_\varepsilon^{(1)}$ blows up $D_\varepsilon$ to $D$ within $\Omega$.
Let
\begin{equation}\label{eq:F2}
F_\varepsilon^{(2)}(x)=\frac{x}{\varepsilon},\quad x\in D_\varepsilon,
\end{equation}
and
\begin{equation}\label{eq:F}
F_\varepsilon=\begin{cases}
\mbox{Identity}\qquad & \mbox{on\quad $\mathbb{R}^N\backslash \Omega$},\\
F_\varepsilon^{(1)}\qquad & \mbox{on\quad $\Omega\backslash D_\varepsilon$},\\
F_\varepsilon^{(2)}\qquad & \mbox{on\quad $D_\varepsilon$}.
\end{cases}
\end{equation}
Formally, we set
\begin{equation}\label{eq:ppvv}
\Sigma=D_{1/2},\ \ K_\varepsilon=D_{\varepsilon},\ \ K_{\varepsilon/2}=D_{\varepsilon/2}.
\end{equation}
Furthermore, we set
\begin{equation}\label{eq:cl1}
\widetilde\sigma_c^\varepsilon=(F_\varepsilon^{(1)})_*I,\quad \widetilde q_c^\varepsilon=(F_\varepsilon^{(1)})_* 1\quad\mbox{in\ $\Omega\backslash D$},
\end{equation}
and
\begin{equation}\label{eq:ly1}
(D\backslash {D}_{1/2};\widetilde\sigma^\varepsilon_l, \widetilde q_l^\varepsilon)=(F_\varepsilon)_*(K_{\varepsilon}\backslash {K}_{\varepsilon/2}; \sigma_l^\varepsilon, q_l^\varepsilon),
\end{equation}
with
\begin{equation}\label{eq:ly2}
\sigma_l^\varepsilon=c_1(x)\varepsilon^{2},\quad q_l^\varepsilon=(c_2(x)+i c_3(x))\varepsilon^{-N+1}.
\end{equation}
In \eqref{eq:ly2}, $c_1(x)$ is a symmetric-matrix valued measurable function, and $c_2(x), c_3(x)$ are bounded real valued measurable function such that
\begin{equation}\label{eq:ll2}
\lambda_0\|\xi\|^2\leq c_1(x)\xi\cdot\xi\leq \Lambda_0\|\xi\|^2,\quad \lambda_0\leq c_2(x), c_3(x)\leq \Lambda_0\qquad\mbox{for a.e. $x\in K_\varepsilon\backslash K_{\varepsilon/2}$},
\end{equation}
where $\lambda_0$ and $\Lambda_0$ are two positive constants independent of $\varepsilon$. $(D\backslash D_{1/2}; \widetilde\sigma_l^\varepsilon, \widetilde q_l^\varepsilon)$ is the chosen lossy layer for our cloaking scheme. By Lemma~\ref{lem:trans acoustics}, it is straightforward to verify that
\begin{equation}\label{eq:ll3}
\widetilde\sigma_l^\varepsilon(x)=c_1(\varepsilon x)\varepsilon^N,\quad \widetilde q_l^\varepsilon(x)=(c_2(\varepsilon x)+ic_3(\varepsilon x))\varepsilon,\quad x\in D\backslash D_{1/2}.
\end{equation}
We would like to emphasize that \eqref{eq:ll3} is only a particular choice for illustration of our general results in Section~\ref{stabsec}, and there are more choices as long as $(K_\varepsilon\backslash K_{\varepsilon/2}; \sigma_l^\varepsilon, q_l^\varepsilon)$ is such chosen that Assumption 1, b) is satisfied.

Next, we first consider the cloaking of passive objects by assuming $-h+\mbox{div}(H)=0$ in the physical space. Let $(\Sigma_1,\Sigma_2,\Sigma_3, s)$ be an admissible obstacle located inside $D_{1/2}$, and $(D_{1/2}\backslash\bigcup_{l=1}^3\Sigma_l; \widetilde\sigma_a, \widetilde q_a)$ be an arbitrary regular medium. Let $(K_\varepsilon^1, K_\varepsilon^2, K_3^\varepsilon, s^\varepsilon)$ and $(K_{\varepsilon/2}; \sigma_a^\varepsilon, q_a^\varepsilon)$ be the corresponding virtual images in the virtual space. Also, we set $K^\varepsilon=\bigcup_{l=1}^3 K_l^\varepsilon$. Hence, in the physical space
\begin{equation}\label{eq:rc2}
(\mathbb{R}^N; \widetilde{\sigma}^\varepsilon, \widetilde{q}^\varepsilon)=\begin{cases}
I, 1\qquad &\ \mbox{in\quad $\mathbb{R}^N\backslash\Omega$},\\
\widetilde{\sigma}_c^\varepsilon, \widetilde{q}_c^\varepsilon\qquad & \ \mbox{in\quad $\Omega\backslash D$},\\
\widetilde\sigma_l^\varepsilon, \widetilde q_l^\varepsilon\qquad & \ \mbox{in\quad $D\backslash D_{1/2}$},\\
\widetilde\sigma_a, \widetilde q_a\qquad &\ \mbox{in\quad $D_{1/2}$},
\end{cases}
\end{equation}

\begin{prop}\label{prop:full1}
Let $\{\Sigma_1,\Sigma_2,\Sigma_3,s,\widetilde\sigma^\varepsilon, \widetilde q^\varepsilon\}$ and $\{K_\varepsilon^1, K_\varepsilon^2, K_\varepsilon^3, s^\varepsilon, \sigma^\varepsilon, q^\varepsilon\}$ be as described above, and $\widetilde u_\varepsilon$ and $u_\varepsilon$ be the corresponding scattering waves \textnormal{(}cf. \eqref{eq:ps1} and \eqref{eq:vs1}\textnormal{)}. We also let $\widetilde u_\infty^\varepsilon$ and $u_\infty^\varepsilon$ be the scattering amplitudes of, respectively, $\widetilde u_\varepsilon$ and $u_\varepsilon$.

Then there exist $\varepsilon_0>0$ and a function $\omega:(0,\varepsilon_0]\to (0,+\infty]$ with $\lim_{s\to 0^+}\omega(s)=0 $, which is independent of $\widetilde\sigma_a, \widetilde q_a$ and $\Sigma_1,\Sigma_2,\Sigma_3$, $s$, and of the impinging direction $d$, such that for any $\varepsilon<\varepsilon_0$
\begin{equation}\label{eq:full1}
\|\widetilde u_\infty^\varepsilon\|_{L^\infty(\mathbb{S}^{N-1})}=\|u_\infty^\varepsilon\|_{L^\infty(\mathbb{S}^{N-1})}\leq \omega(\varepsilon).
\end{equation}
\end{prop}
\begin{proof}
Let $O$ be the origin. For any $x\in\mathbb{R}^N$, let $\tilde x\in\partial D$ be the point lying on the line passing through $O$ and $x$. Set
\[
\tilde d(x)=\mbox{dist}(\tilde x, O)^{-1} d(x)\quad\mbox{for any $x\in\mathbb{R}^N$}.
\]
By taking $K_0:=\{O\}$, it is easily seen that $K_\varepsilon=D_\varepsilon$ and $K_{\varepsilon/2}=D_{\varepsilon/2}$. Since $K_0$ has zero capacity, we know that the solution to \eqref{uscateq} is given by $u=u^i$. Consequently, $u^s=0$ and $u_\infty=0$. One readily has \eqref{eq:full1} by Theorem~\ref{thm:vmain1}.
\end{proof}

Proposition~\ref{prop:full1} indicates that one would have an approximate full invisibility cloak for the construction \eqref{eq:F1}--\eqref{eq:rc2}. The essential point in Proposition~\ref{prop:full1} is that a single point has zero capacity. We recall that $H^1(D)=H^1(D\backslash K_0)$ for any open set $D$ and any compact $K_0\subset D$ with zero capacity.
We would like to emphasize that by following the same spirit, and using Theorem~\ref{thm:vmain1}, one could have more approximate full invisibility cloaks. For example, in $\mathbb{R}^3$, a line segment is also of zero capacity, and hence one could achieve an approximate full cloak by blowing up a `line-segment-like' region in $\mathbb{R}^3$, namely $K_0$ is a line segment; or by blowing up a finite collection of `point-like' and `line-segment-like' regions. In Section~\ref{sect:ABC}, we shall give more discussion on how to construct an approximate full cloak by blowing up a `line-segment-like' region.

Finally, we consider the cloaking of active contents.

\begin{prop}\label{prop:full2}
Under the same assumptions of Proposition~\textnormal{\ref{prop:full1}}, we further assume that the source/sink term satisfies the following. We let $H=0$ and
$h\in L^2(\mathbb{R}^N\backslash\Sigma)$ be such that $supp(h)\subset D_{1/2}\backslash\Sigma$. Moreover, we require that
\begin{equation}\label{eq:abs cond}
\Im \widetilde{q}_a\geq \lambda_0>0\quad\mbox{on\ $supp(h)$},
\end{equation}
where $\lambda_0$ is a constant.

Then there exist $\varepsilon_0>0$ and a function $\omega:(0,\varepsilon_0]\to (0,+\infty]$ with $\lim_{s\to 0^+}\omega(s)=0 $, which is independent of $\widetilde\sigma_a, \Re\widetilde q_a, \Im\widetilde q_a|_{\{h=0\}}$ and $\Sigma_1,\Sigma_2,\Sigma_3$, $s$,  and of the impinging direction $d$, such that for any $\varepsilon<\varepsilon_0$
\begin{equation}\label{eq:full2}
\|\widetilde u_\infty^\varepsilon\|_{L^\infty(\mathbb{S}^{N-1})}=\|u_\infty^\varepsilon\|_{L^\infty(\mathbb{S}^{N-1})}\leq \omega(\varepsilon).
\end{equation}

\end{prop}

\begin{proof}
It is straightforward to verify that
\begin{equation}\label{eq:abs cond 2}
\int_{K_{\varepsilon/2}\backslash K^\varepsilon} (k^2\Im q_a^\varepsilon)^{-1}|h^\varepsilon|^2=\int_{supp(h)} (k^2\Im \widetilde q_a)^{-1} |h|^2\leq k^{-2}\lambda_0^{-1}\int_{D_{1/2}\backslash\Sigma} |h|^2:=E_3<+\infty.
\end{equation}
Hence, Assumption 2, c) is satisfied. The proposition can be proved following a completely similar argument as that for Proposition~\ref{prop:full1} by using Theorem~\ref{thm:vmain1}.
\end{proof}
\begin{rem}
Proposition~\ref{prop:full2} indicates that in addition to passive mediums, the construction \eqref{eq:F1}--\eqref{eq:rc2} is also capable of nearly cloaking active contents. We need that at the place where the active source is located, the medium must be absorbing. It is recalled that the Helmholtz equation can also be used to describe the electromagnetic phenomena. In such a case, $h$ denotes an electric current density, whereas $\Im \widetilde q_a$ denotes the conductivity. At the place where $\Im \widetilde q_a=0$, one must have $h=0$ since there the medium is non-conducting. Hence, \eqref{eq:abs cond 2} is a reasonable physical condition. Moreover, we would like to emphasize that one can nearly cloak a more general source term of the form $-h+\mbox{div}(H)$ by using Assumption~1, c) by imposing certain physical conditions on $\widetilde\sigma_a$
and $\widetilde q_a$.
\end{rem}

\subsection{Regularized partial invisibility cloak}

Our construction of partial cloaking devices will rely on blowing up `partially' small regions in the virtual space. We first present our study in the virtual space.
Let
\begin{equation}\label{eq:k0 2d}
K_0:=\{-a\leq x_1\leq a\}\times \{x_2=0\}\quad\mbox{in\ \ $\mathbb{R}^2$},
\end{equation}
and
\begin{equation}\label{eq:k0 3d}
K_0:=\{-a\leq x_1\leq a\}\times\{ -b\leq x_2\leq b \}\times\{x_3=0\}\quad\mbox{in\ \ $\mathbb{R}^3$}.
\end{equation}

We note that $\nu=(0,1)$ in 2D and $\nu=(0,0,1)$ in 3D for $K_0$.
Let $0\leq\tau\leq 1$ and define
\begin{equation}\label{eq:unit sphere set2}
\mathcal{N}_\tau:=\{\theta\in\mathbb{S}^{N-1}:\
|\nu\cdot \theta|\leq\tau\}.
\end{equation}


Next, we consider the scattering problem \eqref{uscateq} with $K_0$ given in \eqref{eq:k0 2d} and \eqref{eq:k0 3d}, which is known as the {\it screen problem} in the literature.

\begin{prop}\label{prop:partial virtual}
Let $K_0$ be given in \eqref{eq:k0 2d} and \eqref{eq:k0 3d}, and $\mathcal{N}_\tau$ be given in 
\eqref{eq:unit sphere set2}. Let $u\in H_{loc}^1(\mathbb{R}^N\backslash K_0)$ be the solution to \eqref{uscateq} with $u^i(x)=e^{ikx\cdot d}$. Then there exists a constant $C$, depending only on $a, b$ and $k$, such that
\begin{equation}\label{eq:partial virtual 1}
|u_\infty(\hat{x}, d)|\leq C \tau\quad \mbox{for $\hat{x}\in\mathbb{S}^{N-1}$ and $d\in \mathcal{N}_\tau$},
\end{equation}
and
\begin{equation}\label{eq:partial virtual 2}
|u_\infty(\hat{x}, d)|\leq C \tau\quad \mbox{for $\hat{x}\in\mathcal{N}_\tau$ and $d\in \mathbb{S}_{N-1}$}.
\end{equation}
\end{prop}

\begin{proof}
By the well-posedness of the forward scattering problem, we know
\begin{equation}\label{peq:1}
\|u_\infty(\cdot, d)\|_{L^\infty(\mathbb{S}^{N-1})}\leq C \|\frac{\partial u^i}{\partial \nu}\|_{L^2(K_0)}=Ck|\nu\cdot d|(\mathcal{H}^{N-1}(K_0))^{1/2},
\end{equation}
from which one readily has \eqref{eq:partial virtual 1} by recalling the definition of $\mathcal{N}_{\tau}$. \eqref{eq:partial virtual 2} can be obtained from \eqref{eq:partial virtual 1} by using the following reciprocity relation (see, e.g., \cite{ColKre})
\begin{equation}\label{eq:reciprocity}
u_\infty(-d, -\hat{x})=u_\infty(\hat{x}, d)\quad\mbox{for\ $\hat{x}\in\mathbb{S}^{N-1}$ and $d\in\mathbb{S}^{N-1}$}.
\end{equation}

The proof is completed.
\end{proof}

Now, the construction of a partial cloak shall be based on the use of Theorem~\ref{thm:vmain1} and Proposition~\ref{prop:partial virtual}, similar to the one for the full cloaks in Section~\ref{subsect:full} by following the next three steps. First, one chooses $K_\varepsilon$, an $\varepsilon$-neighborhood of $K_0$, and a blow-up transformation $F_\varepsilon$, and through the push-forwards, one constructs the cloaking layer $(\Omega\backslash D; \widetilde\sigma_c^\varepsilon, \widetilde q_c^\varepsilon)$. Second, according to Assumption 1, b), a compatible lossy layer $(K_\varepsilon\backslash K_{\varepsilon/2}; \sigma_l^\varepsilon, q_l^\varepsilon)$ is chosen in the virtual space, and then by the push-forwards, one would have the corresponding lossy layer in the physical space. Finally, according to Assumption 1 or 2, c), one can determine the admissible media, obstacles, or sources that can be partially cloaked. Next, let us consider a simple 2D example by letting $K_\varepsilon=[-1-\varepsilon,1+\varepsilon]\times [-\varepsilon, \varepsilon]$, which is an $\varepsilon$-neighborhood of $K_0=[-1,1]\times\{0\}$. If a uniformly blow-up as the one in \eqref{eq:F2} is used, then $K_\varepsilon$ will be transformed into $[-1/\varepsilon-1,1/\varepsilon+1]\times[-1,1]$. In such a way, one would have a partial cloaking device of a very large size depending on $\varepsilon^{-1}$. For practical considerations, in the next section, we shall develop an assembled-by-components technique on blowing up a partially small region in constructing a regularized partial cloak of compact size.

\section{ABC geometry and concatenated construction of three specific cloaks}\label{sect:ABC}

In this section, we
shall first discuss the assembled-by-components (ABC) geometry in
the virtual and physical spaces, and then give the blow-up construction of three specific cloaks based on concatenating individual components. We would like to emphasize that we would not appeal for a most general study in this aspect, and instead we shall present our study based on these three specific examples. However, the technique developed could be straightforwardly extended to devising many other cloaks.

\subsection{ABC geometry in the virtual and physical spaces}\label{sect:51}

Define for $x=(x_1,\ldots,x_{N})\in\mathbb{R}^N$,
\begin{equation}\label{eq:weighted infinity norm}
|x|_{w,\infty}:=\max_{l}\{|w_lx_l|\}_{l=1}^N,
\end{equation}
and for $1\leq p<+\infty$,
\begin{equation}\label{eq:weighted p norm}
|x|_{w,p}:=\left(\sum_{l=1}^N |w_l x_l|^p\right)^{1/p},
\end{equation}
where $w=(w_1,\ldots,w_{N})\in\mathbb{R}^N$ with $0<w_l\leq 1$
denotes a weight. For any $r>0$ we define the $(w,l^p)$-ball and
sphere, respectively, as follows
\begin{equation}\label{eq:lp ball}
\Omega_{w,r}^{(p)}:=\{x\in\mathbb{R}^N:\
|x|_{w,p}<r\}\quad\mbox{and}\quad
S_{w,r}^{(p)}:=\partial\Omega_{w,r}^{(p)},\quad 1\le p \le
\infty.
\end{equation}
Two semi-$(w,l^p)$-balls can be defined by
\[
{\Omega_{w,r,k}^{(p),+}}:=\{x\in\Omega_{w,r}^{(p)}:\
x_k>0\}\quad\mbox{and}\quad
{\Omega_{w,r,k}^{(p),-}}:=\{x\in\Omega_{w,r}^{(p)}:\ x_k<0\}.
\]
$\Omega_{w,r}^{(p)}$ with $p=1,2,\infty$ will be the three
{\it base geometries} for our subsequent constructions. Next, we
show that by concatenating the base geometries, one could obtain
more practical geometries in both the virtual and physical spaces.

In the sequel, we denote $
x^{\check{k}}:=(x_1,\ldots,x_{k-1},x_{k+1},\ldots,x_N)\in\mathbb{R}^{N-1}$
where the superscript $\check{k}$ indicates that the $k$-th
component $x_k$ is dropped. Thus $N-1$ dimensional
$(w^{\check{k}},l^p)$-balls and semi-balls can be defined,
respectively, by
\[
{\Omega_{w^{\check{k}},r}^{(p)}}:=\{
|x^{\check{k}}|_{w^{\check{k}},p}<r  \}\quad\mbox{and}\quad
{\Omega_{w^{\check{k}},r,j}^{(p),\pm}}:=\{x^{\check{k}}\in\Omega_{w^{\check{k}},r}^{(p)}:\
x_j  \gtrless 0\},
\]
where $j$ is different from $k$.

\subsubsection{2D ABC example $\mathbf{C}$}

%
%

We start with a simple 2D example. Let $w_0:=[1,1]$ and
\begin{equation}\label{eq:c1}
\mathbf{C}_L={\Omega_{w_0,r,1}^{(p),-}}+[-a,0],\quad
\mathbf{C}_R={\Omega_{w_0,r,1}^{(p'),+}}+[a,0]
\end{equation}
and
\begin{equation}\label{eq:c2}
\mathbf{C}_M=\{-a\leq x_1\leq a\}\times\{-r\leq x_2\leq r\},
\end{equation}
where $r$ and $a$ are positive constants, that we shall specify in
the following. By assembling the three components, we obtain
\begin{equation}\label{eq:c3}
\mathbf{C}_{[w_0,r,a]}=\mathbf{C}_L\cup\mathbf{C}_M\cup\mathbf{C}_R.
\end{equation}

\begin{figure}[htbp]
\center
\includegraphics[width=0.32\textwidth]{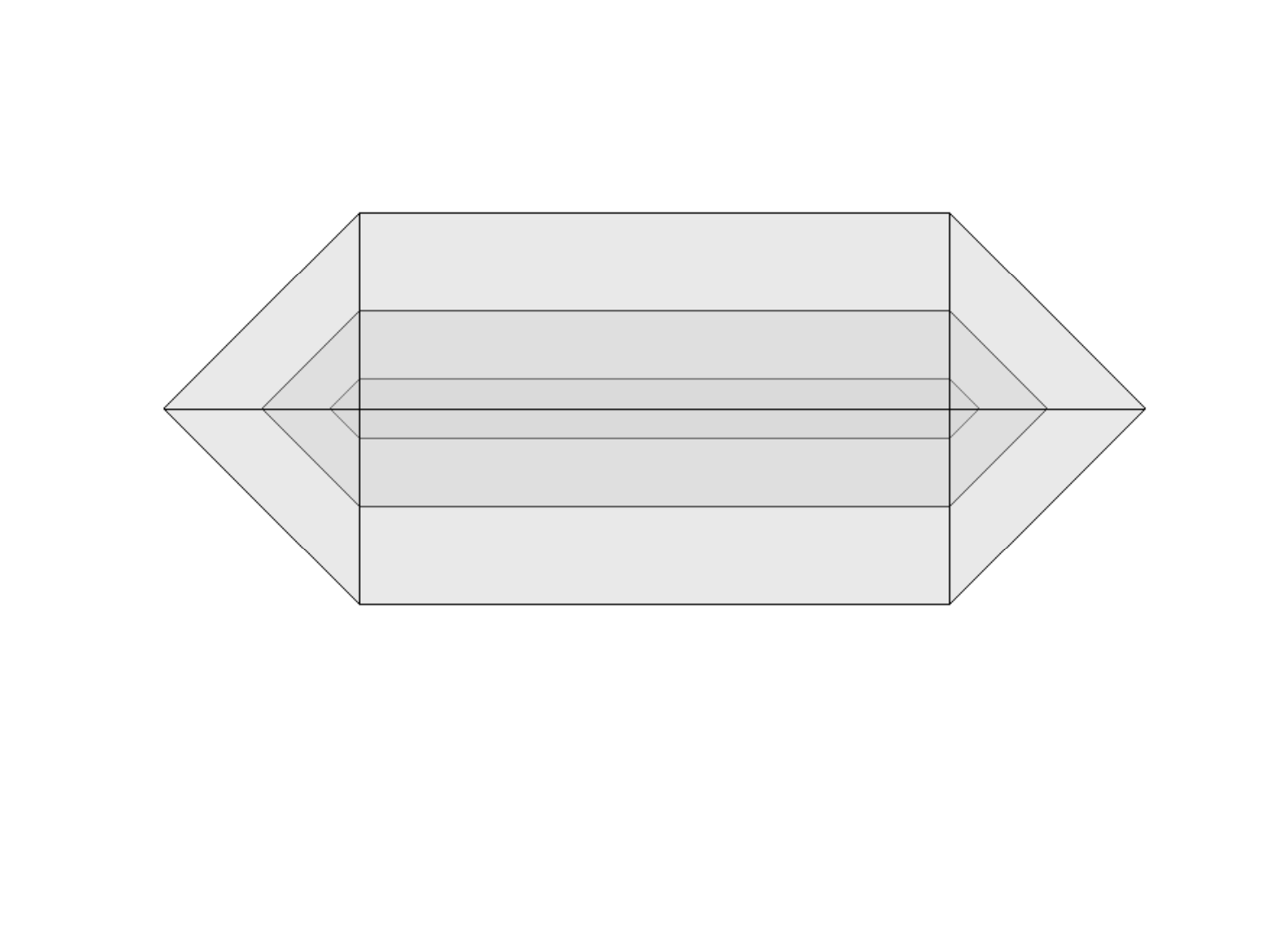}
\includegraphics[width=0.32\textwidth]{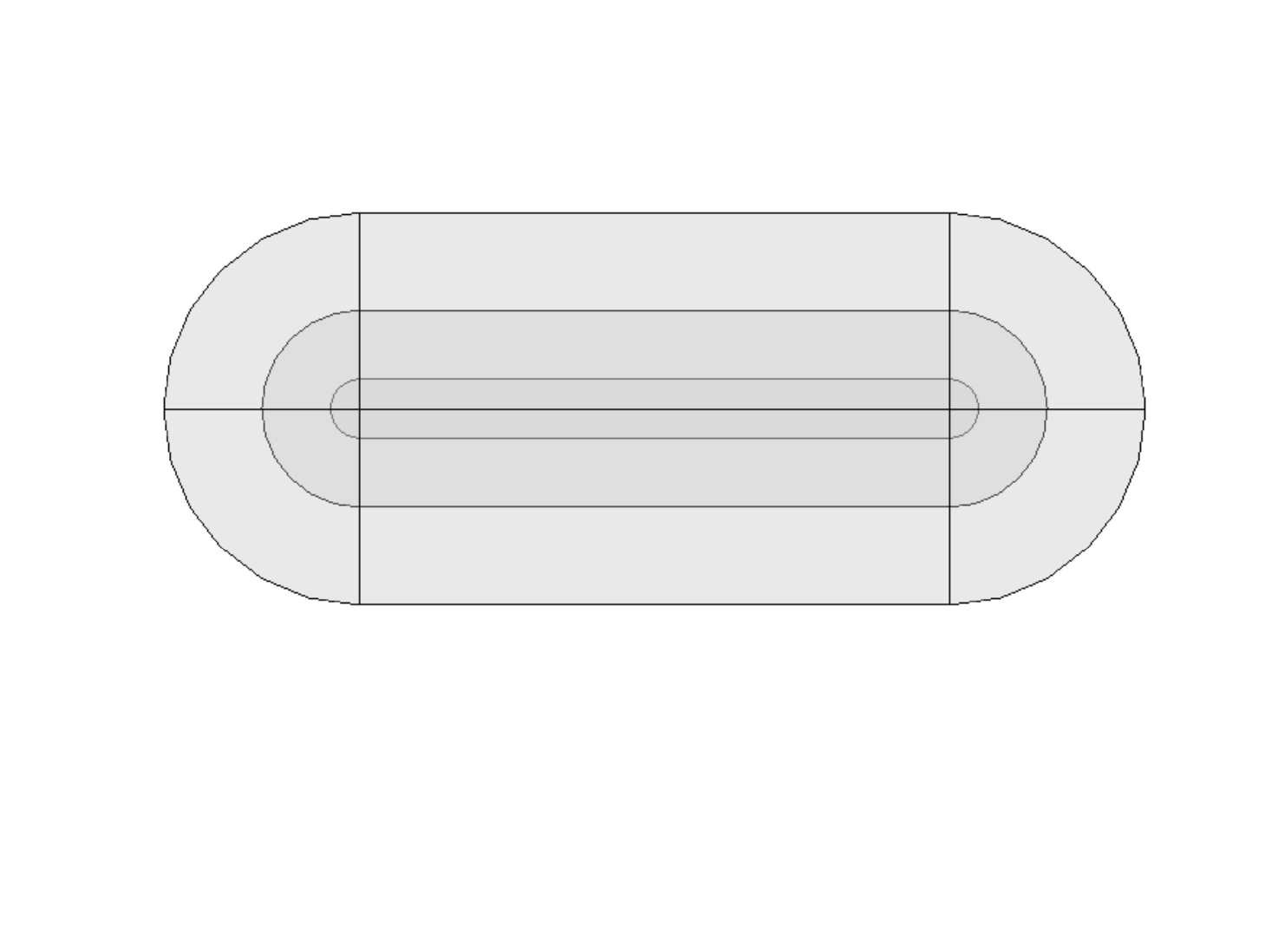}
\includegraphics[width=0.32\textwidth]{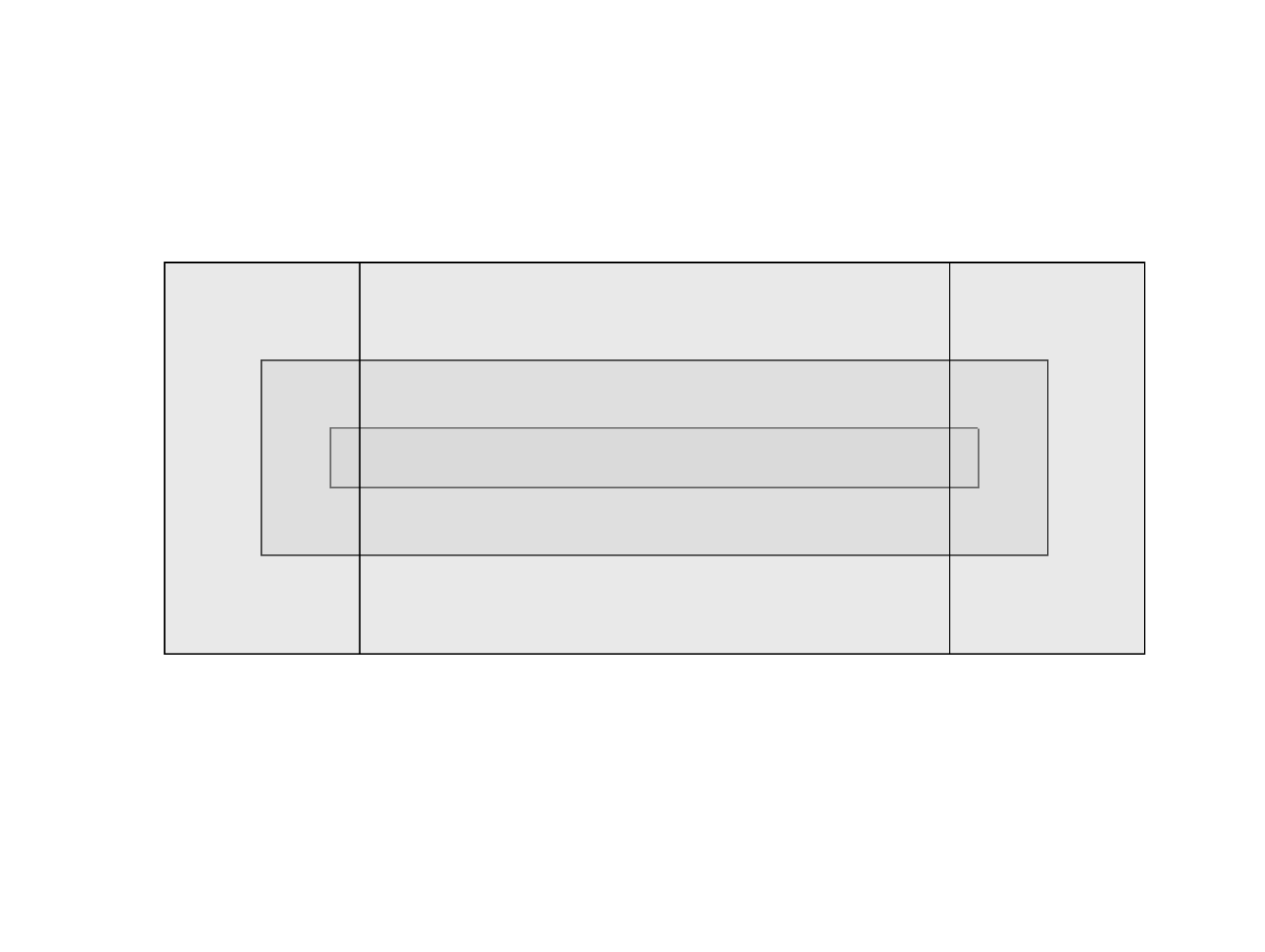}

\hfill{}(a)~~~~~~~~~~~~~\hfill{}(b)\hfill{}~~~~~~~~~~~(c)\hfill{}

\caption{\label{fig:mapping_2d} Two dimensional ABC geometry $\mathbf{C}$ depicted in both the virtual space (between innermost and
outermost boundaries) and the physical space (between intermediate
and outermost boundaries). From left to right:  $l^1$, $l^2$ and
$l^\infty$ cloaks, respectively. }
\end{figure}

If one takes $p=p'=2$, then $\mathbf{C}_{[w_0,r,a]}$ is a 2D {\it
capsule}; if one takes $p=\infty$ and $p'=1$, then
$\mathbf{C}_{[w_0,r,a]}$ gives a 2D {\it nail}; and if
$p=p'=\infty$, then $\mathbf{C}_{[w_0,r,a]}$ is a {\it rectangle};
see Figure~\ref{fig:mapping_2d} for these three constructions.
In our subsequent construction of the
partial cloaking device, we would take $\mathbf{C}_{[w_0,\varepsilon,a]}$
as the `partially' small region in the virtual space, whereas
$\mathbf{C}_{[w_0,r_2,a]}\backslash \mathbf{C}_{[w_0,r_1,a]}$ is the
cloaking region in the physical space. Here $r_2>r_1>\varepsilon$ with
$r_l\sim 1$, $l=1,2$ and $a$ is a free parameter ranges from $\varepsilon$
to $1$ that determines the apertures of the cloaking device. It can
be easily seen that if $a=1$, then $\mathbf{C}_{[w_0,\varepsilon,1]}$ is
`partially' small since it is only small in one dimension and of
regular size in another dimension, whereas if $a=\varepsilon$ then it is
uniformly small since, along both the $x_1$- and $x_2$-dimensions,
the region is small.

%

\subsubsection{3D ABC example $\mathbf{D}$}

We next consider a 3D example. Let $w_0:=[1,1,1]$ and define three
detached components
\begin{equation}\label{eq:d1}
\mathbf{D}_L={\Omega_{w_0,r,1}^{(p),-}}+[-a,0,0],\quad
\mathbf{D}_R={\Omega_{w_0,r,1}^{(p),+}}+[a,0,0]
\end{equation}
and
\begin{equation}\label{eq:d2}
\mathbf{D}_M=\{-a\leq x_1\leq
a\}\times\{x^{\check{1}}\in\overline{\Omega_{w^{\check{1}},r}^{(p)}}\},
\end{equation}
where $r$ and $a$ are positive constants, that we shall specify in
the following. By assembling the three components as shown in
Figure~\ref{fig:D:construction}, we obtain the \emph{slender} cloak
\begin{equation}\label{eq:d3}
\mathbf{D}_{[w,r,a]}=\mathbf{D}_L\cup\mathbf{D}_M\cup\mathbf{D}_R.
\end{equation}

\begin{figure}[htbp]
\hfill{}

\hfill{}\includegraphics[width=0.5\textwidth]{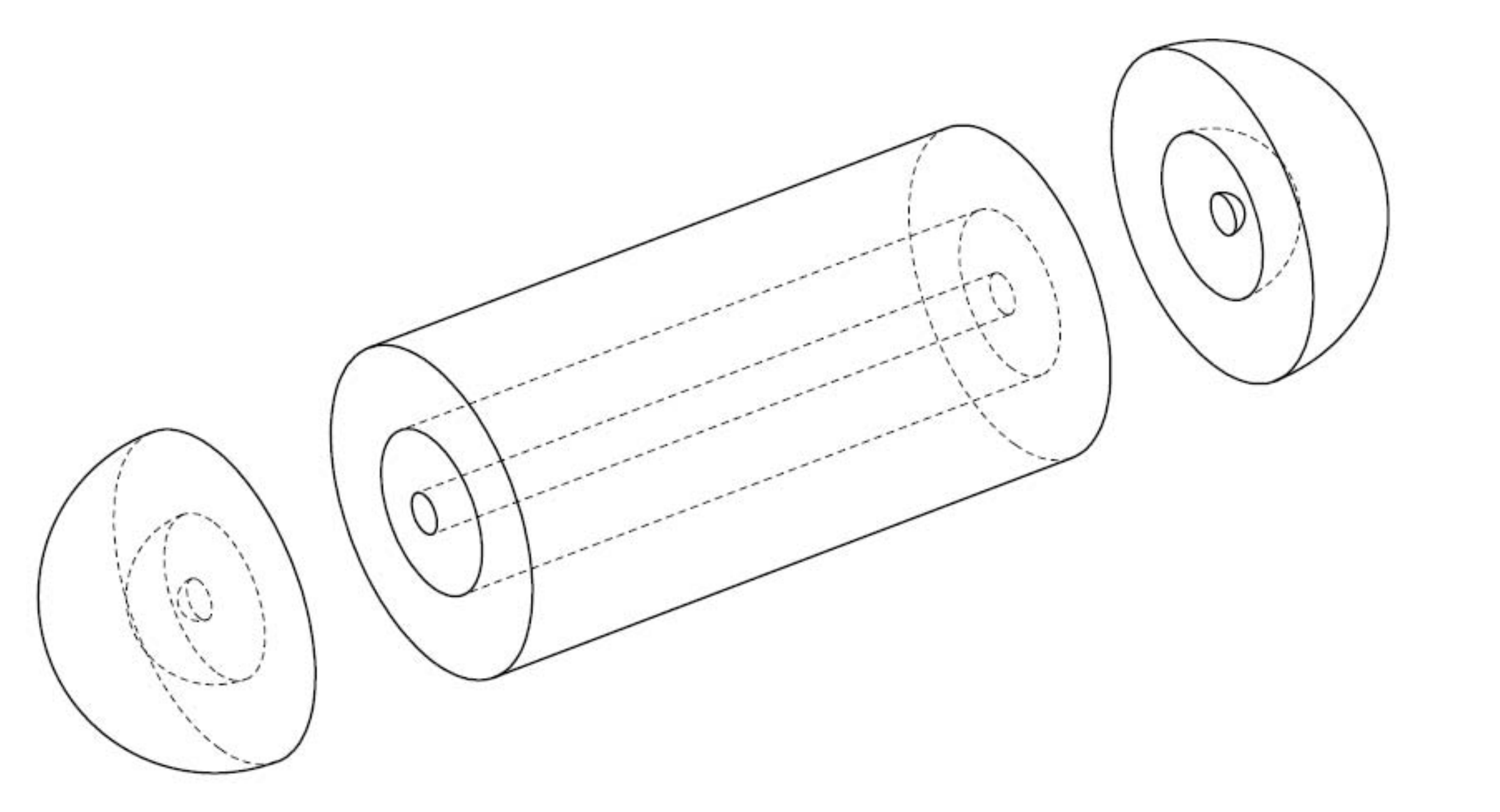}\hfill{}

\hfill{}(a) Detached components in Type $\mathbf{D}$ cloak.\hfill{}

\hfill{}\includegraphics[width=0.5\textwidth]{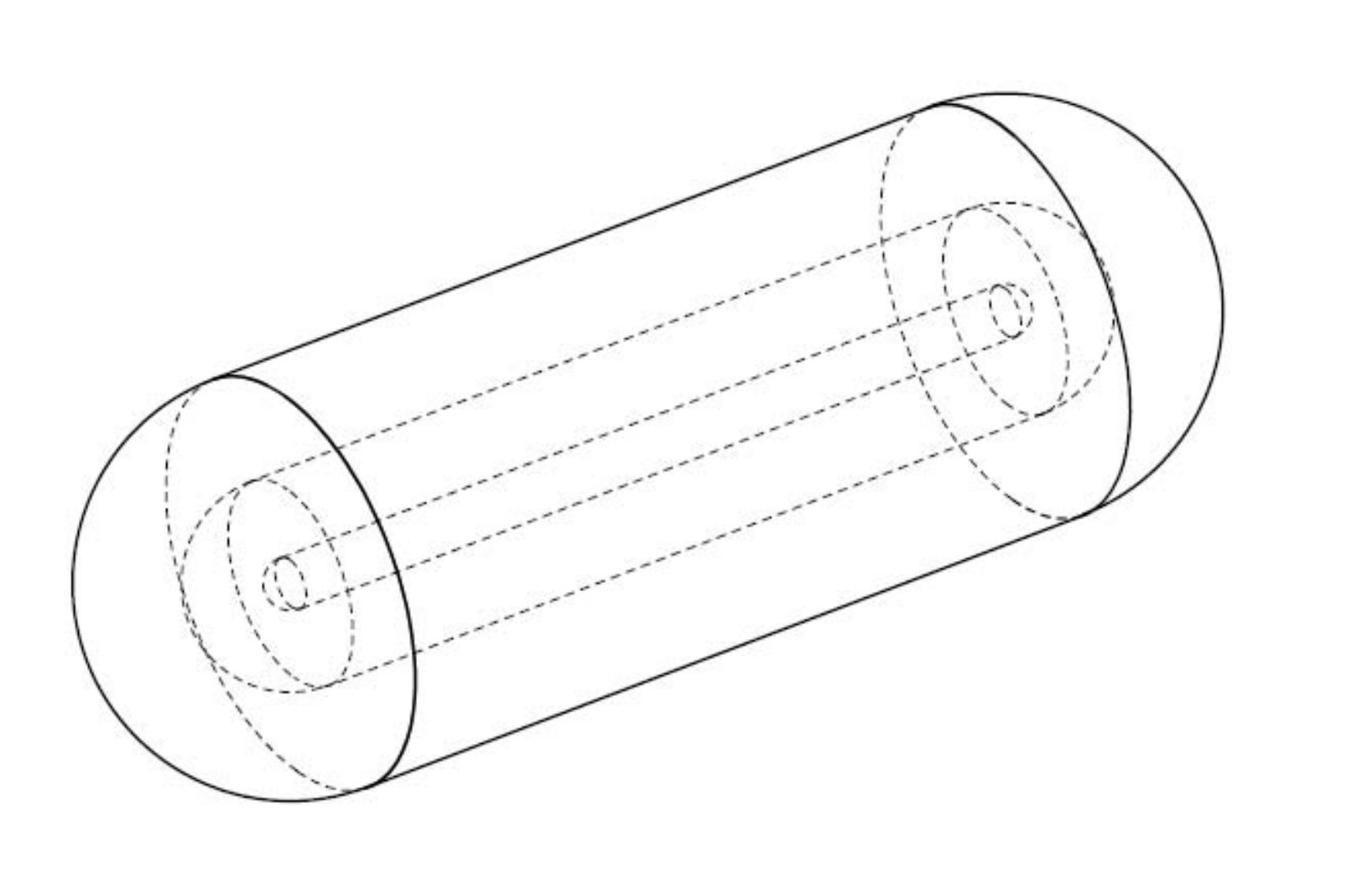}\hfill{}

\hfill{}(b) Slender cloak with assembled components.\hfill{}

\caption{\label{fig:D:construction} Schematic illustration for the construction
of Type $\mathbf{D}$ cloaks.}
\end{figure}

For more concrete examples of this slender cloak construction, let
$w_0=[1,1,1]$, $a=1$. We take $r=\varepsilon$ and $r=1$, respectively, to
represent the geometries in the virtual and physical spaces.
If $p=p'=2$, then
$\mathbf{D}_{[w_0,r,1]}$ is a 3D {\it capsule}; if $p=\infty$ and
$p'=1$, then $\mathbf{D}_{[w_0,r,1]}$ is a 3D {\it nail}; and if
$p=p'=\infty$, then $\mathbf{D}_{[w_0,r,1]}$ is a {\it
rectangular-prism}; see Figure~\ref{fig:mapping_3d} for illustration.
In our subsequent discussion, $\mathbf{D}_{[w_0,\varepsilon,1]}$ will be
the region in the virtual space whereas $\mathbf{D}_{[w_0,1,1]}$
will be the region in the physical space. As can be easily seen, in
the virtual space, $\mathbf{D}_{[w_0,1,\varepsilon]}$ is small along the
$x_2$-, $x_3$-dimensions while it is large along the $x_1$-dimension.
For our subsequent study, we are also interested in the case when
the region in the virtual space is only small in the $x_3$-dimension
while it is large in the $x_1$-, $x_2$-dimensions. It is directly
verified that $\mathbf{D}_{[w_\varepsilon,1,\varepsilon]}$ with
$w_\varepsilon=[1,\varepsilon,1]$ satisfies such requirement. However, it can also
be verified that for such ${w}_\varepsilon$, in the physical space
$\mathbf{D}_{[w_\varepsilon,1,1]}$ is very large along the $x_2$-dimension,
which is actually of size $1/\varepsilon$. In order to construct a more
practical partial cloaking device with its size independent of the
regularization parameter $\varepsilon$, we would like to further develop
the ABC geometry in the following. It is emphasized again that this is our major motivation in developing the ABC geometry.

\begin{figure}[htbp]
\center
\includegraphics[width=0.32\textwidth]{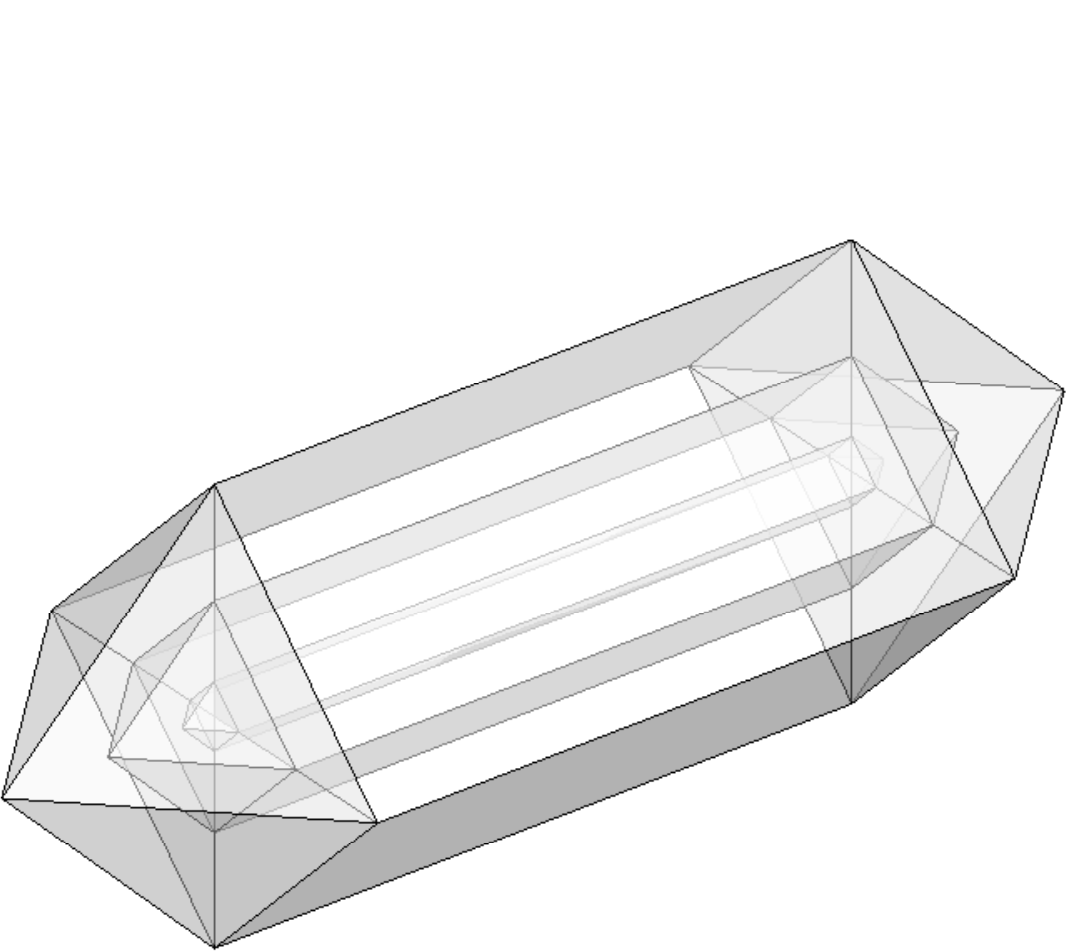}
\includegraphics[width=0.32\textwidth]{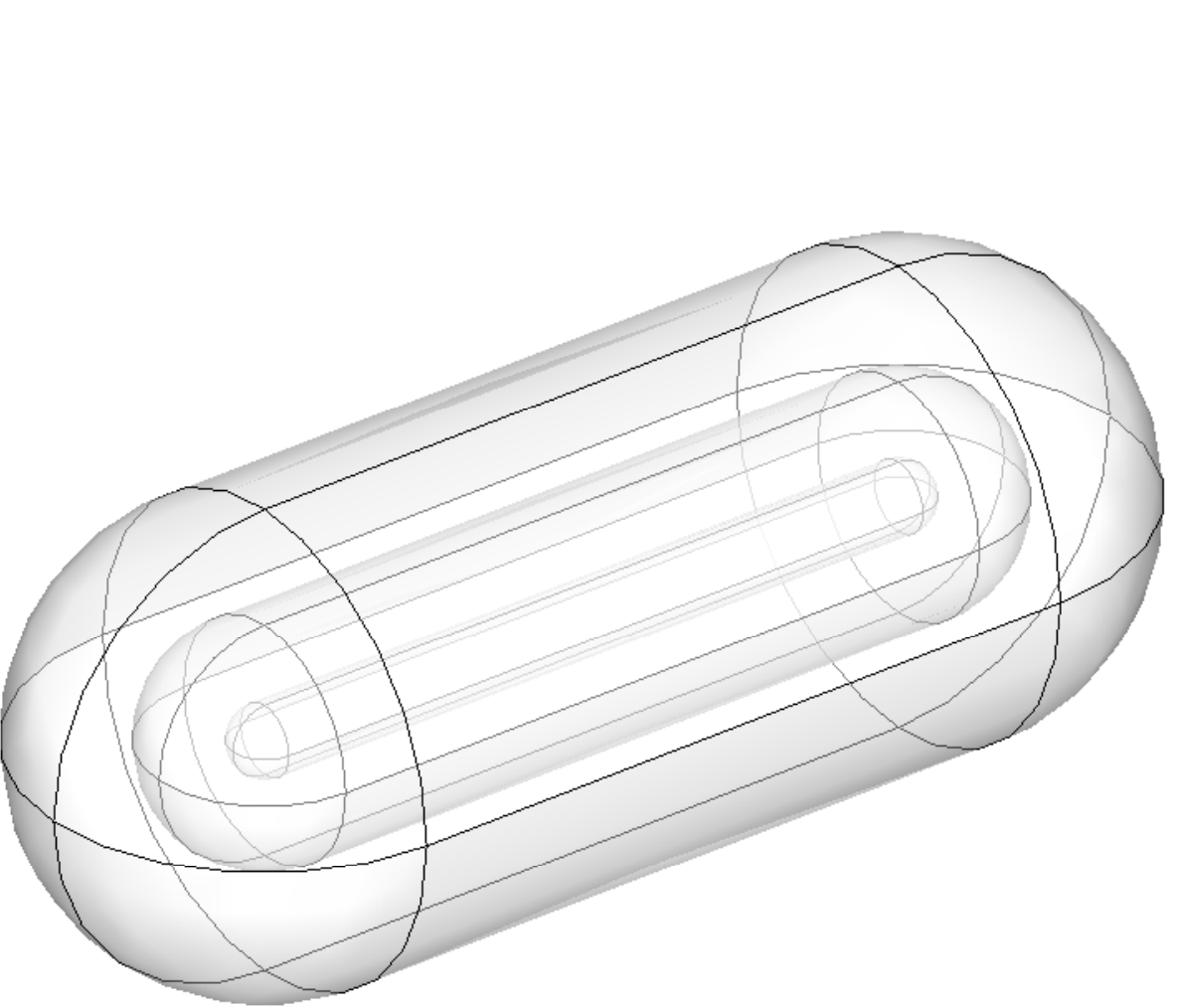}
\includegraphics[width=0.32\textwidth]{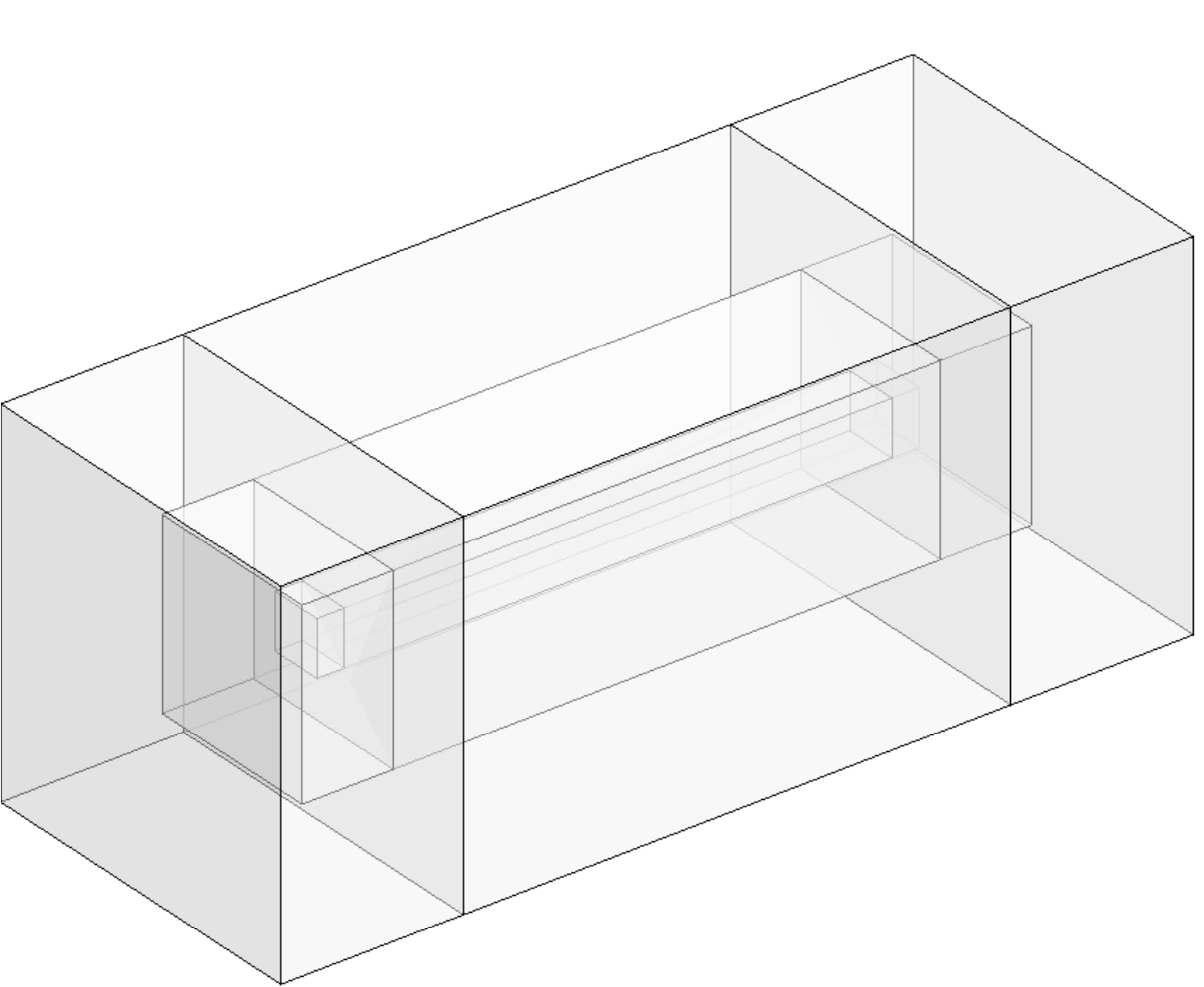}

\caption{\label{fig:mapping_3d} Three dimensional
ABC geometry $\mathbf{D}$ depicted in both the virtual space (between
innermost and outermost boundaries) and the physical space (between
intermediate and outermost boundaries). From left to right:  $l^1$,
$l^2$ and $l^\infty$ cloaks, respectively. }
\end{figure}

\begin{figure}[htbp]
\center

\hfill{}

\hfill{}\includegraphics[width=0.7\textwidth]{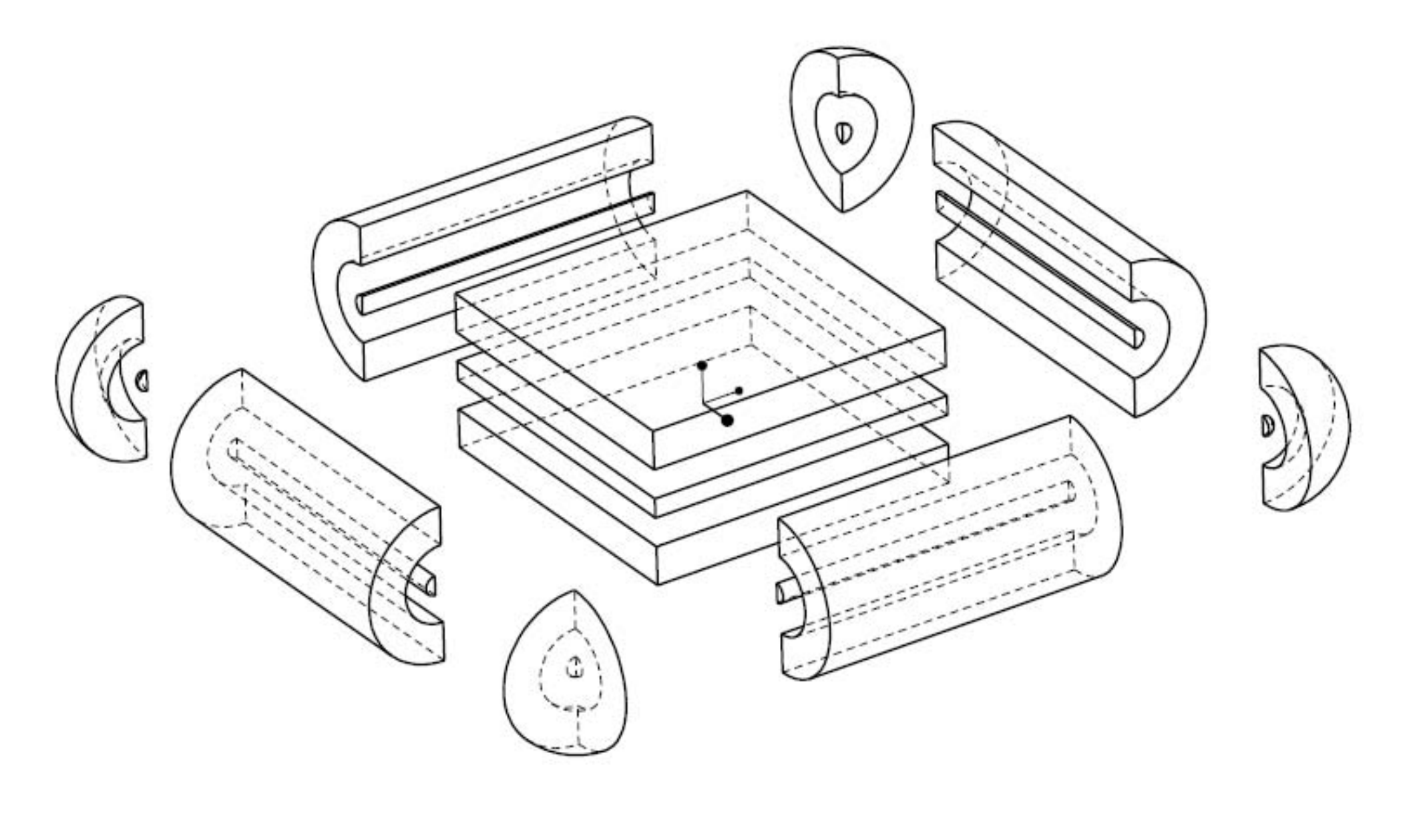}\hfill{}

\hfill{}(a) Detached components in Type $\mathbf{E}$ cloak.\hfill{}

\hfill{}\includegraphics[width=0.7\textwidth]{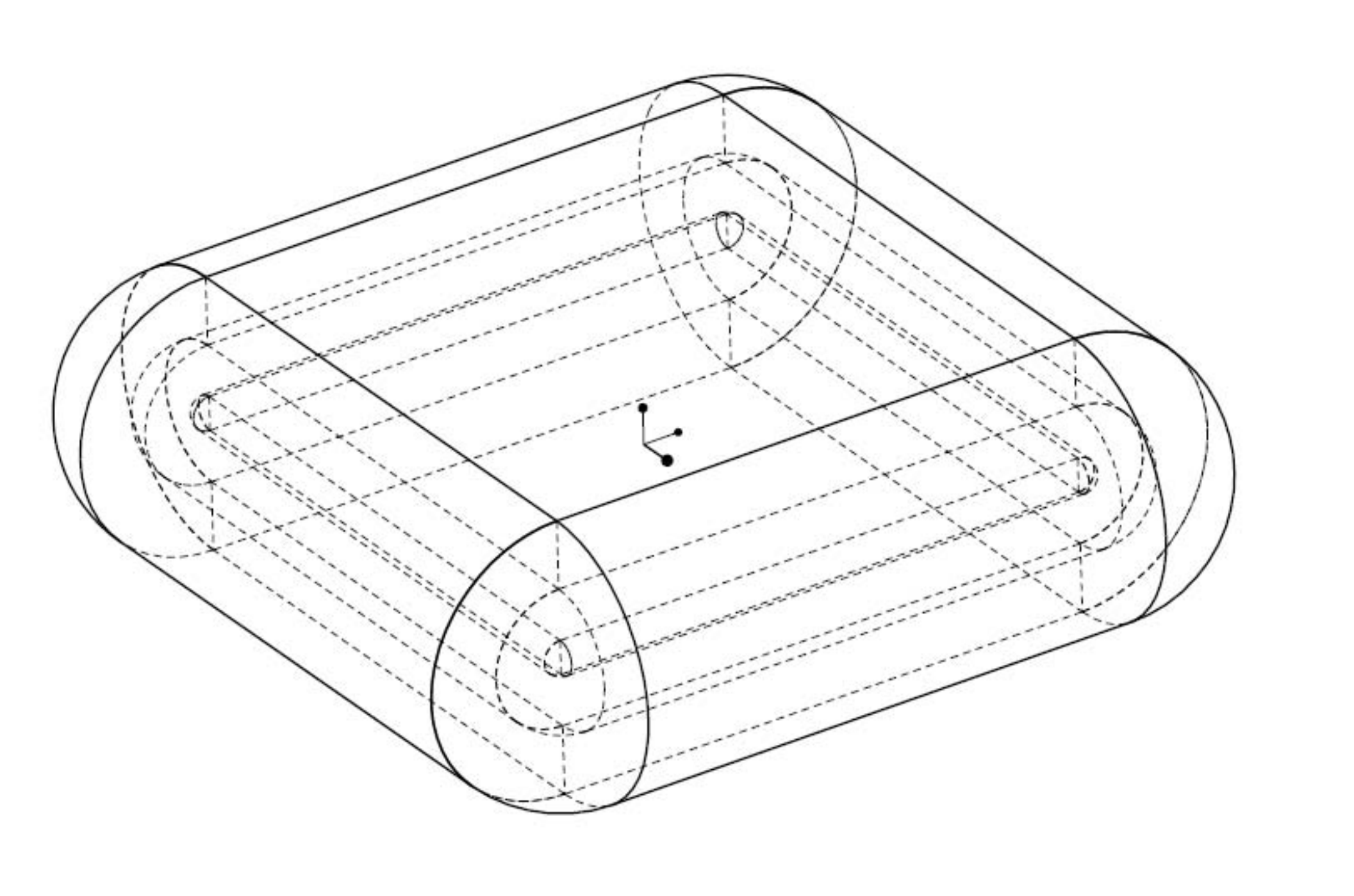}\hfill{}

\hfill{}(b) Flat cloak with assembled components.
\hfill{}

\caption{\label{fig:plane_3d} Schematic illustration for the construction of Type
$\mathbf{E}$ cloaks.}
\end{figure}

%
%

\subsubsection{3D ABC example $\mathbf{E}$}

We next consider a specific example for further developing the ABC
geometry (see Figure~\ref{fig:plane_3d}). Let
\begin{equation}\label{eq:E0}
\mathbf{E}_0=\{-a\leq x_1\leq a\}\times\{-b \leq x_2\leq b\}\times \{-r\leq x_3\leq r\},
\end{equation}
where $r=\varepsilon$ or $r\sim 1$, and $a,b$ are positive constants. For
$x=[x_1,x_2,x_3]\in\mathbb{R}^3$, let  $w:=[1,1,1]$ and
\begin{equation}\label{eq:E12}
\begin{split}
\mathbf{E}_1^{\pm}=&\{x:\ x^{\check{2}}\in\Omega_{w^{\check{2}},r,1}^{(p),\pm}\pm [a,0,0], -b\leq x_2\leq b\},\\
\mathbf{E}_2^{\pm}=&\{x:\
x^{\check{1}}\in\Omega_{w^{\check{1}},r,2}^{(p),\pm}\pm [0,b,0],
-a\leq x_1\leq a\},
\end{split}
\end{equation}
and
\begin{equation}\label{eq:E34}
\begin{split}
\mathbf{E}_3^{\pm}=& \{x+[a,\pm b, 0]:\ x\in\Omega_{w,r}^{(p)}\ \mbox{with}\ x_1>0\ \mbox{and}\ \pm x_2>0\},\\
\mathbf{E}_4^{\pm}=& \{x+[-a, \pm b, 0]:\ x\in \Omega_{w,r}^{(p)}\
\mbox{with}\ x_1<0\ \mbox{and}\ \pm x_2>0\}.
\end{split}
\end{equation}
Let
\begin{equation}\label{eq:E}
\mathbf{E}_{[w_0,r,a,b]}=\mathbf{E}_0\cup\left(\bigcup_{l=1}^4\mathbf{E}_l^{\pm}\right).
\end{equation}
$\mathbf{E}$ is a \emph{rescue cushion} composed of nine components:
the centered component $\mathbf{E}_0$ which is a cuboid of
dimensions $2a, 2b$ and $2r$; four components
$\mathbf{E}_{1,2}^{\pm}$ attached to the side faces and four
components $\mathbf{E}_{3,4}^\pm$ assembled to the four corners. In
our subsequent study, we would take $r=\varepsilon$ as the geometry in the
virtual space while $r\sim 1$ as the geometry in the physical space.

\subsection{Blow-up construction and concatenation}\label{sect:52}
We shall be concerned with two basic types of blow-up transformations. Henceforth, let $r_2>r_1>\varepsilon>0$. Set
\begin{equation}\label{eq:transformation}
F_\varepsilon(x)=\left(A+B|x|_{w,p}\right)\frac{x}{|x|_{w,p}},
\end{equation}
where
\begin{equation}\label{eq:coefficients}
A=\frac{r_1-\varepsilon}{r_2-\varepsilon}r_2,\quad B=\frac{r_2-r_1}{r_2-\varepsilon}.
\end{equation}
It is verified directly that $F_\varepsilon$ blows up $\Omega_{w,\varepsilon}^{(p)}$ to
$\Omega_{w,r_1}^{(p)}$ within $\Omega_{w,r_2}^{(p)}$, namely
\begin{equation}\label{eq:blow up fact}
F_\varepsilon\left(\Omega_{w,r_2}^{(p)}\backslash\overline{\Omega_{w,\varepsilon}^{(p)}}\right)=\Omega_{w,r_2}^{(p)}\backslash\overline{\Omega_{w,r_1}^{(p)}},\quad
F_\varepsilon|_{S_{w,r_2}^{(p)}}=\mbox{Identity}.
\end{equation}
For $A$ and $B$ given in \eqref{eq:coefficients}, we let
\begin{equation}\label{eq:G1}
G_\varepsilon(x)=\begin{cases}
& \displaystyle{(A+B|x_N|)}\ \mbox{sign}(x_N),\\
&\ \ \ \displaystyle{ x_l,\ \ \ 1\leq l\leq N-1}
\end{cases}
\end{equation}
and it can be verified that
\begin{equation}\label{eq:blow up fact 2}
G_\varepsilon\left(Q_{r_2}\backslash Q_{\varepsilon}\right)=Q_{r_2}\backslash Q_{r_1},\quad G_\varepsilon|_{\Pi_{r_2}}=\mbox{Identity},
\end{equation}
where
\begin{align*}
Q_r:=&\{|x_l|\leq a_l:\ 1\leq l\leq N-1\}\times \{|x_N|\leq r\},\\
\Pi_r:=& \{|x_l|\leq a_l:\ 1\leq l\leq N-1\}\times \{|x_N|= r\}.
\end{align*}

Now, by using blow-up transformations of type
\eqref{eq:transformation} or type \eqref{eq:G1}, and concatenating
each sub-component, it is straightforward to see that there exists a
bi-Lipschitz and orientation-preserving map $\mathscr{F}_1$, such that
for the region constructed in \eqref{eq:c3}
\begin{equation}\label{eq:trans F}
\mathscr{F}_1\left(\mathbf{C}_{[w_0,r_2,a]}\backslash \mathbf{C}_{[w_0,\varepsilon,a]}\right)=\mathbf{C}_{[w_0,r_2,a]}\backslash \mathbf{C}_{[w_0,r_1,a]},\quad \mathscr{F}_1|_{\partial\mathbf{C}_{[w_0,r_2,a]}}=\mbox{Identity}.
\end{equation}
In a similar manner, one can show that for $\mathbf{D}$ and $\mathbf{E}$, respectively, in \eqref{eq:d3} and \eqref{eq:E}, there exist $\mathscr{G}_1$ and $\mathscr{H}_1$ such that
\begin{equation}\label{eq:G}
\mathscr{G}_1\left(\mathbf{D}_{[w,r_2,a]}\backslash\mathbf{D}_{[w,\varepsilon,a]}\right)=\mathbf{D}_{[w,r_2,a]}\backslash\mathbf{D}_{[w,r_1,a]},\quad \mathscr{G}_1|_{\partial\mathbf{D}_{[w,r_2,a]}}=\mbox{Identity},
\end{equation}
and
\begin{equation}\label{eq:H}
\mathscr{H}_1\left(\mathbf{E}_{[w_0,r_2,a,b]}\backslash\mathbf{E}_{[w_0,\varepsilon,a,b]}\right)=\mathbf{E}_{[w_0,r_2,a,b]}\backslash\mathbf{E}_{[w_0,r_1,a,b]},\quad \mathscr{H}_1|_{\partial\mathbf{E}_{[w_0,r_2,a,b]}}=\mbox{Identity}.
\end{equation}

\subsection{Partial cloaking devices}\label{sec:cloak:formulae}

With the preparations in Sections~\ref{sect:51} and \ref{sect:52}, we shall present our cloaking construction in this section. We start with a simple 2D case. Let $\mathscr{F}_1$ be given in \eqref{eq:trans F} and set
\begin{equation}
\mathscr{F}(x)=\begin{cases}
\text{Identity} & \mbox{for $x\in \mathbb{R}^2\backslash \overline{\mathbf{C}}_{[w_0, r_2,a]}$},\\
\mathscr{F}_{1}(x)\qquad & \mbox{for\ \ensuremath{x\in\mathbf{C}_{[w_{0},r_{2},a]}\backslash\overline{\mathbf{C}}_{[w_{0},\varepsilon,a]}}},\\
\mathscr{F}_{s}(x)\qquad & \mbox{for\
\ensuremath{x\in\mathbf{C}_{[w_{0},\varepsilon,a]}=\mathbf{C}_{L}\cup\mathbf{C}_{M}\cup\mathbf{C}_{R}}},
\end{cases}\label{eq:transff}
\end{equation}
where
\begin{equation}
\mathscr{F}_{s}\,:\, x=\left(\begin{array}{c}
x_{1}\\
x_{2}
\end{array}\right)\mapsto\begin{cases}
\left(\begin{array}{c}
\frac{r_{1}}{\varepsilon}\left(x_{1}+a\right)-a\\
\frac{r_{1}}{\varepsilon}x_{2}
\end{array}\right)\qquad & \mbox{for\ \ensuremath{x\in\mathbf{C}_{L}}},\\
\left(\begin{array}{c}
x_{1}\\
\frac{r_{1}}{\varepsilon}x_{2}
\end{array}\right) & \mbox{for\ \ensuremath{x\in\mathbf{C}_{M}}},\\
\left(\begin{array}{c}
\frac{r_{1}}{\varepsilon}\left(x_{1}-a\right)+a\\
\frac{r_{1}}{\varepsilon}x_{2}
\end{array}\right) & \mbox{for\ \ensuremath{x\in\mathbf{C}_{R}}}.
\end{cases}\label{eq:transFs}
\end{equation}
 %
Clearly, we have that $\mathscr{F}: \mathbf{C}_{[w_0,r_2,a]}\rightarrow
\mathbf{C}_{[w_0,r_2,a]}$ is bi-Lipschitz and orientation-preserving
and $\mathscr{F}|_{\partial\mathbf{C}_{[w_0,r_2,a]}}$
$=\mbox{Identity}.$ Next, let
\begin{equation}\label{eq:fsh layer}
\left( \mathbf{C}_{[w_0,\varepsilon,a]} \backslash\overline{\mathbf{C}}_{[w_0,\varepsilon/2,a]}; \sigma_l^\varepsilon, q_l^\varepsilon \right), \quad \sigma_l^\varepsilon(x)=\begin{pmatrix} c_1(x) & 0\\
0 & c_2(x) \varepsilon^2
\end{pmatrix}, \ q_l^\varepsilon(x)=(c_3(x)+ic_4(x))\varepsilon^{-1/2},
\end{equation}
where $c_l, 1\leq l\leq 4$ are positive functions satisfying
\[
\lambda_0\leq c_l(x)\leq \Lambda_0\quad \mbox{for a.e. $x\in \mathbf{C}_{[w_0,\varepsilon,a]} \backslash\overline{\mathbf{C}}_{[w_0,\varepsilon/2,a]}$},
\]
with $\lambda_0$ and $\Lambda_0$ two positive constants independent of $\varepsilon$. Then we let
\begin{equation}\label{eq:lossy c}
\left( \mathbf{C}_{[w_0,r_1,a]} \backslash\overline{\mathbf{C}}_{[w_0,r_1/2,a]}; \widetilde\sigma^\varepsilon, \widetilde q^\varepsilon \right)=\mathscr{F}_*\left( \mathbf{C}_{[w_0,\varepsilon,a]} \backslash\overline{\mathbf{C}}_{[w_0,\varepsilon/2,a]}; \sigma_l^\varepsilon, q_l^\varepsilon \right).
\end{equation}
We further let
\begin{equation}\label{eq:p2d1}
\left( \mathbf{C}_{[w_0,r_2,a]}\backslash \overline{\mathbf{C}}_{[w_0,r_1,a]}; \widetilde\sigma^\varepsilon, \widetilde q^\varepsilon\right)=\mathscr{F}_*\left( \mathbf{C}_{[w_0,r_2,a]}\backslash \overline{\mathbf{C}}_{[w_0,\varepsilon,a]}; I, 1\right ).
\end{equation}

\begin{defin}\label{def:51}
Let $\widetilde{\mathcal{C}}^\varepsilon_{\mathbf{C}}=(\Sigma_1,\Sigma_2,\Sigma_3,s,\widetilde{\sigma}^{\varepsilon},\widetilde{q}^{\varepsilon},h,H)$ be an admissible scattering configuration satisfying the following.
\begin{enumerate}[i)]
\item $\Sigma_1,\Sigma_2,\Sigma_3$ are all contained and $h, H$ are both supported in $\mathbf{C}_{[w_0,r_1/2,a]}$.

\item In the regions $\mathbf{C}_{[w_0,r_1,a]} \backslash\overline{\mathbf{C}}_{[w_0,r_1/2,a]}$ and $ \mathbf{C}_{[w_0,r_2,a]}\backslash \overline{\mathbf{C}}_{[w_0,r_1,a]}$, $(\widetilde\sigma^\varepsilon, \widetilde q^\varepsilon)$ are, respectively, given by \eqref{eq:lossy c} and \eqref{eq:p2d1}. Moreover, in $\mathbb{R}^2\backslash\overline{\mathbf{C}}_{[w_0,r_2,a]}$, $\widetilde\sigma^\varepsilon=I$ and $\widetilde q^\varepsilon=1$.
\item If $H\equiv 0$, then it is required that
\begin{equation}\label{eq:cond2d1}
\Im\widetilde{q}^\varepsilon(x)\geq \lambda_0\quad \mbox{for a.e. $x\in supp(h)$},
\end{equation}
where $\lambda_0$ is a positive constant independent of $\varepsilon$.

\item If both $H$ and $h$ are not identically vanishing, then it is required that
\begin{equation}\label{eq:cond2d2}
\lambda_0\leq \Im \widetilde{q}^\varepsilon(x), \Re\widetilde{q}^\varepsilon(x)\leq \Lambda_0\quad \mbox{for a.e. $x\in\mathbf{C}_{[w_0,r_1/2,a]}$},
\end{equation}
and
\begin{equation}\label{eq:cond2d3}
\widetilde{\sigma}^\varepsilon(x)\xi\cdot\xi\geq \lambda_0 \|\xi\|^2\quad\mbox{for a.e. $x\in supp(H)$},
\end{equation}
where $\lambda_0$ and $\Lambda_0$ are positive constants independent of $\varepsilon$.
\end{enumerate}
\end{defin}

Then, we have

\begin{prop}\label{prop:c}
Let $\widetilde{\mathcal{C}}_{\mathbf{C}}^\varepsilon$ be as in Definition~\textnormal{\ref{def:51}}, and let $\widetilde u_\infty^\varepsilon(\hat x, d):=u_\infty(\hat x, d; \widetilde{\mathcal{C}}_{\mathbf{C}}^\varepsilon)$ be the scattering amplitude corresponding to the scattering configuration $\widetilde{\mathcal{C}}_{\mathbf{C}}^\varepsilon$. Let $r_l\sim 1$, $l=1,2$ and $a\sim 1$ in $\mathbf{C}_{[w_0,a,r_l]}$ and let $\mathcal{N}_\tau$ be given in \eqref{eq:unit sphere set2}. Then there exists $\varepsilon_0>0$ and a function $\omega:(0,\varepsilon_0]\rightarrow (0,+\infty]$ with $\lim_{s\rightarrow 0^+}\omega(s)=0$ such that for any $\varepsilon<\varepsilon_0$
\begin{equation}\label{eq:partial c1}
\|\widetilde{u}_\infty^\varepsilon(\cdot,d)\|_{L^\infty(\mathbb{S}^1)}\leq \omega(\varepsilon)+C\tau\quad \mbox{for $d\in \mathcal{N}_\tau$},
\end{equation}
and
\begin{equation}\label{eq:partial c2}
\|\widetilde{u}_\infty^\varepsilon(\cdot,d)\|_{L^\infty(\mathcal{N}_\tau)}\leq \omega(\varepsilon)+C\tau\quad
\mbox{for $d\in \mathbb{S}^1$},
\end{equation}
where $C$ is the constant from Proposition~\textnormal{\ref{prop:partial virtual}} and $\omega$ is independent of $(\Sigma_1,\Sigma_2,\Sigma_3, s)$.
Moreover, if $h\equiv 0$ and $H\equiv 0$, then $\omega$ is also independent of $(\mathbf{C}_{[w_0, r_1/2,a]}; \widetilde\sigma^\varepsilon, \widetilde q^\varepsilon)$\textnormal{;} and if $H\equiv 0$, then $\omega$ is independent of $(\mathbf{C}_{[w_0, r_1/2,a]}; \widetilde\sigma^\varepsilon, \Re\widetilde q^\varepsilon, \Im\widetilde q^\varepsilon|_{\{h=0\}})$\textnormal{;} and if $h\not\equiv 0$ and $H\not\equiv 0$, then $\omega$ is independent of $\widetilde{\sigma}^\varepsilon|_{\{H=0\}}$.
\end{prop}

\begin{proof}
Let $\mathcal{C}^\varepsilon_{\mathbf{C}}=(K_1^{\varepsilon},K_2^{\varepsilon},K_3^{\varepsilon},s^{\varepsilon},\sigma^{\varepsilon},q^{\varepsilon},h^{\varepsilon},H^{\varepsilon})$ be defined by
\[
\mathcal{C}^\varepsilon_{\mathbf{C}}:=(\mathscr{F}^{-1})_*\widetilde{\mathcal{C}}_{\mathbf{C}}^\varepsilon.
\]
That is, $\mathcal{C}^\varepsilon_{\mathbf{C}}$ is the virtual configuration in the virtual space of the physical configuration $\widetilde{\mathcal{C}}_{\mathbf{C}}^\varepsilon$ in the physical space. Then, by Lemma~\ref{lem:trans acoustics}, we know $u_\infty(\hat x, d; \widetilde{\mathcal{C}}_{\mathbf{C}}^\varepsilon)=u_\infty(\hat x, d; {\mathcal{C}}_{\mathbf{C}}^\varepsilon)$. Now, the proposition can be proved by combining Theorem~\ref{thm:vmain1} and Proposition~\ref{prop:partial virtual}, in a completely similar manner as for the proofs of Propositions~\ref{prop:full1} and \ref{prop:full2}.
\end{proof}

\begin{rem}\label{rem:51}
By Proposition~\ref{prop:c}, we see that if $\tau$ and $\varepsilon$ are chosen to be small, the construction corresponding to the scattering configuration $\widetilde{\mathbf{C}}^\varepsilon_{\mathbf{C}}$ yields a regularized partial cloak with a limited observation aperture $\mathcal{N}_\tau$ if the impinging angle is from $\mathbb{S}^1$, and also a regularized partial cloak with a limited impinging aperture $\mathcal{N}_\tau$ and if the observation is from $\mathbb{S}^1$. Moreover, in the cloaked region $\mathbf{C}_{[w_0,r_1/2,a]}$, if no source/sink $(h,H)$ is presented, the passive medium inside could be arbitrary (but regular); and if a source/sink term $h$ is presented, one only need impose the generic condition that at the place where the source/sink is located, the medium is absorbing.
\end{rem}

\begin{rem}\label{rem:52}
The lossy layer \eqref{eq:fsh layer}--\eqref{eq:lossy c} is only a very particular choice of illustrating our general theory. Indeed, one can devise more general lossy layers as long as Assumption~\ref{assumption1}, b) is satisfied. For example, one may choose the lossy layer in \eqref{eq:fsh layer} to be $\sigma_l^\varepsilon=\varepsilon^2 I$ and $q_l^\varepsilon=(1+i)\varepsilon^{-1/2}$. Moreover, with such a lossy layer, one could set the parameter $a$ in defining $\mathbf{C}_{[w_0,r,a]}$ to be $\varepsilon$. In doing this, in the virtual space, $\mathbf{C}_{[w_0,\varepsilon,\varepsilon]}$ is `uniformly' small and hence the construction $\widetilde{\mathcal{C}}^\varepsilon_{\mathbf{C}}$ would yield a full regularized cloak. Furthermore, we would like to note that by adjusting the parameter $a$, one can achieve a customized partial cloak. That is, for any given aperture $\mathcal{N}_\tau$ and any small $\delta>0$, one can choose a sufficiently small $\varepsilon$ and $a\in (\varepsilon, 1)$ such that $\|u^\infty(\cdot, d; \widetilde{\mathcal{C}}^\varepsilon_{\mathbf{C}})\|_{L^\infty(\mathbb{S}^1)}\leq \delta$ for $d\in \mathcal{N}_\tau$. This can be straightforwardly shown by combining Theorem~\ref{thm:vmain1} and Proposition~\ref{prop:partial virtual}. However, there is one point that one needs to take care of in \eqref{peq:1}. In \eqref{peq:1}, the constant $C$ in the estimate is dependent on $a$. However, heuristically speaking, $C$ could be independent of $a$ for $a$ from a certain range. If this is the case, the construction of the customized partial cloak holds, and we shall not explore further in this aspect in the current article.
\end{rem}

For the three-dimensional ABC geometry $\mathbf{D}$, in a similar manner, one lets
\begin{equation}\label{eq:transff d}
\mathscr{G}(x)=\begin{cases}
\text{Identity} & \mbox{for $x\in \mathbb{R}^3\backslash \overline{\mathbf{D}}_{[w_0, r_2,a]}$},\\
\mathscr{G}_1(x)\qquad &\mbox{for\ $x\in \mathbf{D}_{[w_0,r_2,a]}\backslash \overline{\mathbf{D}}_{[w,\varepsilon,a]}$},\\
\mathscr{G}_{s}(x)\qquad & \mbox{for\
\ensuremath{x\in\mathbf{D}_{[w_{0},\varepsilon,a]}=\mathbf{D}_{L}\cup\mathbf{D}_{M}\cup\mathbf{D}_{R}}}.
\end{cases}
\end{equation}
where
\begin{equation}
\mathscr{G}_{s}\,:\, x=\left(\begin{array}{c}
x_{1}\\
x_{2}\\
x_{3}
\end{array}\right)\mapsto\begin{cases}
\left(\begin{array}{c}
\frac{r_{1}}{\varepsilon}\left(x_{1}+a\right)-a\\
\frac{r_{1}}{\varepsilon}x_{2} \\
\frac{r_{1}}{\varepsilon}x_{3}
\end{array}\right)\qquad & \mbox{for\ \ensuremath{x\in\mathbf{D}_{L}}},\\
\left(\begin{array}{c}
x_{1}\\
\frac{r_{1}}{\varepsilon}x_{2}\\
\frac{r_{1}}{\varepsilon}x_{3}
\end{array}\right) & \mbox{for\ \ensuremath{x\in\mathbf{D}_{M}}},\\
\left(\begin{array}{c}
\frac{r_{1}}{\varepsilon}\left(x_{1}-a\right)+a\\
\frac{r_{1}}{\varepsilon}x_{2}\\
\frac{r_{1}}{\varepsilon}x_{3}
\end{array}\right) & \mbox{for\ \ensuremath{x\in\mathbf{D}_{R}}}.
\end{cases}\label{eq:transGs1}
\end{equation}
Then $\mathscr{G}: \mathbf{D}_{[w,r_2,a]}\rightarrow \mathbf{D}_{[w,r_2,a]}$ is bi-Lipschitz and orientation-preserving and
$\mathscr{G}|_{\partial\mathbf{D}_{[w,r_2,a]}}$ $=\mbox{Identity}.$ Next, we again introduce a lossy layer as follows
\begin{equation}\label{eq:fsh layer d}
\left( \mathbf{D}_{[w,\varepsilon,a]} \backslash\overline{\mathbf{D}}_{[w,\varepsilon/2,a]}; \sigma_l^\varepsilon, q_l^\varepsilon \right), \quad \sigma_l^\varepsilon=c_1 \varepsilon^{2}I, \ q_l^\varepsilon=(c_2+ic_3)\varepsilon^{-1/2},
\end{equation}
where $c_l, 1\leq l\leq 3$ are positive constants, and let
\begin{equation}\label{eq:lossy c d}
\left( \mathbf{D}_{[w,r_1,a]} \backslash\overline{\mathbf{D}}_{[w,r_1/2,a]}; \widetilde{\sigma}^\varepsilon, \widetilde{q}^\varepsilon \right)=\mathscr{G}_*\left( \mathbf{D}_{[w,\varepsilon,a]} \backslash\overline{\mathbf{D}}_{[w,\varepsilon/2,a]}; \sigma_l^\varepsilon, q_l^\varepsilon \right).
\end{equation}
We also let
\begin{equation}\label{eq:p2d1 d}
\left( \mathbf{D}_{[w,r_2,a]}\backslash \overline{\mathbf{D}}_{[w,r_1,a]}; \widetilde{\sigma}^\varepsilon, \widetilde{q}^\varepsilon\right)=\mathscr{G}_*\left( \mathbf{D}_{[w,r_2,a]}\backslash \mathbf{D}_{[w,\varepsilon,a]}; I, 1\right ),
\end{equation}

\begin{defin}\label{def:51 d}
Let $\widetilde{\mathcal{C}}^\varepsilon_{\mathbf{D}}=(\Sigma_1,\Sigma_2,\Sigma_3,s,\widetilde{\sigma}^{\varepsilon},\widetilde{q}^{\varepsilon},h,H)$ be an admissible scattering configuration satisfying the following.
\begin{enumerate}[i)]
\item $\Sigma_1,\Sigma_2,\Sigma_3$ are all contained and $h, H$ are both supported in $\mathbf{D}_{[w,r_1/2,a]}$.

\item In the regions $\mathbf{D}_{[w,r_1,a]} \backslash\overline{\mathbf{D}}_{[w,r_1/2,a]}$ and $ \mathbf{D}_{[w,r_2,a]}\backslash \overline{\mathbf{D}}_{[w,r_1,a]}$, $(\widetilde\sigma^\varepsilon, \widetilde q^\varepsilon)$ are, respectively, given by \eqref{eq:lossy c d} and \eqref{eq:p2d1 d}. Moreover, in $\mathbb{R}^3\backslash\overline{\mathbf{D}}_{[w,r_2,a]}$, $\widetilde\sigma^\varepsilon=I$ and $\widetilde q^\varepsilon=1$.
\item If $H\equiv 0$, then it is required that
\begin{equation}\label{eq:cond2d1 d}
\Im\widetilde{q}^\varepsilon(x)\geq \lambda_0\quad \mbox{for a.e. $x\in supp(h)$},
\end{equation}
where $\lambda_0$ is a positive constant independent of $\varepsilon$.

\item If both $H$ and $h$ are not identically vanishing, then it is required that
\begin{equation}\label{eq:cond2d2 d}
\lambda_0\leq \Im \widetilde{q}^\varepsilon(x), \Re\widetilde{q}^\varepsilon(x)\leq \Lambda_0\quad \mbox{for a.e. $x\in\mathbf{D}_{[w,r_1/2,a]}$},
\end{equation}
and
\begin{equation}\label{eq:cond2d3 d}
\widetilde{\sigma}^\varepsilon(x)\xi\cdot\xi\geq \lambda_0 \|\xi\|^2\quad\mbox{for a.e. $x\in supp(H)$},
\end{equation}
where $\lambda_0$ and $\Lambda_0$ are positive constants independent of $\varepsilon$.
\end{enumerate}
\end{defin}

Then, we have

\begin{prop}\label{prop:c d}
Let $\widetilde{\mathcal{C}}_{\mathbf{D}}^\varepsilon$ be as in Definition~\textnormal{\ref{def:51 d}}, and let $\widetilde u_\infty^\varepsilon(\hat x, d):=u_\infty(\hat x, d; \widetilde{\mathcal{C}}_{\mathbf{D}}^\varepsilon)$ be the scattering amplitude corresponding to the scattering configuration $\widetilde{\mathcal{C}}_{\mathbf{D}}^\varepsilon$. Let $r_l\sim 1$, $l=1,2$ and $a\sim 1$ in $\mathbf{D}_{[w,a,r_l]}$. Then there exists $\varepsilon_0>0$ and a function $\omega:(0,\varepsilon_0]\rightarrow (0,+\infty]$ with $\lim_{s\rightarrow 0^+}\omega(s)=0$ such that for any $\varepsilon<\varepsilon_0$
\begin{equation}\label{eq:partial c1 d}
\|\widetilde{u}_\infty^\varepsilon(\cdot,d)\|_{L^\infty(\mathbb{S}^2)}\leq \omega(\varepsilon)\quad \mbox{for $d\in\mathbb{S}^2$}.
\end{equation}
where $\omega$ is independent of $(\Sigma_1,\Sigma_2,\Sigma_3, s)$.
Moreover, if $h\equiv 0$ and $H\equiv 0$, then $\omega$ is also independent of $(\mathbf{D}_{[w, r_1/2,a]}; \widetilde\sigma^\varepsilon, \widetilde q^\varepsilon)$\textnormal{;} and if $H\equiv 0$, then $\omega$ is independent of $(\mathbf{D}_{[w, r_1/2,a]}; \widetilde\sigma^\varepsilon, \Re\widetilde q^\varepsilon, \Im\widetilde q^\varepsilon|_{\{h=0\}})$\textnormal{;} and if $h\not\equiv 0$ and $H\not\equiv 0$, then $\omega$ is independent of $\widetilde{\sigma}^\varepsilon|_{\{H=0\}}$.
\end{prop}

\begin{proof}
The proof can be obtained in a similar manner to that for Proposition~\ref{prop:c}. The major different point is that $\mathbf{D}_{[w,\varepsilon,\varepsilon]}$ in the virtual space degenerates to a line-segment as $\varepsilon\rightarrow 0^+$, which has zero capacity in $\mathbb{R}^3$. Hence, one would have a regularized full cloak for the construction corresponding to $\widetilde{\mathcal{C}}_{\mathbf{D}}^\varepsilon$; see also the proofs of Propositions~\ref{prop:full1} and \ref{prop:full2}.
\end{proof}

\begin{rem}\label{rem:53}
Similar to Remark~\ref{rem:52}, the lossy layer \eqref{eq:fsh layer d}--\eqref{eq:lossy c d} is only a very particular choice for illustrating our general theory. Indeed, one can also devise much more general lossy layers as long as Assumption~\ref{assumption1}, b) is satisfied. For example, one may choose an anisotropic lossy layer with variable coefficients as the one in \eqref{eq:fsh layer}. Moreover, we would like to remark that by adjusting the free parameter $a$ to be smaller in $\mathbf{D}_{[w,r,a]}$, one can improve the accuracy of approximation for the cloak construction corresponding to $\widetilde{\mathcal{C}}_{\mathbf{D}}^\varepsilon$.
\end{rem}

Finally, for the ABC geometry $\mathbf{E}$, one can introduce
\begin{equation}\label{eq:transff e}
\mathscr{H}(x)=\begin{cases}
\text{Identity}\qquad &\mbox{for\ $x\in\mathbb{R}^3\backslash\overline{\mathbf{E}}_{[w_0,r_2,a,b]}$},\\
\mathscr{H}_1(x)\qquad &\mbox{for\ $x\in \mathbf{E}_{[w_0,r_2,a,b]}\backslash \overline{\mathbf{E}}_{[w_0,\varepsilon,a,b]}$},\\
\mathscr{H}_{s}(x) \qquad &\mbox{for\ $x\in
\mathbf{E}_{[w_0,\varepsilon,a,b]}=\mathbf{E}_0\cup\left(\bigcup_{l=1}^4\mathbf{E}_l^{\pm}\right)$},
\end{cases}
\end{equation}
where
\begin{equation}
\mathscr{H}_{s}\,:\, x=\left(\begin{array}{c}
x_{1}\\
x_{2}\\
x_{3}
\end{array}\right)\mapsto\begin{cases}
\left(\begin{array}{c}
x_{1}\\
x_{2}\\
\frac{r_{1}}{\varepsilon}x_{3}
\end{array}\right) & \mbox{for\ \ensuremath{x\in\mathbf{E}_{0}}},\\
\left(\begin{array}{c}
\frac{r_{1}}{\varepsilon}\left(x_{1}\mp a\right)\pm a\\
x_{2} \\
\frac{r_{1}}{\varepsilon}x_{3}
\end{array}\right)\qquad & \mbox{for\ \ensuremath{x\in\mathbf{E}_{1}^{\pm}}},\\
\left(\begin{array}{c}
x_{1}\\
\frac{r_{1}}{\varepsilon}\left(x_{2}\mp b\right)\pm b \\
\frac{r_{1}}{\varepsilon}x_{3}
\end{array}\right)\qquad & \mbox{for\ \ensuremath{x\in\mathbf{E}_{2}^{\pm}}},\\
\left(\begin{array}{c}
\frac{r_{1}}{\varepsilon}\left(x_{1} - a\right) + a\\
\frac{r_{1}}{\varepsilon}\left(x_{2}\mp b\right)\pm b \\
\frac{r_{1}}{\varepsilon}x_{3}
\end{array}\right)\qquad & \mbox{for\ \ensuremath{x\in\mathbf{E}_{3}^{\pm}}},\\
\left(\begin{array}{c}
\frac{r_{1}}{\varepsilon}\left(x_{1} + a\right) - a\\
\frac{r_{1}}{\varepsilon}\left(x_{2}\mp b\right)\pm b \\
\frac{r_{1}}{\varepsilon}x_{3}
\end{array}\right)\qquad & \mbox{for\
\ensuremath{x\in\mathbf{E}_{4}^{\pm}}}.
\end{cases}\label{eq:transGs2}
\end{equation}
Let a particular lossy layer be chosen to be
\begin{equation}\label{eq:lossy c e}
\left( \mathbf{E}_{[w_0,r_1,a,b]} \backslash\overline{\mathbf{E}}_{[w_0,r_1/2,a,b]}; \widetilde{\sigma}^\varepsilon, \widetilde{q}^\varepsilon \right)=\mathscr{H}_*\left( \mathbf{E}_{[w_0,\varepsilon,a,b]} \backslash\overline{\mathbf{E}}_{[w_0,\varepsilon/2,a,b]}; \sigma_l^\varepsilon, q_l^\varepsilon \right),
\end{equation}
with
\begin{equation}\label{eq:fsh layer e}
\left( \mathbf{E}_{[w_0,\varepsilon,a,b]} \backslash\overline{\mathbf{E}}_{[w_0,\varepsilon/2,a,b]}; \sigma_l^\varepsilon, q_l^\varepsilon \right), \quad \sigma_l^\varepsilon=c_1 \varepsilon^{2}I, \ q_l^\varepsilon=(c_2+ic_3)\varepsilon^{-1/2},
\end{equation}
where $c_l, 1\leq l\leq 3$ are positive constants.
Set the cloaking medium to be
\begin{equation}\label{eq:p2d1 e}
\left( \mathbf{E}_{[w_0,r_2,a,b]}\backslash \overline{\mathbf{E}}_{[w_0,r_1,a,b]};\widetilde{\sigma}^\varepsilon, \widetilde{q}^\varepsilon \right)=\mathscr{H}_*\left( \mathbf{E}_{[w_0,r_2,a,b]}\backslash \mathbf{E}_{[w_0,\varepsilon,a,b]}; I, 1\right ).
\end{equation}

\begin{defin}\label{def:51 e}
Let $\widetilde{\mathcal{C}}^\varepsilon_{\mathbf{E}}=(\Sigma_1,\Sigma_2,\Sigma_3,s,\widetilde{\sigma}^{\varepsilon},\widetilde{q}^{\varepsilon},h,H)$ be an admissible scattering configuration satisfying the following.
\begin{enumerate}[i)]
\item $\Sigma_1,\Sigma_2,\Sigma_3$ are all contained and $h, H$ are both supported in $\mathbf{E}_{[w_0,r_1/2,a,b]}$.

\item In the regions $\mathbf{E}_{[w_0,r_1,a,b]} \backslash\overline{\mathbf{E}}_{[w_0,r_1/2,a,b]}$ and $ \mathbf{E}_{[w_0,r_2,a,b]}\backslash \overline{\mathbf{E}}_{[w_0,r_1,a,b]}$, $(\widetilde\sigma^\varepsilon, \widetilde q^\varepsilon)$ are, respectively, given by \eqref{eq:lossy c e} and \eqref{eq:p2d1 e}. Moreover, in $\mathbb{R}^3\backslash\overline{\mathbf{E}}_{[w_0,r_2,a,b]}$, $\widetilde\sigma^\varepsilon=I$ and $\widetilde q^\varepsilon=1$.
\item If $H\equiv 0$, then it is required that
\begin{equation}\label{eq:cond2d1 e}
\Im\widetilde{q}^\varepsilon(x)\geq \lambda_0\quad \mbox{for a.e. $x\in supp(h)$},
\end{equation}
where $\lambda_0$ is a positive constant independent of $\varepsilon$.

\item If both $H$ and $h$ are not identically vanishing, then it is required that
\begin{equation}\label{eq:cond2d2 e}
\lambda_0\leq \Im \widetilde{q}^\varepsilon(x), \Re\widetilde{q}^\varepsilon(x)\leq \Lambda_0\quad \mbox{for a.e. $x\in\mathbf{E}_{[w_0,r_1/2,a,b]}$},
\end{equation}
and
\begin{equation}\label{eq:cond2d3 e}
\widetilde{\sigma}^\varepsilon(x)\xi\cdot\xi\geq \lambda_0 \|\xi\|^2\quad\mbox{for a.e. $x\in supp(H)$},
\end{equation}
where $\lambda_0$ and $\Lambda_0$ are positive constants independent of $\varepsilon$.
\end{enumerate}
\end{defin}

Then, we have

\begin{prop}\label{prop:e}
Let $\widetilde{\mathcal{C}}_{\mathbf{E}}^\varepsilon$ be as in Definition~\textnormal{\ref{def:51 e}}, and let $\widetilde u_\infty^\varepsilon(\hat x, d):=u_\infty(\hat x, d; \widetilde{\mathcal{C}}_{\mathbf{E}}^\varepsilon)$ be the scattering amplitude corresponding to the scattering configuration $\widetilde{\mathcal{C}}_{\mathbf{E}}^\varepsilon$. Let $r_l\sim 1$, $l=1,2$ and $a, b\sim 1$ in $\mathbf{E}_{[w_0,,r_l,a,b]}$ and let $\mathcal{N}_\tau$ be given in \eqref{eq:unit sphere set2}. Then there exists $\varepsilon_0>0$ and a function $\omega:(0,\varepsilon_0]\rightarrow (0,+\infty]$ with $\lim_{s\rightarrow 0^+}\omega(s)=0$ such that for any $\varepsilon<\varepsilon_0$
\begin{equation}\label{eq:partial e1}
\|\widetilde{u}_\infty^\varepsilon(\cdot,d)\|_{L^\infty(\mathbb{S}^2)}\leq \omega(\varepsilon)+C\tau\quad \mbox{for $d\in\mathcal{N}_{\tau}$},
\end{equation}
and
\begin{equation}\label{eq:partial e2}
\|\widetilde{u}_\infty^\varepsilon(\cdot,d)\|_{L^\infty(\mathcal{N}_\tau)}\leq \omega(\varepsilon)+C\tau\quad \mbox{for $d\in\mathbb{S}^2$},
\end{equation}
where $C$ is the constant from Proposition~\textnormal{\ref{prop:partial virtual}} and $\omega$ is independent of $(\Sigma_1,\Sigma_2,\Sigma_3, s)$.
Moreover, if $h\equiv 0$ and $H\equiv 0$, then $\omega$ is also independent of $(\mathbf{E}_{[w_0, r_1/2,a,b]}; \widetilde\sigma^\varepsilon, \widetilde q^\varepsilon)$\textnormal{;} and if $H\equiv 0$, then $\omega$ is independent of $(\mathbf{E}_{[w_0, r_1/2,a,b]}; \widetilde\sigma^\varepsilon, \Re\widetilde q^\varepsilon, \Im\widetilde q^\varepsilon|_{\{h=0\}})$\textnormal{;} and if $h\not\equiv 0$ and $H\not\equiv 0$, then $\omega$ is independent of $\widetilde{\sigma}^\varepsilon|_{\{H=0\}}$.
\end{prop}

\begin{proof}
The proof follows in a similar manner to that for Proposition~\ref{prop:c}.
\end{proof}

\begin{rem}
Similar to Remark~\ref{rem:52}, the lossy layer \eqref{eq:lossy c e}--\eqref{eq:fsh layer e} is only a very particular choice for illustrating our general theory. One can also devise much more general lossy layers as long as Assumption~\ref{assumption1}, b) is satisfied and it could be anisotropic with variable coefficients as the one in \eqref{eq:fsh layer}. Also analogously to Remark~\ref{rem:52}, by adjusting the two parameters $a$ and $b$, one may have a customized partial cloak by the construction corresponding to $\widetilde{\mathcal{C}}_{\mathbf{E}}^\varepsilon$.
\end{rem}

\section{Numerical results and discussion}\label{sect:numerical}

In this section, we carry out systematic experiments based on the
discussions in Sections~\ref{sect:4} and
\ref{sect:ABC} to investigate the cloaking performance
with respect to the regularization parameter $\varepsilon$ in two and three
dimensions. Numerical simulations are carried out by employing the
C++/Matlab coupled programming and the pardiso solver. Finite
element method is used in a truncated domain coated by a perfectly
matched layer to evaluate the performance of the proposed cloaks.
The far-field pattern of the scattered field is calculated via the
Kirchhoff-Helmholtz formula by surface integral on a sphere, within
which the cloaking medium is compactly imbedded.

In the following experiments, we use the
${L^\infty(\mathbb{S}^{N-1})}$-norm of the far-field patterns
$\mathcal{A}_{\mathbf{T}}(\cdot, d):=u_\infty(\cdot,d;\widetilde{\mathcal{C}}^\varepsilon_{\mathbf{T}})$ ($\mathbf{T}=\mathbf{C}$,
$\mathbf{D}$, or $\mathbf{E}$) to evaluate the cloaking
performance of the specific cloaks. More precisely,  we compute the
module of the far field patterns on $100$ equidistant points along
the unit circle in two dimensions, or $100\times 100$ points on the
unit sphere by choosing uniformly distributed polar angles and
azimuth angles in three dimensions, respectively. The exact value of the norm of
$u_\infty(\cdot,d;\widetilde{\mathcal{C}}^\varepsilon_{\mathbf{T}})$
is then approximated by the maximum value of these modules, which is
adopted for the performance evaluation of cloaks.

\subsection{2D cloaks}

For acoustic scattering in two dimensions, let the incident wave be
a plane wave of wavelength $l_e = 2$. We choose $a=1\gg\varepsilon$ to
test partial cloaks and $a=\varepsilon$ to check full cloaks.

First of all, we study the scattering performance of Type
$\mathbf{C}$ cloaks. Let us visualize the effect of the $l^{1}$,
$l^{2}$ and $l^{\infty}$ cloaks of Type $\mathbf{C}$ by looking at
the instantaneous pressure. As shown by
Figure~\ref{fig:2d:cloak:instantaneous:pressure},  the total
acoustic pressure is bent and/or compressed in the cloaking medium.
There is clearly a thin surface layer within the lossy layer between the cloaked and cloaking regions
and prevents the incident wave from penetrating into the cloaked
region when the incident direction is parallel to the axial direction
of Type $\mathbf{C}$ cloaks. At the mean time, the total pressure
outside the cloaking media deviates only slightly from that of a
perfectly plane wave (except tiny numerical deviation due to the
finite spatial resolution in numerical simulations).
\begin{figure}[htbp]
\hfill{}\includegraphics[width=0.45\textwidth]{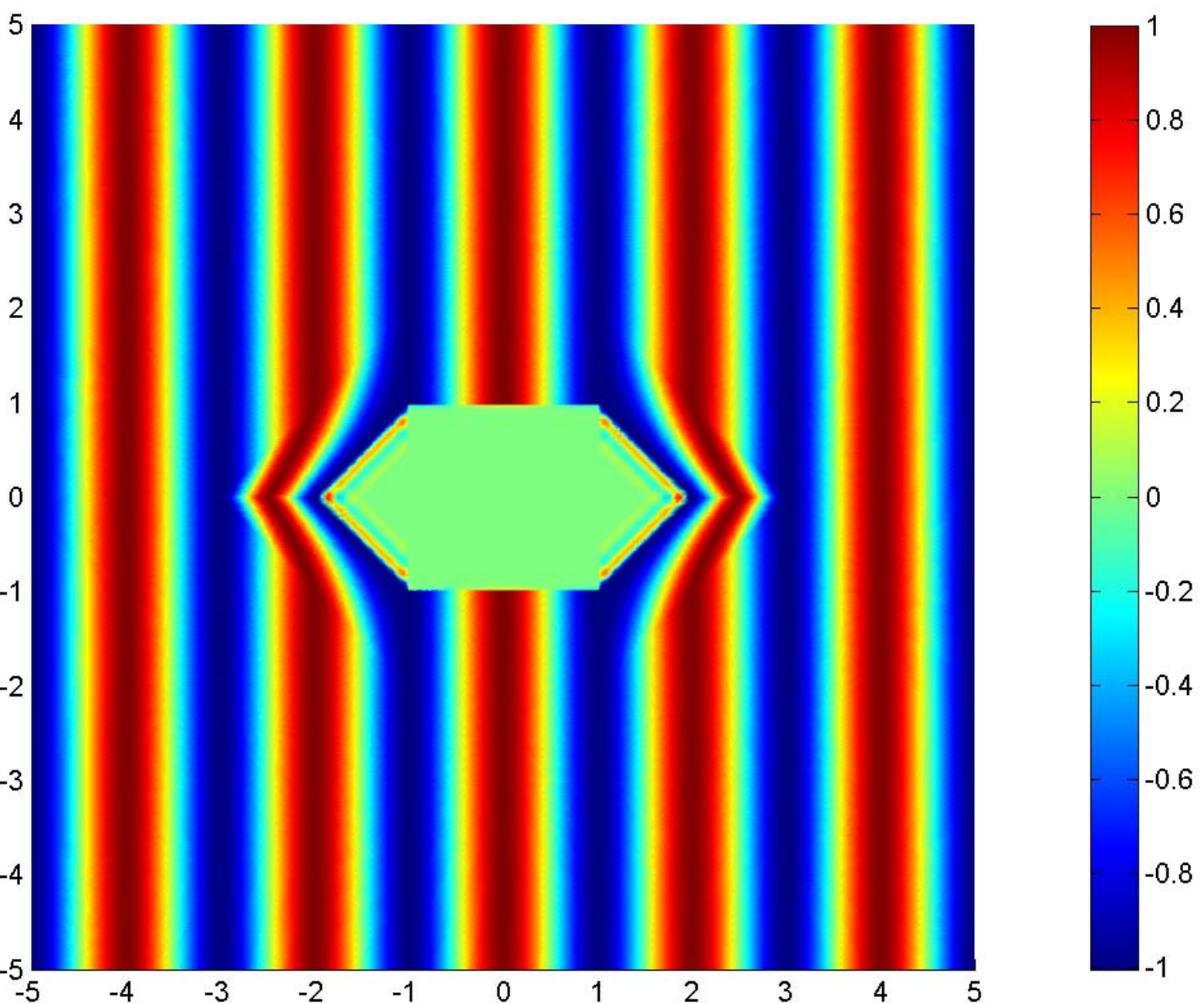}\hfill{}

\hfill{}(a)\hfill{}

\hfill{}\includegraphics[width=0.45\textwidth]{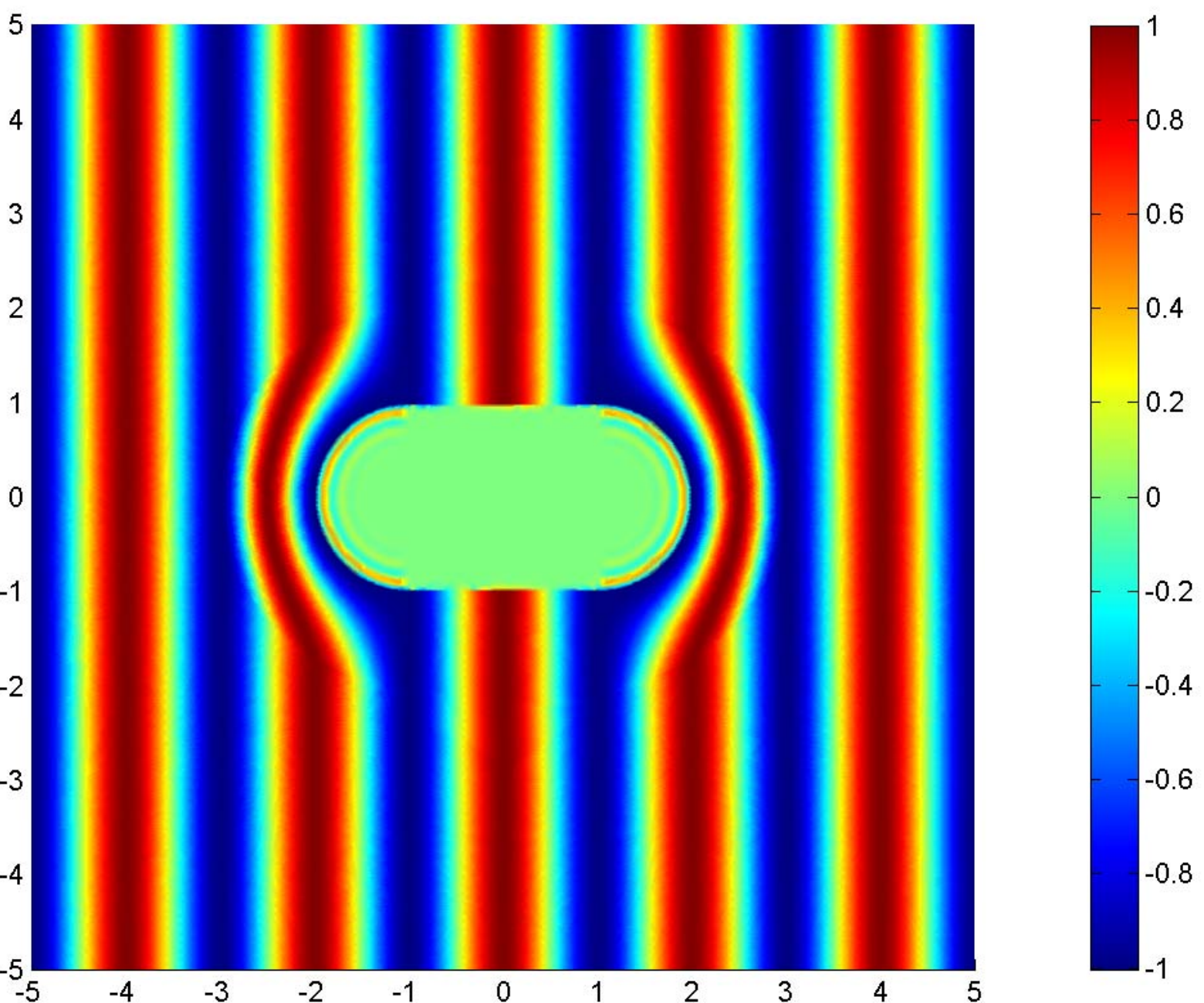}\hfill{}
\includegraphics[width=0.45\textwidth]{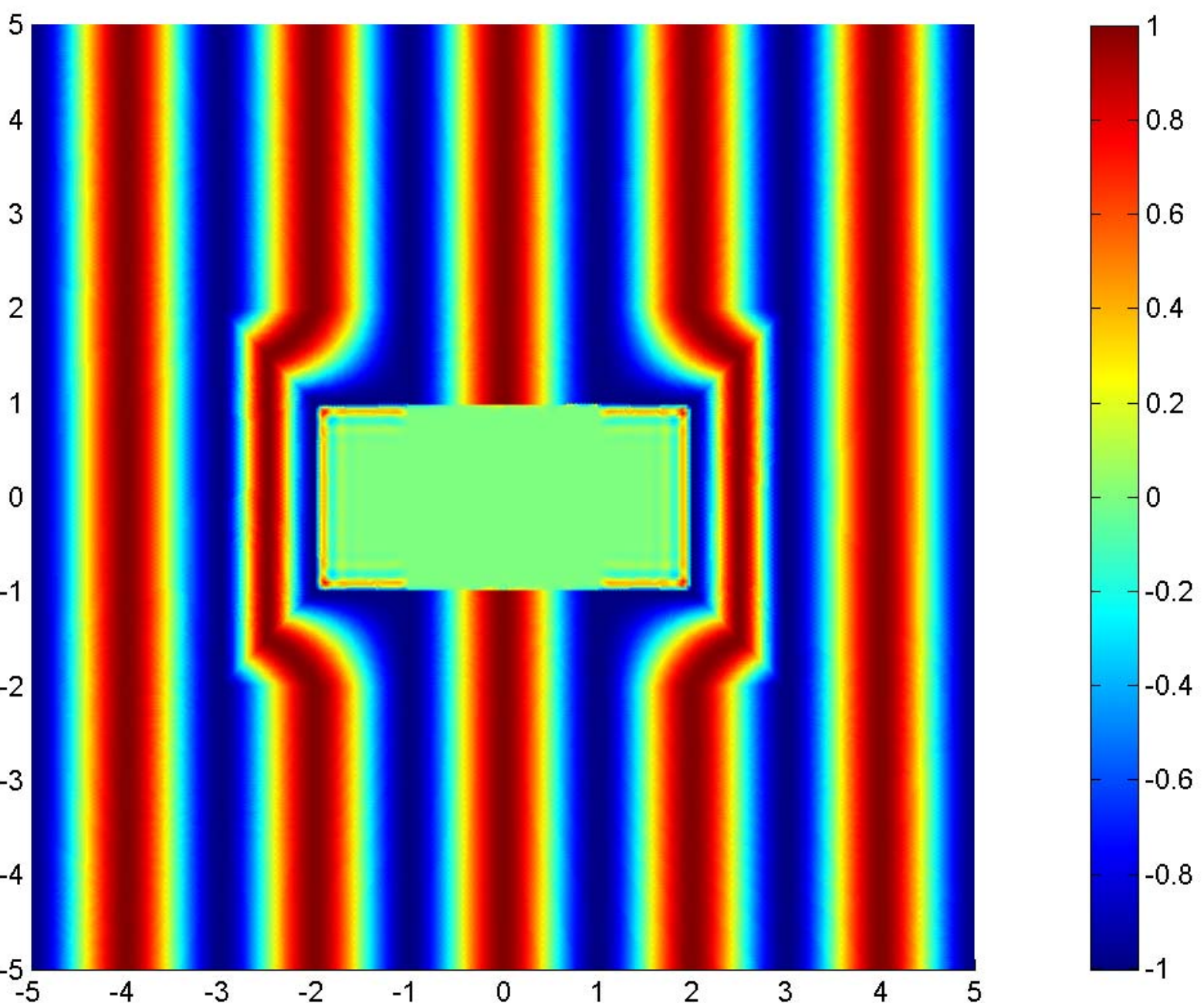}\hfill{}

\hfill{}~~~~~~~~~~(b)~~~~~~~~~~~~~~~~~~~~~~~~~~~~~~~~~~~~~\hfill{}(c)~~~~~~~~~~~~~~\hfill{}

\caption{\label{fig:2d:cloak:instantaneous:pressure} Real part of
instantaneous acoustic pressure distribution using two-dimensional
(a) $l^{1}$, (b) $l^{2}$  and (c) $l^{\infty}$  cloaks,
respectively. Incident plane wave from left to right.
Parameters: $\varepsilon=10^{-2}$, $r_{1}=1$, $r_{2}=2$,
$a=1$.  }
\end{figure}

Next, we test the performance of the 2D $l^2$ partial cloak in case of
tilted incident angles. From the instantaneous pressure in
Figure~\ref{fig:2d:partial:cloak}(a), (b) and (c), it can be
observed that there is increasingly more  pressure deviation of the
total field from that of a perfectly plane wave outside the cloak
when we increase the incident angle from $5$ and $10$ to $20$
degrees. When the incident angle is perpendicular to the axial
direction of the partial cloak, strong scattering amplitude occurs
and this suggests the failure at the extreme case, see
Figure~\ref{fig:2d:partial:cloak}(d).

By symmetry, the same observation can be explained from the other side of the partial cloak.
Figure~\ref{fig:2d:partial:cloak:far:field} shows the far field
pattern in decimal Bell {[}dB{]} with respect to the polar angle. It
can be seen that when the incident direction is along the ideal axis
(x-axis, namely 0 or 180 degrees in polar angle),
the far field is always below $-40$ dB%
\footnote{$-40$ dB amounts to $10^{-4}$ in magnitude in far field
patterns as opposed to the unit magnitude in the incident wave%
}, which is small and negligible in practice.
Figure~\ref{fig:2d:partial:cloak}(a) suggest that for a small tilted incident angle below
$5$ degrees, approximate partial cloak effects
can be achieved with the far-field pattern below $-20$ dB.

In addition, there does exist some angle $\tau\sim15$, $7$ and $3$
degrees corresponding to the tilted incident angles $5$, $10$ and
$20$, respectively, such that with such limited observation aperture perturbed off the
ideal axis, the far field patterns are still below $-40$ dB
(cf.~cusps near the $x$-axis within the $-40$ dB contour curve in
Figure~\ref{fig:2d:partial:cloak:far:field}). Hence \eqref{eq:partial c2} of
Proposition~\ref{prop:c} is confirmed numerically. That is, the
partial cloaking is still in effect in a limited observation aperture around the
ideal direction parallel to the cloak even for full impinging
aperture.

\begin{figure}[htbp]

\hfill{}\includegraphics[width=0.45\textwidth]{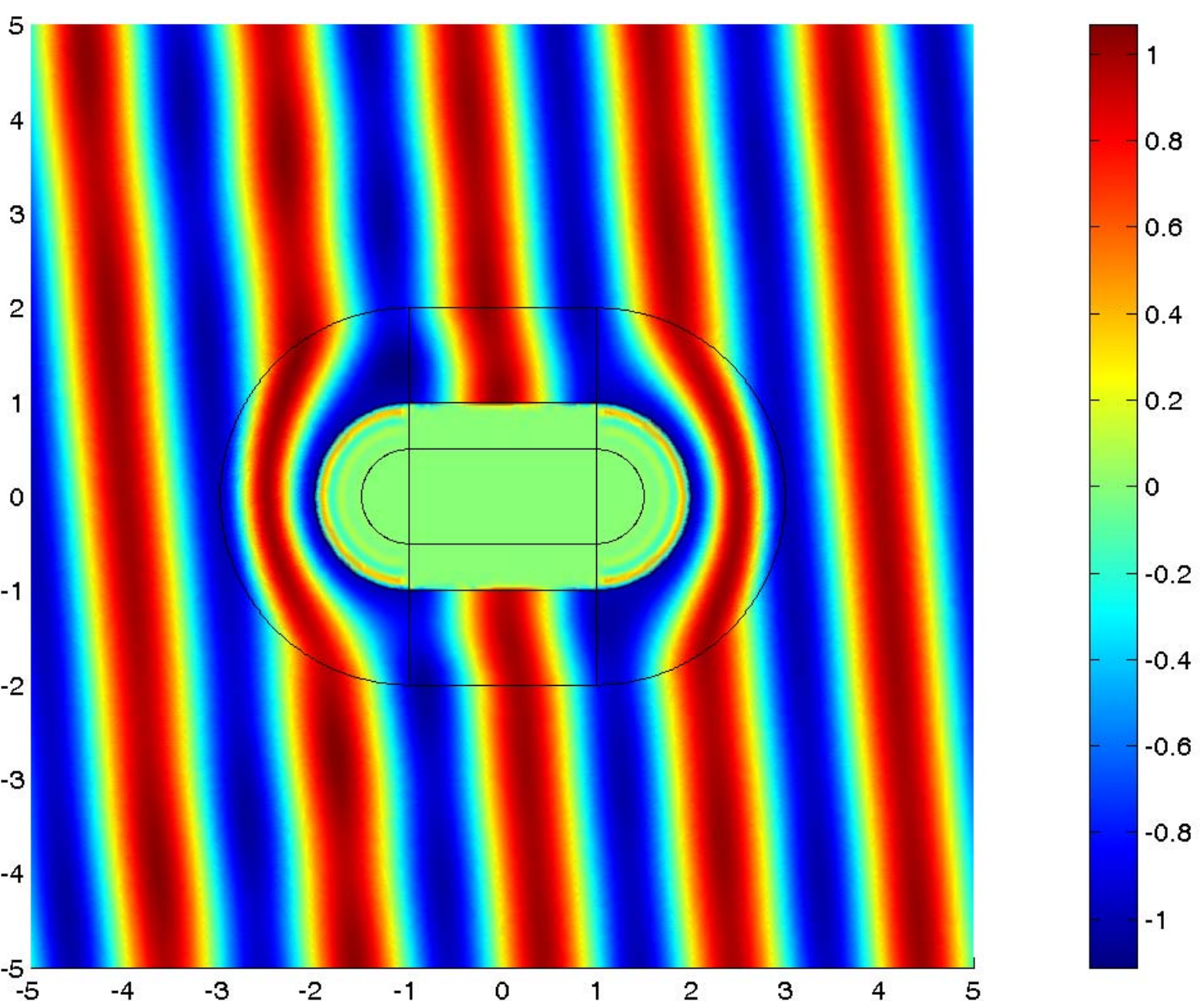}\hfill{}\includegraphics[width=0.45\textwidth]{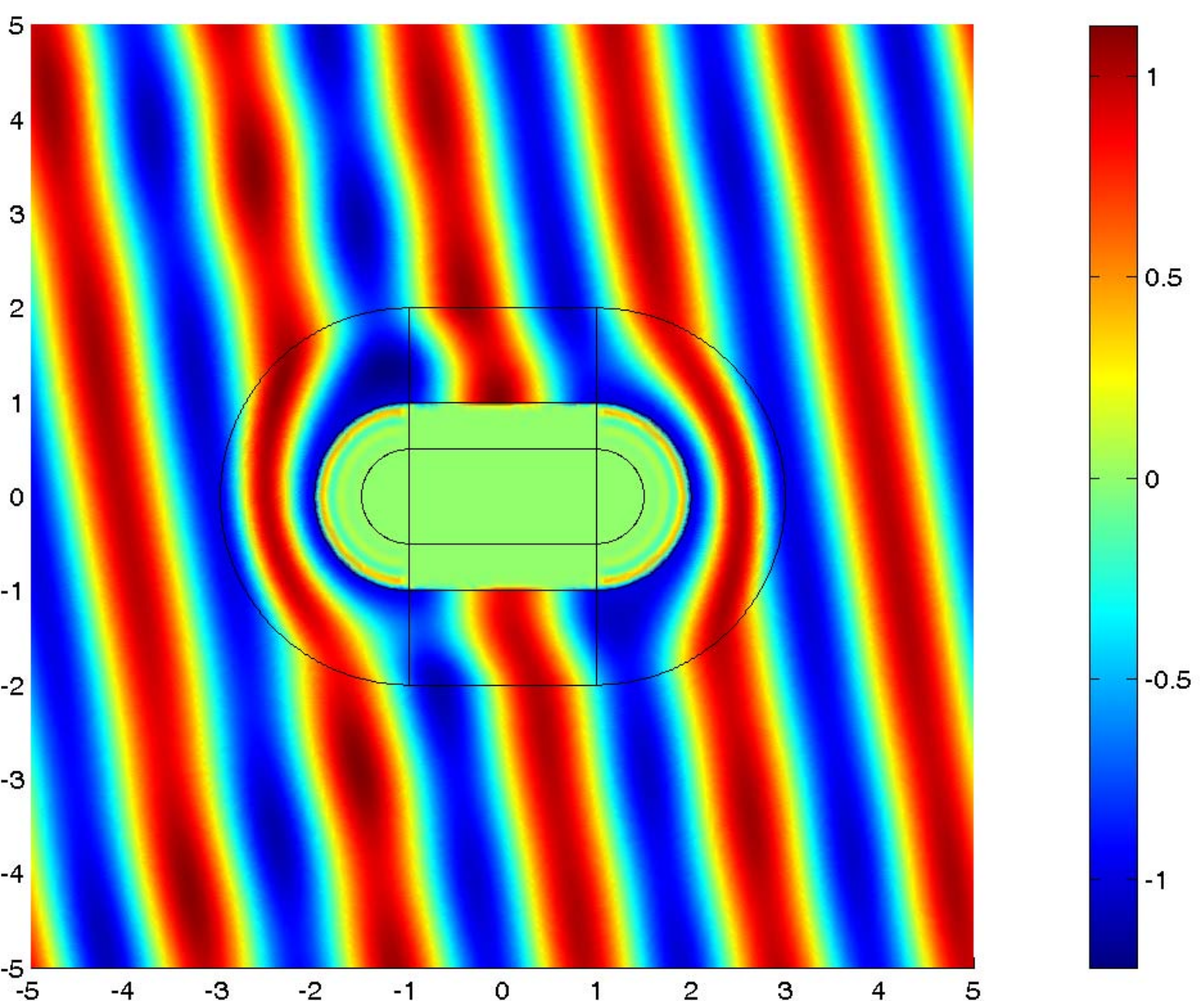}
\hfill{}

\hfill{}

\hfill{}\includegraphics[width=0.45\textwidth]{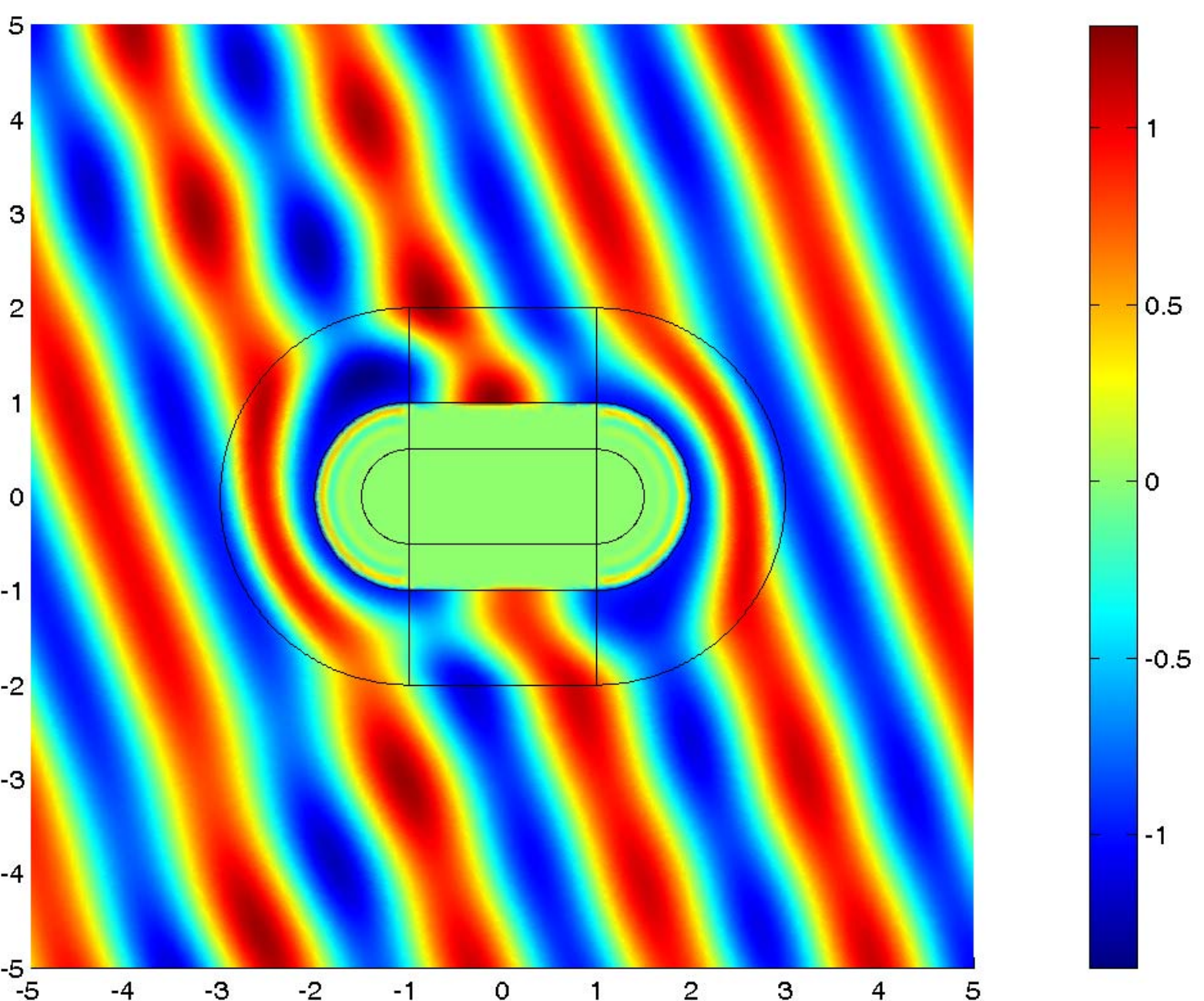}\hfill{}\includegraphics[width=0.45\textwidth]{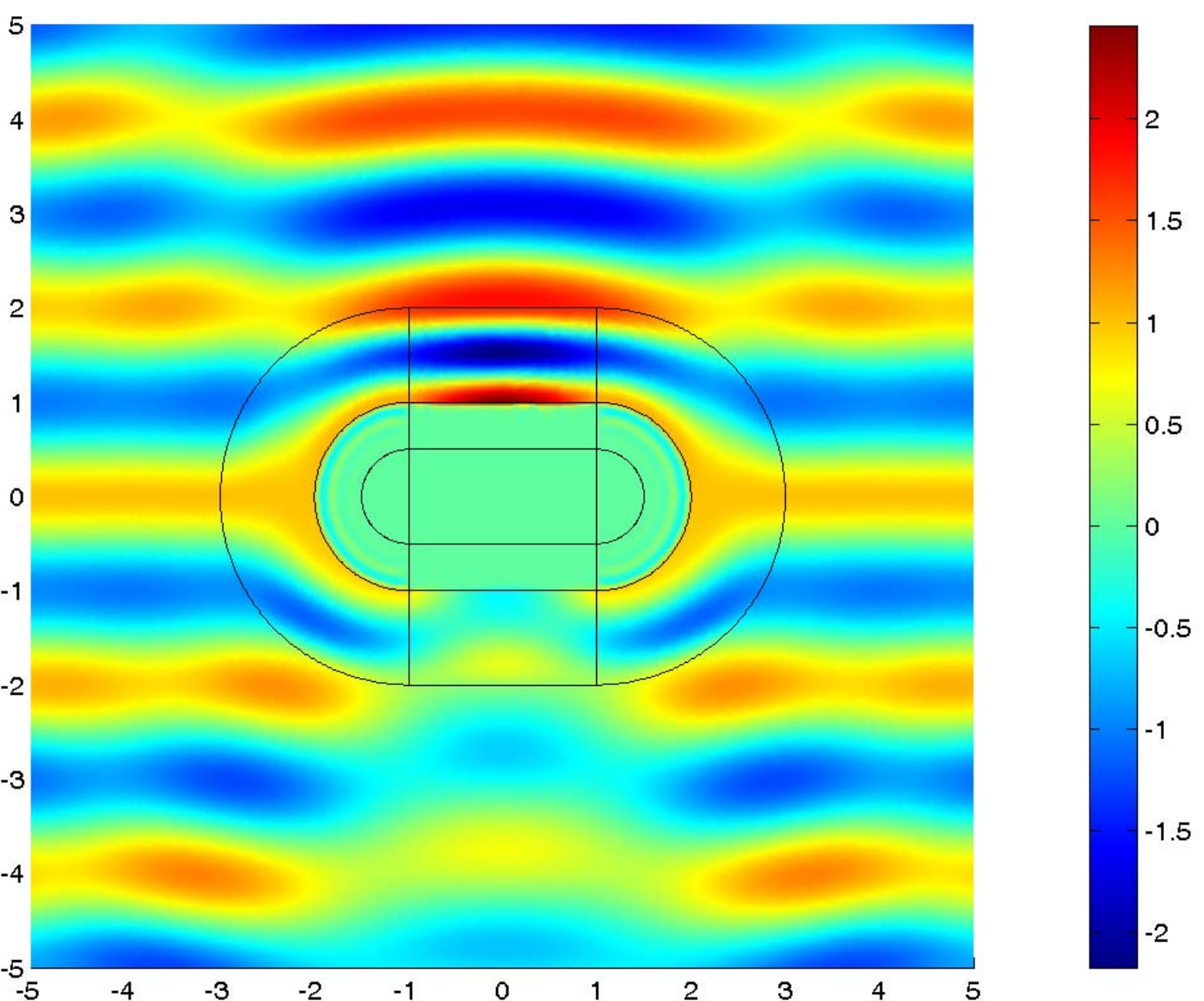}\hfill{}

\caption{\label{fig:2d:partial:cloak}(From left to right) Real part of
instantaneous
acoustic pressure distribution using two-dimensional
$l^{2}$ cloak with titled incident angles, 5, 10,  20 and 90
degrees, respectively. Parameters: $\varepsilon=10^{-2}$, $r_{1}=1$, $r_{2}=2$,
$a=1$.}
\end{figure}

\begin{figure}[htbp]
\hfill{}\includegraphics[width=0.8\textwidth]{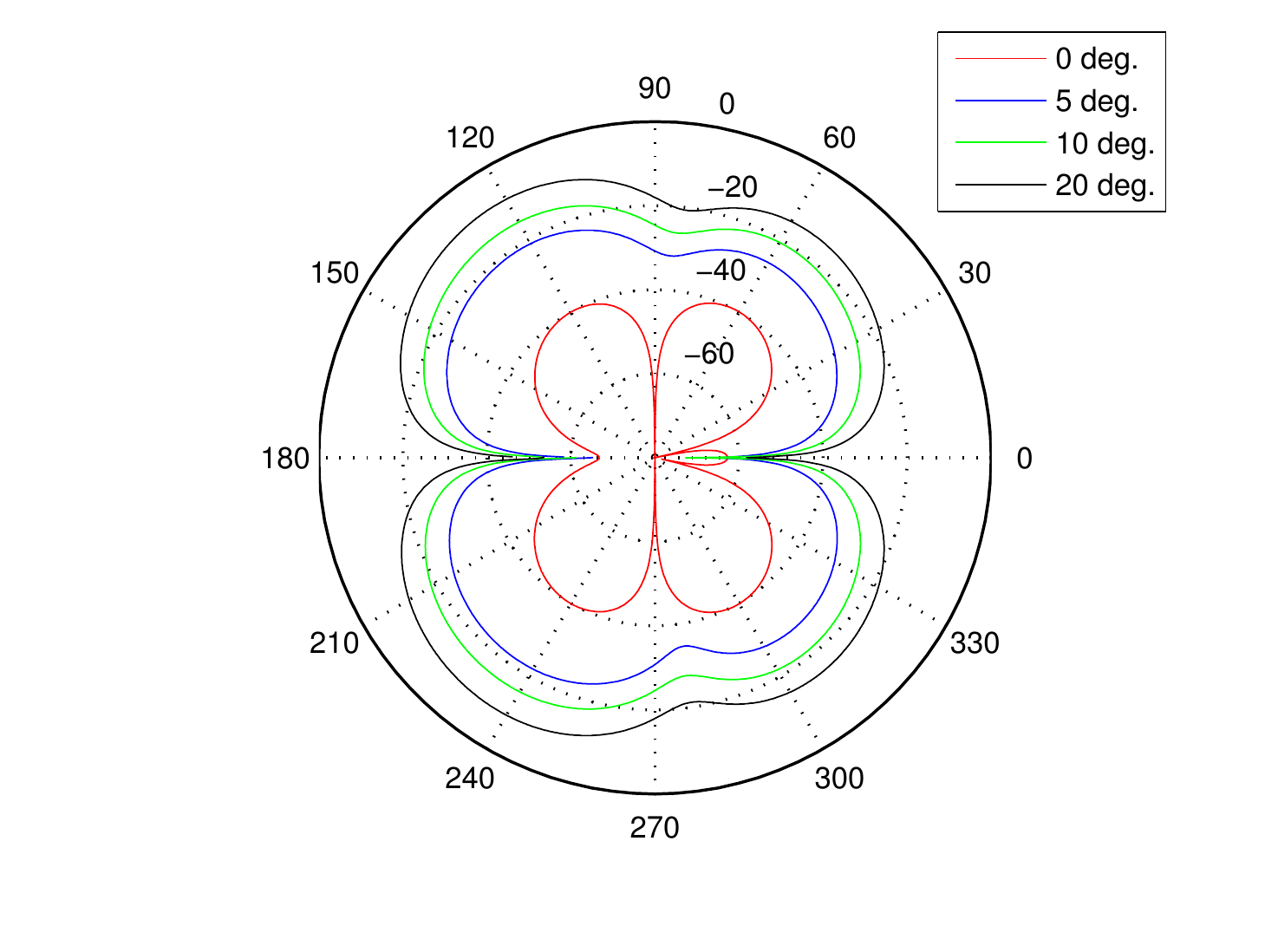}\hfill{}

\caption{\label{fig:2d:partial:cloak:far:field}Polar plot of far
field pattern as a function of polar angle using two-dimensional
$l^{2}$ cloak with tilted incident angles, 0, 5, 10 and  20 degrees,
respectively. }
\end{figure}

Finally, for the regularized full cloak case, we set $a=\varepsilon$ and
the incident angle even perpendicular to the axial direction of the
cloak and test the cloaking performance for the $l^2$ cloak of Type
$\mathbf{C}$. It can be observed from Figure~\ref{fig:2d:cloak:full}
that compared with the red reference line denoting the second order
decay rate, the far field pattern decays quadratically with
respect to $\varepsilon$. That is, the construction yields an approximate full cloak within $\varepsilon^2$-accuracy of
the ideal full cloak. This observation complies with Remark~\ref{rem:52}. Moreover, the second order rate of approximation complies with the estimate in \cite{LiuSun} for regularized full cloaks.
Similar observations are made for $l^{1}$ and $l^{\infty}$ cloaks of
Type $\mathbf{C}$.
\begin{figure}[htbp]
\hfill{}\includegraphics[width=0.6\textwidth]{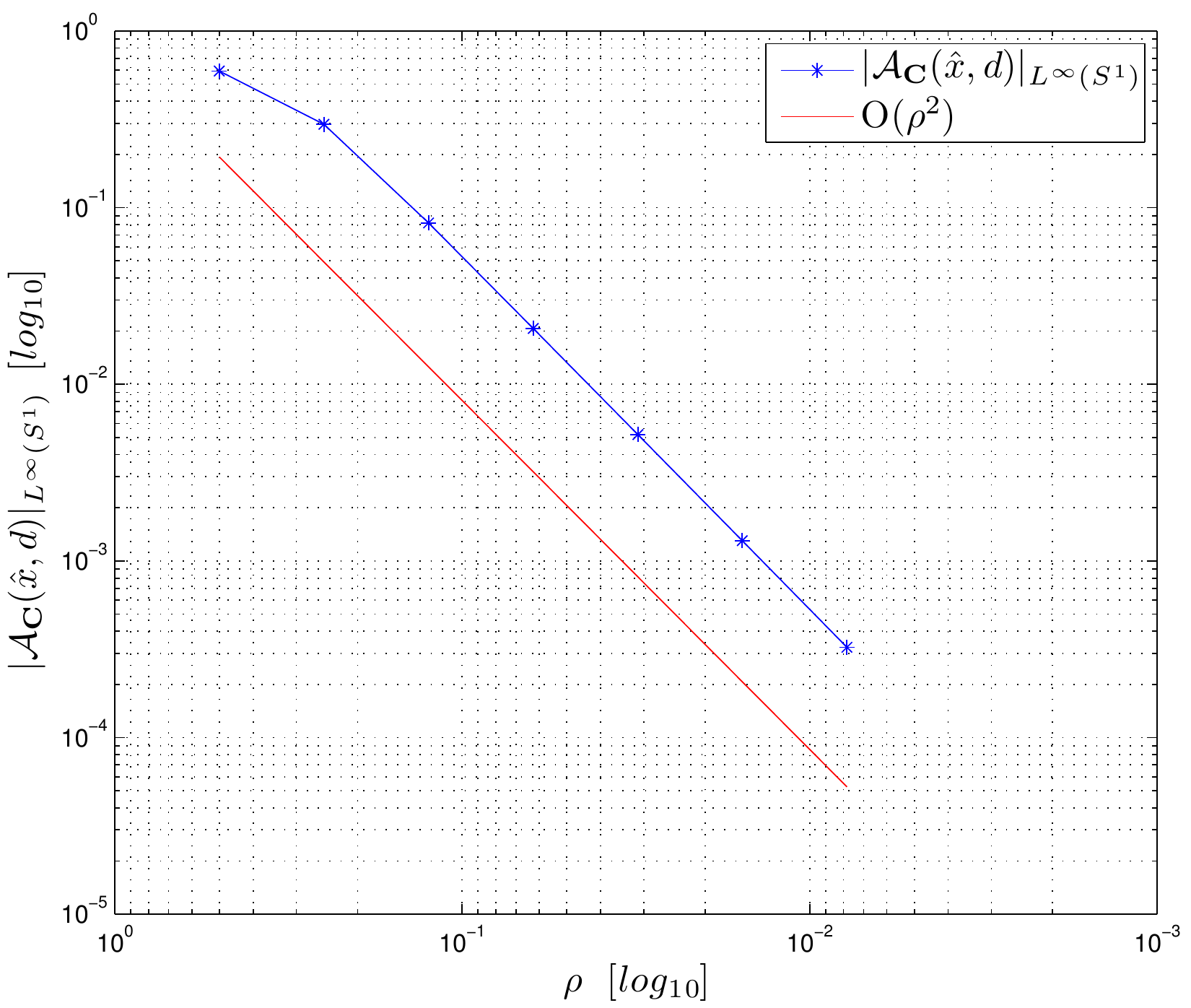}\hfill{}

\caption{\label{fig:2d:cloak:full}Convergence history of scattering
measurement data versus $\varepsilon$ for the 2D $l^{2}$ cloak when
$a=\varepsilon$.}
\end{figure}

\subsection{3D cloaks}

In the 3D cloaks, we fix the wavelength $l_e =3$ and the
regularization parameter $\varepsilon=10^{-2}$ in acoustic scattering and
use the spherical domain with radius $5$ centered at the origin as
the truncated domain for numerical simulations.

First, we check the cloaking performance of Type $\mathbf{D}$ cloaks
in terms of $a$. All the impinging direction of the incident wave is
perpendicular to the axial direction of the cloaks. It is noted that
we always adopt the 3D $l^2$ cloak and the right angle incidence for
tests. Similar observations are made for other combinations.

We show in Figure~\ref{fig:Type-D}(a) the plane impinging wave with
the incident direction along the $x$-axis. The resulting scattered
wave is depicted in Figure~\ref{fig:Type-D}(b), which nearly
vanishes outside the cloak. In Figure~\ref{fig:Type-D}(c), it can be
obviously seen that the total wave is bent and compressed within the
cloaking medium.

For a full 3D cloak with $a=\varepsilon$, Figure~\ref{fig:3d:cloak:full}
demonstrates the third order approximation of the far-field pattern with respect to
$\varepsilon$. That is, the construction yields an approximate full cloak within $\varepsilon^3$-accuracy of the ideal full cloak, which is consistent with the result in \cite{LiuSun} for regularized full cloaks obtained by blowing up a `uniformly' small region in the virtual space. As we increase $a$ to be $1$,  the far-field patter decays only
quadratically in terms of $\varepsilon$, see Figure~\ref{fig:3d:cloak:p2}. In other words, the construction by blowing up a `partially' small region in the virtual space will suffer a reduction in the rate of approximation. This observation complies with Remark~\ref{rem:53}.

\begin{figure}[htbp]
\hfill{}\includegraphics[width=0.48\textwidth]{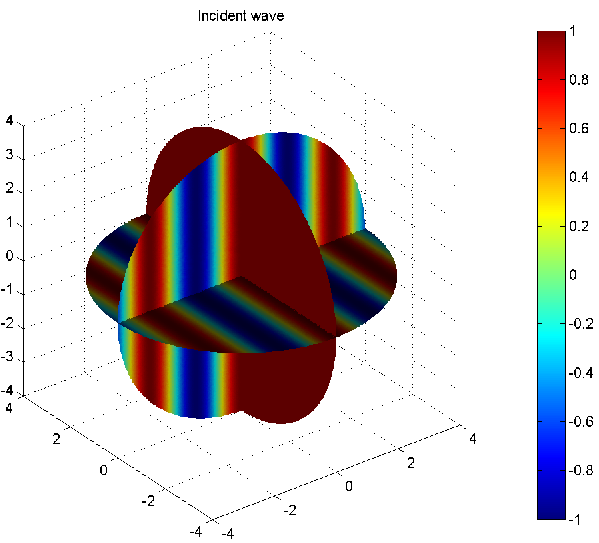}\hfill{}

\hfill{}(a)\hfill{}

\hfill{}\includegraphics[width=0.48\textwidth]{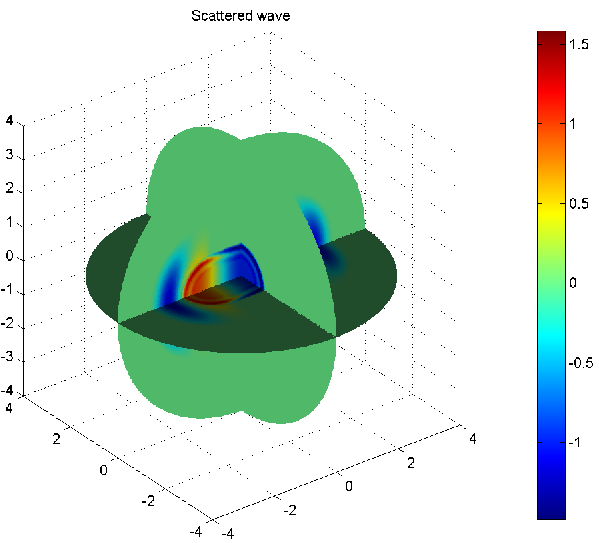}\hfill{}\includegraphics[width=0.48\textwidth]{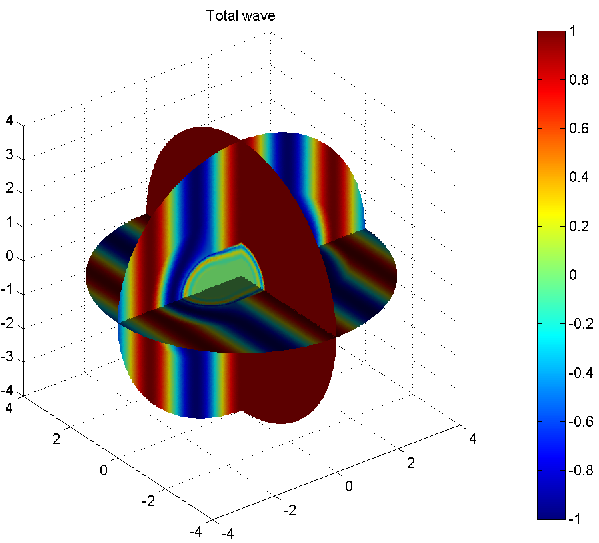}\hfill{}

\hfill{}~~~~~~~~~~(b)~~~~~~~~~~~~~~~~~~~~~~~~~~~~~~~~~~~~~\hfill{}(c)~~~~~~~~~~~~~~\hfill{}

\caption{\label{fig:Type-D}Type \textbf{D} when $\varepsilon=0.01$: Real
part of (a) the incident plane wave $\exp\{{i}kx\cdot d_0\}$,
(b) the scattered field and (c) the total field  sliced at $x=0$,
$z=0$ and $y=0$, respectively. .}
\end{figure}

\begin{figure}[htbp]
\hfill{}\includegraphics[width=0.6\textwidth]{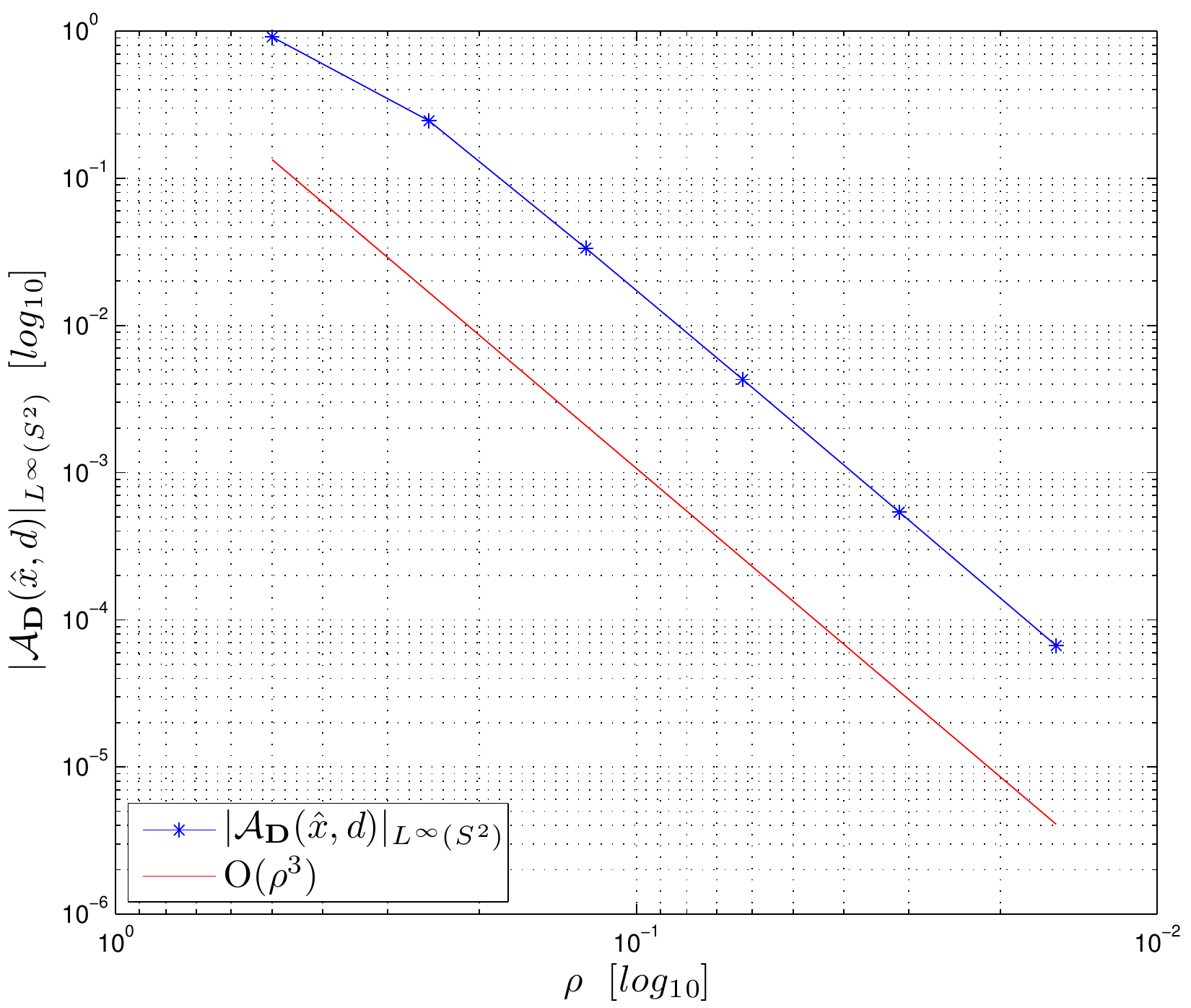}\hfill{}

\caption{\label{fig:3d:cloak:full}Convergence history of scattering
measurement data versus $\varepsilon$ for the 3D $l^{2}$ cloak $\textbf{D}$
when $a=\varepsilon$.}
\end{figure}


%
\begin{figure}[htbp]
\hfill{}\includegraphics[width=0.6\textwidth]{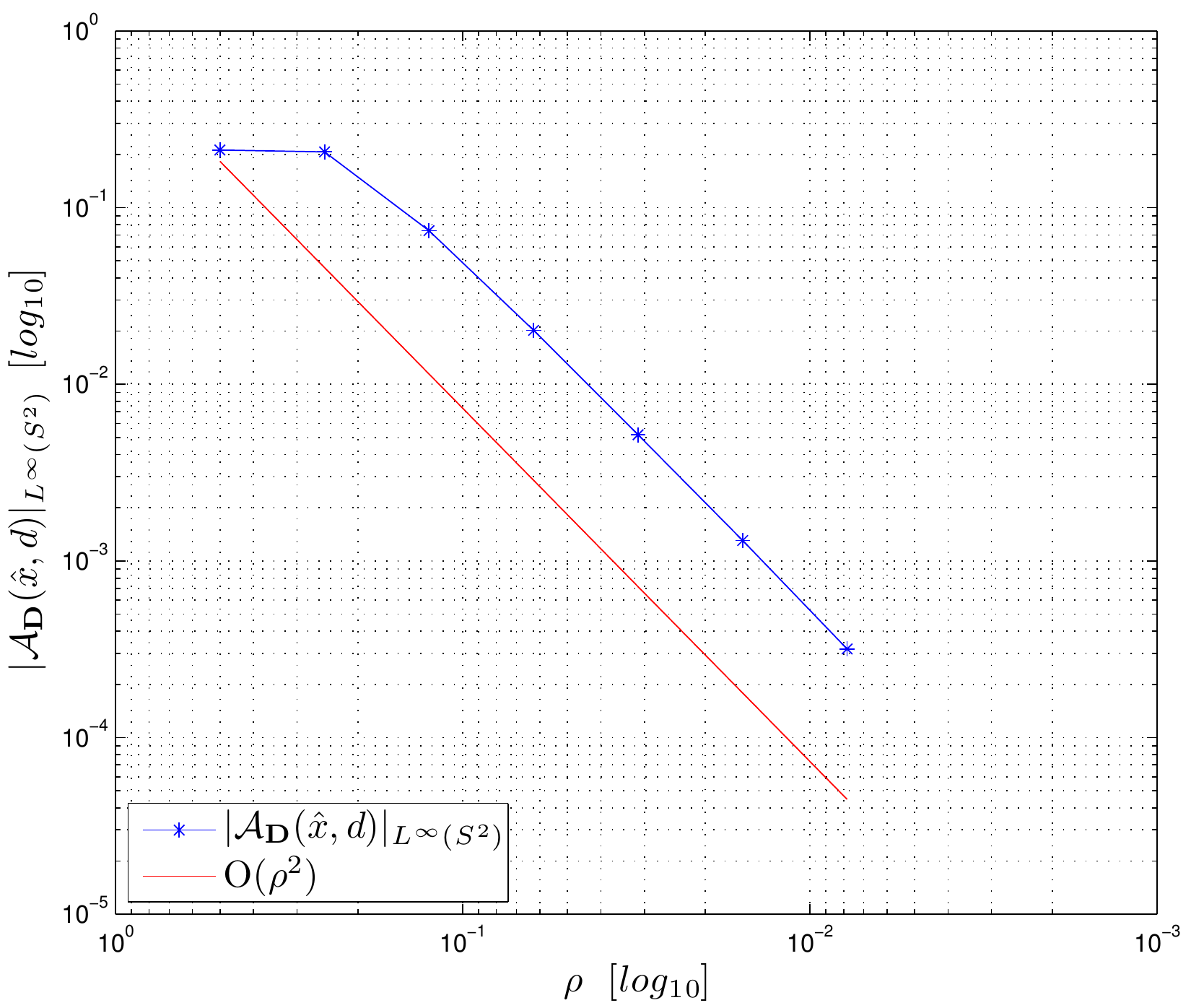}\hfill{}

\caption{\label{fig:3d:cloak:p2}Convergence history of scattering
measurement data versus $\varepsilon$ for 3D $l^{2}$ cloak $\textbf{D}$
when $a=1$ and $\varepsilon$ varies.}
\end{figure}

Finally, we investigate the 3D cloak of type $\textbf{E}$ proposed in Definition~\ref{def:51 e}. We choose fixed length $a=1$ and width $b=1$
but vary the regularization parameter $\varepsilon$. The incident plane
wave is impinging upon the $\textbf{E}$ cloak along the
$x$-direction which is parallel to the $xy$ plane where the type
$\textbf{E}$ cloak lies. Figure~\ref{fig:Type-E} shows us the
scattered and total pressure fields in sliced plots. Obviously the
scattered field is almost null outside the cloak, while the
total field behaves nearly perfectly through the cloak by bending
and compressing in the cloaking medium.


\begin{figure}[htbp]
\hfill{}

\hfill{}\includegraphics[width=0.48\textwidth]{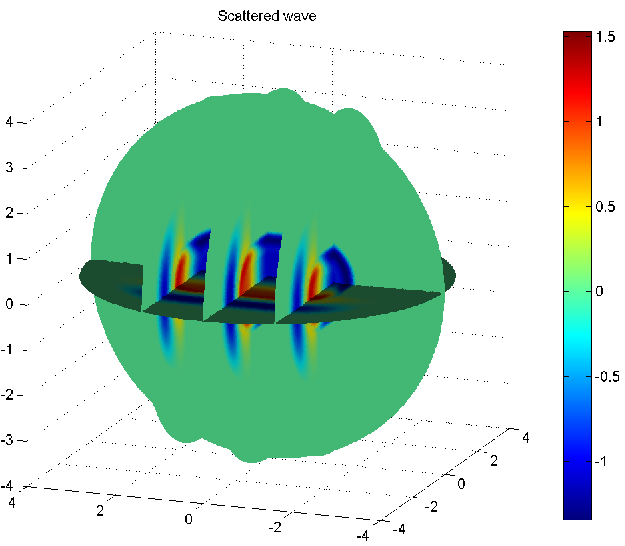}\hfill{}\includegraphics[width=0.48\textwidth]{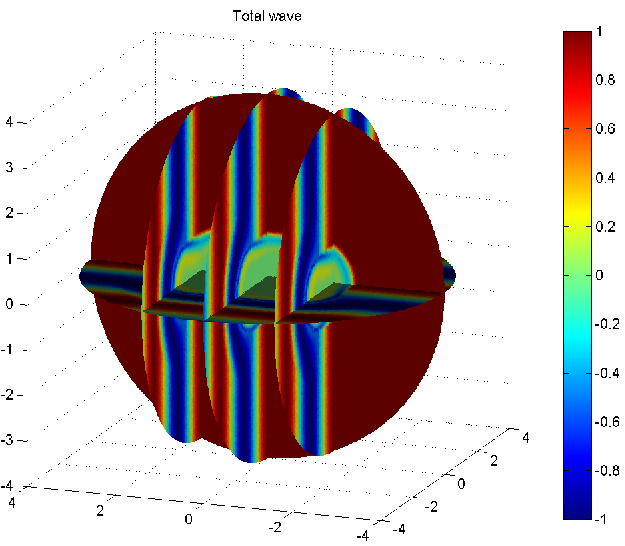}\hfill{}

\hfill{}~~~~~~~~~~(a)~~~~~~~~~~~~~~~~~~~~~~~~~~~~~~~~~~~~~\hfill{}(b)~~~~~~~~~~~~~~\hfill{}

\caption{\label{fig:Type-E}Type \textbf{E} when $\varepsilon=0.01$: Real part of (a)
the scattered field  and (b) the total field  sliced at $x=0$, $z=0$
and $y=-1.5$, $0$ and $1.5$, respectively.}
\end{figure}

\section*{Acknowledgement}

The work of Jingzhi Li is supported by the NSF of China. (No. 11201453 and 91130022). The work of Hongyu Liu is supported by NSF grant, DMS 1207784.
Luca Rondi is partly supported by Universit\`a
degli Studi di Trieste through Finanziamento per Ricercatori di Ateneo 2009 and by GNAMPA, INdAM, through 2012 projects. The work of Gunther Uhlmann is partly
supported by NSF and the Fondation des Sciences Math\'ematiques de Paris.

%


\begin{thebibliography}{99}

\bibitem{Ada}
{R.~A.~Adams},
{\it Sobolev Spaces},
Academic Press, New York, 1975.

%
%
\bibitem{AE} {A.~Alu and N.~Engheta}, {\it Achieving transparency with plasmonic and
metamaterial coatings}, Phys. Rev. E, {\bf 72 }(2005), 016623.

\bibitem{Ammari0} {H.~Ammari, G.~Ciraolo, H.~Kang, H.~Lee and G.~Milton}, {\it Spectral analysis of a Neumann-Poincar\'e operator and analysis of cloaking due to anomalous localized resonance}, Arch. Ration. Mech. Anal., at press.

\bibitem{Ammari3}  {H.~Ammari, J.~Garnier, V.~Jugnon, H.~Kang, M.~Lim and H.~Lee},
{\it Enhancement of near-cloaking. Part III: Numerical simulations,
statistical stability, and related questions}, Contemporary
Mathematics, \textbf{577} (2012), 1--24.


\bibitem{Ammari1} {H.~Ammari, H.~Kang, H.~Lee and M.~Lim}, {\it Enhancement of near-cloaking using generalized polarization tensors vanishing structures. Part I: The conductivity problem}, Comm. Math. Phys., {\bf 317} (2013), 253--266.

\bibitem{Ammari2} {H.~Ammari, H.~Kang, H.~Lee and M.~Lim}, {\it Enhancement of near-cloaking. Part II: The Helmholtz equation}, Comm. Math. Phys., {\bf 317} (2013), 485--502.

%
%
%

\bibitem{CakCol} {F.~Cakoni and D.~Colton}, {\it Qualitative Methods in Inverse Scattering Theory}, Springer-Verlag,
Berlin Heidelberg, 2006.


%





\bibitem{CC} {H.~Chen and C.~T.~Chan}, {\it Acoustic cloaking and transformation acoustics},
J. Phys. D: Appl. Phys., \textbf{43} (2010), 113001.





\bibitem{ColKre} {D.~Colton and R.~Kress}, {\it Inverse
Acoustic and Electromagnetic Scattering Theory}, 2nd Edition,
Springer-Verlag, Berlin, 1998.

%










\bibitem{GKLUoe} {A.~Greenleaf, Y.~Kurylev, M.~Lassas and G.~Uhlmann}, {\it Improvement of cylindrical cloaking with SHS lining},
Optics Express, {\bf 15} (2007), 12717--12734.

\bibitem{GKLU3} {A.~Greenleaf, Y.~Kurylev, M.~Lassas and G.~Uhlmann}, {\it Full-wave invisibility of active devices at all
frequencies}, Comm. Math. Phys., {\bf 279} (2007), 749--789.

\bibitem{GKLU_2} {A.~Greenleaf, Y.~Kurylev, M.~Lassas and G.~Uhlmann},
{\it Isotropic transformation optics: approximate acoustic and
quantum cloaking}, New J. Phys., {\bf 10} (2008), 115024.

\bibitem{GKLU2} {A.~Greenleaf, Y.~Kurylev, M.~Lassas, and G.~Uhlmann}, {\it Electromagnetic wormholes via handlebody constructions}, Comm.
Math. Phys., {\bf 281} (2008), 369--385.

\bibitem{GKLU4} {A.~Greenleaf, Y.~Kurylev, M.~Lassas and G.~Uhlmann}, {\it Invisibility and inverse prolems}, Bulletin A. M. S., {\bf
46} (2009), 55--97.

\bibitem{GKLU5} {A.~Greenleaf, Y.~Kurylev, M.~Lassas and G.~Uhlmann}, {\it Cloaking devices, electromagnetic wormholes and
transformation optics}, SIAM Review, {\bf 51} (2009), 3--33.

\bibitem{GLU} {A.~Greenleaf, M.~Lassas and G.~Uhlmann},
{\it Anisotropic conductivities that cannot be detected by EIT}, Physiolog.
Meas, (special issue on Impedance Tomography), {\bf 24} (2003), 413.

\bibitem{GLU2} {A.~Greenleaf, M.~Lassas and G.~Uhlmann},
{\it On nonuniqueness for Calder\'on's inverse problem}, Math. Res.
Lett., {\bf 10} (2003), 685--693.

%










\bibitem{Isa} {V.~Isakov}, {\it Inverse Problems for Partial
Differential Equations}, 2nd Edition, Springer-Verlag, New York, 2006.

\bibitem{KocLiuSunUhl} {I.~Kocyigit, H.~Y.~Liu and H.~Sun}, {\it Regular scattering patterns from near-cloaking devices and their implications for invisibility cloaking}, preprint, 2012.

\bibitem{KOVW} {R.~Kohn, O.~Onofrei, M.~Vogelius and M.~Weinstein}, {\it Cloaking via change of variables for the Helmholtz
equation}, Comm. Pure Appl. Math., {\bf 63} (2010), 973--1016.


\bibitem{KSVW} {R.~Kohn, H.~Shen, M.~Vogelius and M.~Weinstein}, {\it Cloaking via change of variables in electrical impedance
tomography}, Inverse Problems, {\bf 24} (2008), 015016.

%
%

\bibitem{Leo} {U.~Leonhardt}, {\it Optical conformal mapping},
Science, {\bf 312} (2006), 1777--1780.



\bibitem{LiLiuSun} {J.~Li, H.~Y.~Liu and H.~Sun}, {\it Enhanced approximate cloaking by SH and FSH lining}, Inverse Problems, {\bf 28} (2012), 075011.

\bibitem{LiP} {J.~Li and J.~B.~Pendry}, {\it Hiding under the carpet: a new strategy for cloaking}, Phys. Rev. Lett., {\bf 101} (2008), 203901.

\bibitem{Lio} {J.L.~Lions and E.~Magenes},  {\it Non-Homogeneous Boundary Value Prob\-lems and Applications I}, Springer-Verlag, 1970.


\bibitem{Liu} {H.~Y.~Liu}, {\it Virtual reshaping and
invisibility in obstacle scattering}, Inverse Problems, {\bf 25}
(2009), 045006.




\bibitem{LSSZ} {H.~Y.~Liu, Z.~Shang, H.~Sun and J.~Zou}, {\it Singular perturbation of reduced wave equation and scattering from an embedded obstacle}, J. Dyn. Diff. Eq., {\bf 24} (2012), 803--821.

\bibitem{LiuSun} {H.~Y.~Liu and H.~Sun}, {\it Enhanced near-cloak by FSH lining},  J. Math. Pures Appl., {\bf 99} (2013), 17--42.


\bibitem{LZ1} {H.~Y.~Liu and T.~Zhou}, {\it Two dimensional
invisibility cloaking by transformation optics}, Discrete Contin. Dyn. Syst., {\bf 31} (2011), 525--543.

%

\bibitem{Men-Ron}
{G.~Menegatti and L.~Rondi},
{\it Stability for the acoustic scattering problem for sound-hard
scatterers},
preprint, 2013.




\bibitem{MN} {G.~W.~Milton and N.-A.~P.~Nicorovici}, {\it On the
cloaking effects associated with anomalous localized resonance},
Proc. Roy. Soc. Lond. A, {\bf 462} (2006), 3027--3095.



\bibitem{Ned}
{J.~C.~N\'ed\'elec}, {\it Acoustic and Electromagnetic Equations: Integral Representations for Harmonic Problems},
Springer-Verlag, New York, 2001.

\bibitem{Nor} {A.~N.~Norris}, {\it Acoustic cloaking theory}, Proc. R. Soc. Lond. A, {\bf 464} (2008),
2411--2434.

\bibitem{N1} {H.~Nguyen}, {\it Cloaking via change of variables for the Helmholtz equation in the whole space}, Comm. Pure Appl. Math., {\bf 63} (2010), 1505--1524.

\bibitem{N2} {H.~Nguyen and M.~S.~Vogelius}, {\it Full range scattering estimates and their application to cloaking}, Arch. Ration. Mech. Anal., {\bf 203} (2012), 769--807.



%

\bibitem{PenSchSmi} {J.~B.~Pendry, D.~Schurig and D.~R.~Smith}, {\it Controlling electromagnetic fields}, Science, {\bf
312} (2006), 1780--1782.

\bibitem{Ron03}
{L.~Rondi},
{\it Unique determination of non-smooth sound-soft scatterers by finite\-ly many
far-field measurements},
Indiana Univ. Math. J., \textbf{52} (2003), 1631--1662.

\bibitem{RYNQ} {Z.~Ruan, M.~Yan, C.~W.~Neff and M.~Qiu},
{\it Ideal cylyindrical cloak: Perfect but sensitive to tiny
 perturbations}, Phy. Rev. Lett., {\bf 99} (2007), 113903.


%
%
\bibitem{U2} {G.~Uhlmann}, {\it Visibility and invisibility}, ICIAM 07--6th International Congress on Industrial and Applied Mathematics, Eur. Math. Soc., Z\"urich, pp.~381--408, 2009.
%
%







\bibitem{YYQ}
{M.~Yan, W.~Yan and M.~Qiu}, {\it Invisibility cloaking by
coordinate transformation}, Chapter 4 of {\it Progress in
Optics}--Vol. 52, Elsevier, pp.~261--304, 2008.



\end{thebibliography}
\end{document}